\documentclass{amsart}

\author{Owen Gwilliam}
\address{Max-Planck-Institut f\"ur Mathematik, Bonn, Germany}
\email{gwilliam@mpim-bonn.mpg.de}
\urladdr{http://people.mpim-bonn.mpg.de/gwilliam/}

\usepackage[subsec,skipfonts]{enrabbrev}
\usepackage{mathrsfs}

\title[Linear Batalin-Vilkovisky Quantization as a Functor of $\infty$-Categories]{Linear Batalin-Vilkovisky Quantization\\ as a Functor of $\infty$-Categories}

\newcommand{\bC}{\mathbf{C}}
\newcommand{\bCc}{\bC^{\txt{c}}}
\newcommand{\bCf}{\bC^{\txt{f}}}
\newcommand{\bCcf}{\bC^{\txt{cf}}}
\newcommand{\Lie}{\txt{Lie}}

\newcommand{\bCs}{\bC_{\Delta}}
\newcommand{\bCcs}{\bCc_{\Delta}}
\newcommand{\bCfs}{\bCf_{\Delta}}
\newcommand{\bCcfs}{\bCcf_{\Delta}}
\newcommand{\obCs}{\bar{\bC}_{\Delta}}
\newcommand{\obCcs}{\bar{\bC}^{\txt{c}}_{\Delta}}
\newcommand{\obCfs}{\bar{\bC}^{\txt{f}}_{\Delta}}
\newcommand{\obCcfs}{\bar{\bC}^{\txt{cf}}_{\Delta}}
\newcommand{\NobCs}{\mathrm{N}\obCs}
\newcommand{\NobCcs}{\mathrm{N}\obCcs}
\newcommand{\NobCfs}{\mathrm{N}\obCfs}
\newcommand{\NobCcfs}{\mathrm{N}\obCcfs}

\newcommand{\NbCcfs}{\mathrm{N}\bCcfs}

\newcommand{\angled}[1]{\langle #1 \rangle}

\newcommand{\QQuad}{\mathcal{Q}\txt{uad}}
\newcommand{\MMod}{\mathcal{M}\txt{od}}
\newcommand{\LLie}{\mathcal{L}\txt{ie}}
\newcommand{\LLI}{\mathcal{L}_{\infty}}
\newcommand{\AAlg}{\mathcal{A}\txt{lg}}

\newcommand{\SSymp}{\mathcal{S}\txt{ymp}}
\newcommand{\QQCoh}{\mathcal{Q}\txt{Coh}}
\newcommand{\PPic}{\mathcal{P}\txt{ic}}
\newcommand{\Moduli}{\mathcal{M}\txt{oduli}}
\newcommand{\sQuad}{\txt{\textbf{Quad}}}
\newcommand{\sMod}{\txt{\textbf{Mod}}}
\newcommand{\sHom}{\txt{\textbf{Hom}}}
\newcommand{\sLie}{\txt{\textbf{Lie}}}
\newcommand{\sLI}{\mathbf{L}_{\infty}}
\newcommand{\LI}{\mathrm{L}_{\infty}}
\newcommand{\sAlg}{\txt{\textbf{Alg}}}

\newcommand{\Quad}{\txt{Quad}}
\newcommand{\BD}{\txt{BD}}
\newcommand{\Comm}{\txt{Comm}}
\newcommand{\sComm}{\txt{\textbf{Comm}}}
\newcommand{\CComm}{\mathcal{C}\txt{omm}}
\newcommand{\dSt}{\txt{dSt}}
\newcommand{\QCoh}{\txt{QCoh}}
\newcommand{\OOpd}{\mathcal{O}\txt{pd}}

\newcommand{\eg}{e.g.\@}
\newcommand{\ie}{i.e.\@}
\newcommand{\cf}{cf.\@}

\newcommand{\CAT}{\txt{CAT}}
\newcommand{\LCAT}{\widehat{\CAT}}

\def\PZ{{{\rm P}_0}}
\def\teo{{\widetilde{\mathrm E}_0}}
\def\g{\mathfrak g}
\def\Sym{{\rm Sym}}
\def\d{{\rm d}}
\def\dq{{\rm dequant}}

\def\ev{\mathrm{ev}}
\def\CL{\mathrm{C}^{\mathrm{L}}}
\def\c{\mathbf{c}}
\def\H{\mathcal{H}}

\def\quant{\mathcal{Q}}
\def\Spec{{\rm Spec}}

\def\Heis{\txt{Heis}}
\def\HH{\mathcal{H}}
\def\htimes{\underset{\mathrm{H}}{\otimes}}
\def\bvq{\mathcal{BVQ}}
\def\VB{\mathcal{VB}}
\def\GVB{\mathcal{GVB}}
\def\cq{\mathcal{CQ}}

\def\div{{\txt{div}}}
\def\rE{{\txt{E}}}

\newcommand{\CNCoalg}{\txt{Cocomm}^{\txt{conil}}}
\newcommand{\sCNCoalg}{\mathbf{Cocomm}^{\txt{conil}}}

\usepackage{eucal}

\begin{document}
\begin{abstract}
  We study linear Batalin-Vilkovisky (BV) quantization, which is a derived and shifted
  version of the Weyl quantization of symplectic vector spaces. Using
  a variety of homotopical machinery, we implement this construction 
  as a symmetric monoidal functor of \icats{}. We also show that this construction
  has a number of pleasant properties: It has a natural extension to
  derived algebraic geometry, it can be fed into the higher Morita
  category of $\mathrm{E}_{n}$-algebras to produce a ``higher BV
  quantization'' functor, and when restricted to formal moduli
  problems, it behaves like a determinant. Along the way we also use
  our machinery to give an algebraic construction of
  $\mathrm{E}_{n}$-enveloping algebras for shifted Lie algebras.
\end{abstract}

\maketitle
\tableofcontents

\section{Introduction}

A well-known adage in mathematical physics is that ``quantization is not a functor,'' 
but with suitable restrictions, there are situations where quantization is functorial.  
Our goal here is to articulate the simplest piece of the Batalin-Vilkovisky formalism 
as a functor, using modern machinery of higher categories and derived geometry.

The most fundamental case of quantization assigns to the vector space
$\mathbb{R}^{2n}$ the \emph{Weyl algebra}, which is the associative
algebra on $2n$ generators $p_{1},\ldots,p_{n},q_{1},\ldots,q_{n}$
with relations $[p_{i},p_{j}] = 0 = [q_{i},q_{j}]$ and
$[p_{i},q_{j}]=\delta_{ij}$. Here $\mathbb{R}^{2n}$ should be thought
of as the cotangent bundle of $\mathbb{R}^{n}$, equipped with its
standard symplectic structure, which is the arena for classical mechanics on
$\mathbb{R}^{n}$. This assignment can formulated as a functor,
known as \emph{Weyl quantization}, from symplectic vector spaces (or
more generally, vector spaces with a skew-symmetric pairing) 
to assocative algebras. For us this is the model case of functorial quantization. 
This construction naturally breaks up into three steps:
\begin{enumerate}[(1)]
\item To a vector space $V$ with skew-symmetric pairing $\omega$ we
  associate its \emph{Heisenberg Lie algebra} $\Heis(V,\omega)$,
  which is the direct sum $V \oplus \mathbb{R}\hbar$ 
  equipped with the Lie bracket where
  \[
  [x,y] = \omega(x,y)\hbar
  \]
  for $x, y$ in $V$, and all other brackets are zero.
\item To the Lie algebra $\Heis(V,\omega)$, we assign its
  universal enveloping algebra $U\Heis(V, \omega)$.
\item If we now set $\hbar$ to 1, we get the \emph{Weyl algebra}:
  $\txt{Weyl}(V,\omega) := U\Heis(V,\omega)/(\hbar = 1)$.
\end{enumerate}
On the other hand, if we set $\hbar = 0$ we get $\Sym(V)$, equipped
with the Poisson bracket 
\[\{x,y\} = \lim_{\hbar \to 0} [x,y]/\hbar = \omega(x,y),\] for $x,y
\in V$, which is an algebraic version of the Poisson algebra of
classical observables.  The universal enveloping algebra
$U\Heis(V,\omega)$ can thus be viewed as a deformation quantization of
the Poisson algebra $\Sym(V)$.
This procedure is at the core of all approaches to ``free theories,''
and hence the base case for the more challenging and more interesting
interacting theories.

Our main object of study in this paper is a \emph{derived} and
\emph{shifted} version of this construction --- \emph{derived} in the
sense that we replace vector spaces by cochain complexes (or more
generally modules over a commutative differential graded algebra) and
\emph{shifted} in the sense that we consider skew-symmetric pairings
of degree 1. We will construct a functorial quantization of these
objects to \emph{$\mathrm{E}_{0}$-algebras} (which are just pointed
cochain complexes), using shifted versions of the Heisenberg Lie
algebra construction and of the universal enveloping algebra. (In
general, we expect that there is a functorial quantization of cochain
complexes with $(1-n)$-shifted skew-symmetric pairings to
$\mathrm{E}_{n}$-algebras, and we discuss below how we believe this
arises naturally from the case $n = 0$ that we consider.)

Our construction produces the simplest possible examples of
\emph{Batalin-Vilkovisky quantization}. This homological approach
to quantization of field theories was introduced by Batalin and
Vilkovisky \cite{BV1,BV2,BV3} as a generalization of the BRST formalism, 
in an effort to deal with complicated field theories such as supergravity. 
Their formalism for field theory, both classical and quantum, 
has broad application and conceptual depth.  
(For recent work on these issues, see \cite{Cat,CMR1,CMR2,Cos,CG}.) 
For brevity, we will talk about the \emph{BV formalism} and 
\emph{BV quantization}.  (We should point out that we mean here the Lagrangian
formulation, whose quantum aspect is focused on a homological approach
to the path integral.)

In this introduction, we begin by describing our main results in
\S\ref{subsec:mainresults}. Next we describe
some consequences in the setting of derived geometry in
\S\ref{subsec:DAG} and then discuss an extension to ``higher BV
quantization'' in \S\ref{subsec:aksz}, where we also sketch how
we expect our results to relate to the simplest examples of AKSZ
theories.

Afterwards, in \S\ref{div}, we discuss BV quantization from the
perspective of physics --- notably, how it is a homological version of
integration --- and explain how the standard approach relates to our
work here.  (A reader coming from a field theory setting might
prefer to read that discussion before the preceding sections; on the
other hand, readers without such background should feel free to skip
it, as nothing in the rest of the paper depends on it.)

\subsection{Our Main Results}\label{subsec:mainresults}

As in the case of Weyl quantization, our construction naturally breaks
up into three steps: first we apply an analogue of the Heisenberg Lie
algebra construction, then an enveloping algebra functor, and finally
we ``set $\hbar = 1$.'' These constructions are certainly well-known
among those who work with BV quantization, although rarely articulated
in this way, and they can be found, in a slightly different form, in
\cite{BD} and later in \cite{CG}.

In the first step we start with a cochain complex $V$ over the base field
$k$, equipped with a 1-shifted symmetric pairing $\omega \colon \Sym^{2}V \to
k[1]$ --- we call such objects \emph{1-shifted quadratic modules}. 
We then define a 1-shifted Heisenberg Lie algebra
$\txt{Heis}_{1}(V, \omega)$ by equipping the cochain complex 
$V \oplus k\mathbf{c}$, where
$\mathbf{c}$ is an added central element, with the bracket $[v, w] =
\omega(v,w)\mathbf{c}$ for $v,w \in V$. 
Unfortunately, this simple construction is not homotopically meaningful, 
which requires us to do a bit of work. In \S\ref{sec:Heis}, we show:
\begin{thm}
  For $A$ a commutative differential graded algebra over $k$, let
  $\QQuad_{1}(A)$ denote the \icat{} of 1-shifted quadratic
  $A$-modules and $\LLie_{1}(A)$ the \icat{} of 1-shifted Lie algebras
  over $A$. Then there is a lax symmetric monoidal functor of \icats{}
  \[ \mathcal{H} \colon \QQuad_{1}(A)^\oplus \to \LLie_{1}(A)^\oplus \] that takes
  $(V,\omega) \in \QQuad_{1}(A)$ to a cofibrant replacement of
  $\txt{Heis}_{1}(V, \omega)$. The construction is natural in $A$.
  
  Moreover, letting $\MMod_{A\mathbf{c}}(\LLie_{1}(A))$ denote the \icat{}
  of modules in $\LLie_1(A)^\oplus$ over the abelian 1-shifted Lie algebra $A\c$, 
  the induced functor
  \[ \widetilde{\mathcal{H}} \colon \QQuad_{1}(A)^\oplus \to
  \MMod_{A\mathbf{c}}(\LLie_{1}(A))^{\sqcup_{A\c}} \]
  is symmetric monoidal.
\end{thm}

In the second step we apply an enveloping algebra functor. However,
this no longer produces an associative algebra, but rather a \emph{BD-algebra} 
in the following sense:
\begin{defn}\label{def bd}
  A \emph{Beilinson-Drinfeld (BD) algebra} is a differential graded
  module $(M,\d)$ over $k[\hbar]$ equipped with an $\hbar$-linear
  unital graded-commutative product of degree zero and an $\hbar$-linear
  shifted Poisson bracket bracket of degree one such that
  \[
  \d(\alpha \beta) = \d(\alpha)\beta + (-1)^{\alpha} \alpha\,\d(\beta) + \hbar\{\alpha,\beta\}
  \]
  for any $\alpha, \beta$ in $M$.
\end{defn}

BD-algebras can be encoded as algebras for an operad, and since we are
working over a base field of characteristic zero, there are
well-behaved model categories of such operad algebras.  Using this
machinery we explicitly describe the BD-enveloping algebra of a 1-shifted Lie
algebra and show that it gives a symmetric monoidal
functor of \icats{} from $\LLie_{1}(A)$ to the \icat{}
$\AAlg_{\txt{BD}}(A[\hbar])$ of $\txt{BD}$-algebras in differential
graded modules over $A[\hbar]$. The second claim is not entirely obvious, 
since the model categories in question are not compatible with the tensor
products in the usual sense.

In the Weyl quantization story, we produced an associative algebra
with a parameter $\hbar$. Setting $\hbar = 0$ this algebra reduced to a
Poisson algebra, and setting $\hbar = 1$ it was the Weyl algebra. Hence
the parameter $\hbar$ explicitly describes a deformation quantization. 
Interpreting the symplectic vector space as the \emph{phase space} 
of a classical system, the Weyl quantization procedure gave us 
both the classical observables --- the Poisson algebra
of functions on the phase space, when $\hbar = 0$ --- and the quantum
observables, when $\hbar = 1$.

In the BV formalism, the classical observables form a 1-shifted
Poisson algebra, and the quantum observables are just a pointed
cochain complex. Starting with a BD-algebra $M$ over $A$, we can recover
both of these structures by taking $\hbar$ to be $0$ or $1$.
If we set $\hbar$ to $0$, \ie{} we pass to the quotient $M/(\hbar)$,
then we obtain a shifted Poisson algebra structure, which we interpret as the
\emph{dequantization} of the BD-algebra $M$.
If we set $\hbar$ to $1$, \ie{} we pass to the quotient $M/(\hbar -1)$, 
then the differential is not a derivation, and so up to quasi-isomorphism
the only remaining algebraic structure is the unit. 
That is, the reduction $M/(\hbar-1)$ is essentially just a pointed
$A$-module. More precisely, a pointed $A$-module is the same thing as
an algebra in $A$-modules for an operad $\rE_{0}$, and the structure
we have on $M/(\hbar-1)$ is encoded by an operad $\teo$; these operads
are weakly equivalent, and so they encode the same kind of information.

The abstract problem of BV quantization is: given a shifted Poisson
algebra $R$, produce a BD-algebra $\widetilde{R}$ whose dequantization
is quasi-isomorphic to $R$. Composing our shifted Heisenberg functor
with the BD-enveloping algebra, we thus get a functorial quantization
procedure that abstracts and encodes the usual approach to BV
quantization for linear systems. The shifted Poisson algebra obtained
from this quantization by setting $\hbar = 0$ can be identified with
the enveloping shifted Poisson algebra of the shifted Lie algebra we
started with. Moreover, the $\teo$-algebra we get from taking $\hbar =
1$ is also an enveloping algebra. More formally, we can sum up our
work on operads and enveloping algebras in \S\ref{sec:opd} as:
\begin{thm}
  There is a commutative diagram of \icats{} and symmetric monoidal
  functors
\[
  \begin{tikzcd}
    {} & & \AAlg_{\teo}(A) \\
    \LLie_{1}(A) \arrow{r}{U_{\txt{BD}}} \arrow[bend
    left]{urr}{U_{\teo}} \arrow[bend right]{drr}{U_{\mathrm{P}_{0}}} &
    \AAlg_{\txt{BD}}(A[\hbar]) \arrow{ur}{\hbar = 1} \arrow{dr}{\hbar = 0}
    \\
     & & \AAlg_{\mathrm{P}_{0}}(A),
  \end{tikzcd}
\]
 that is natural in the commutative differential graded $k$-algebra $A$.
\end{thm}

Although not strictly needed for our main results, we take the time to
prove some further interesting results concerning the symmetric
monoidal enveloping functor $U_{\teo} \colon \LLie_{1}(A) \to
\AAlg_{\teo}(A)$. Firstly, in \S\ref{sec:CE} we show that
(at the model category level) it can be identified with (a shifted
version of) the Chevalley-Eilenberg chains, which describe Lie algebra
homology. Secondly, in \S\ref{sec:Enenv} we use it to construct an
\emph{enveloping $E_{n}$-algebra functor} $\LLie_{1-n}(A) \to
\AAlg_{\mathrm{E}_{n}}(A)$, by a different approach than those of
Fresse \cite{FresseKoszul} and Knudsen \cite{KnudsenEn}. Let us
briefly sketch the idea: By using $\infty$-operads, 
we get from $U_{\teo}$ a functor between \icats{} of $\rE_{n}$-algebras
$\AAlg_{\rE_{n}}(\LLie_{1}(A)) \to \AAlg_{\rE_{n}}(\AAlg_{\rE_{0}}(A))
\simeq \AAlg_{\rE_{n}}(A)$ (this approach does not work on the model
category level); using the bar/cobar adjunction, we then
show that the \icat{} $\AAlg_{\rE_{n}}(\LLie_{1}(A))$ is equivalent to
$\LLie_{1-n}(A)$, by an argument due to To\"en.

\subsection{Extension to Derived Algebraic Geometry}\label{subsec:DAG}
In \S\ref{subsec:BVstack} we show that our functors all have
natural extensions to derived stacks: for instance, given a
quasi-coherent sheaf of 1-shifted quadratic modules on a derived
stack, there is a functorial way to quantize it to a quasi-coherent
sheaf of $\rE_{0}$-algebras. This result is essentially a formal consequence
of the naturality of our constructions in the \icatl{} setting,
together with Lurie's descent theorem for \icats{} of modules.

This result begs the question of what such a quantization means in natural
geometric examples, which we hope to explore in future work. For
example, in light of \cite{CPTVV, Hennion,GR1,GR2}, for a well-behaved
0-shifted symplectic stack $X$, its relative tangent complex
$\mathbb{T}_{X/X_{dR}}$ can be input to our construction.  What does
this quantization mean?

In \S\S\ref{subsec:quantk}--\ref{subsec:quantsympvb} we will focus on
{\em cotangent quantization}, which sends a graded vector bundle $V$
(\ie{} a finite direct sum of shifts of vector bundles) to the
quantization of $V \oplus V^\vee[1]$.  This case has striking
behaviour: the cotangent quantization of a graded vector bundle is a
line bundle (up to a shift).  In other words, this functor factors
through $\PPic$, the stack of invertible sheaves, and hence it behaves
likes a determinant functor.  (It does not possess all the properties
of the determinant, however.)  This feature demonstrates a sense in
which BV quantization is a kind of homological encoding of the path
integral, since the determinant line of a vector space is the natural
home for translation-invariant volume forms on the vector space.

\begin{remark}
  From the viewpoint of physics, more specifically the divergence
  complex perspective that we discuss in Section \ref{div} below, this
  behavior is not too surprising. Indeed, the standard toy example
  of BV quantization produces a cochain complex that is isomorphic 
  to the polynomial de Rham complex on a vector space $V$, 
  shifted down by the dimension of $V$.  
  Poincar\'e's lemma then tells us that we get a one-dimensional
  vector space in degree $-\dim(V)$.
  We leverage this example as far as it can easily go.
\end{remark}

We expect this result to be true in somewhat greater generality, 
likely for 1-shifted symplectic vector bundles, 
which are quadratic modules whose pairing is non-degenerate
and whose underlying module is a sum of shifts of vector bundles. 
For more general quasicoherent sheaves, however, 
the quantization is not invertible. 
On the other hand, we show it is \emph{constructibly} invertible in a certain sense. 
Interpreting the meaning of this behavior is an intriguing question.  
We expect that it is closely related to recent work \cite{Behrend-Fantechi,Joyceetal,Pridham} on vanishing cycles on stacks. 
  (In a sense, the BV formalism is an obfuscated version of the twisted de
Rham complex, as explained in Section \ref{div}, and hence closely related
to vanishing cycles.)

\subsection{Higher BV Quantization and AKSZ
  Theories}\label{subsec:aksz}
As we discussed above, we  construct, for any derived stack $X$, 
a sequence of symmetric monoidal functors
\[ \QQuad_{1}(X) \to \MMod_{\mathcal{O}_{X}\mathbf{c}}{\LLie_{1}}(X) \to
\MMod_{\mathcal{O}_{X}[\hbar,\mathbf{c}]}\AAlg_{\BD}(X) \to \AAlg_{\teo}(X).\]
Using $\infty$-operads, this sequence immediately induces functors between
\icats{} of $\mathrm{E}_{n}$-algebras:
\[ \AAlg_{\mathrm{E}_{n}}(\QQuad_{1}(X)) \to \AAlg_{\mathrm{E}_{n}}(\MMod_{\mathcal{O}_{X}\mathbf{c}}{\LLie_{1}}(X)) \to
\AAlg_{\mathrm{E}_{n}}(\MMod_{\mathcal{O}_{X}[\hbar,\mathbf{c}]}\AAlg_{\BD}(X))
\to \AAlg_{\mathrm{E}_{n}}(X).\]
We mentioned above that $\mathrm{E}_{n}$-algebras in
${\LLie_{1}}(X)$ are equivalent to $(1-n)$-shifted Lie algebras, and
heuristically it looks like there is an analogous description of
$\mathrm{E}_{n}$-algebras in $\QQuad_{1}(X)$. More precisely, we
expect the following:
\begin{conjecture}\label{conj:QuadEn}
  There is a natural equivalence
  $\AAlg_{\mathrm{E}_{n}}(\QQuad_{1}(X)) \simeq
  \QQuad_{1-n}(X)$. The induced functor
  \[ \QQuad_{1-n}(X) \simeq \AAlg_{\mathrm{E}_{n}}(\QQuad_{1}(X)) \to
  \AAlg_{\mathrm{E}_{n}}({\LLie_{1}}(X)) \simeq {\LLie_{1-n}}(X) \] is a
  $(1-n)$-shifted version of the Heisenberg Lie algebra, so the
  composite functor $\QQuad_{1-n}(X) \to \AAlg_{\mathrm{E}_{n}}(X)$ is
  the $\mathrm{E}_{n}$-enveloping algebra of the shifted Heisenberg
  Lie algebra.
\end{conjecture}

\begin{remark}
We expect that this construction recovers on objects the \emph{Weyl $n$-algebras} described in \cite{mark}.
This expectation is motivated by recognizing that our construction here, 
when formulated using factorization algebras,
corresponds to the BV quantization of a very simple AKSZ theory,
with target a shifted symplectic cochain complex,
and hence behaves like an abelian Chern-Simons-type theory 
(more accurately, an abelian BF theory).
This construction is developed in \cite{CG, OGthesis}.
Markarian's construction seems to be an alternative description of this construction.
\end{remark}

In fact, this construction has a further interesting extension. Recall
that given a nice symmetric monoidal \icat{}
$\mathcal{C}$, the \emph{higher Morita category} of $\rE_{n}$-algebras
$\mathfrak{Alg}_{n}(\mathcal{C})$, as constructed in \cite{nmorita},
is a symmetric monoidal $(\infty,n)$-category with
\begin{itemize}
\item objects $\rE_{n}$-algebras in $\mathcal{C}$,
\item 1-morphisms $\rE_{n-1}$-algebras in bimodules in $\mathcal{C}$,
\item 2-morphisms $\rE_{n-2}$-algebra in bimodules in bimodules in $\mathcal{C}$,
\item \ldots
\item $n$-morphisms bimodules in \ldots{} in bimodules in $\mathcal{C}$.
\end{itemize}
In \S\ref{subsec:higherbv} we show that the \icats{} we work with
satisfy the requirements for these higher Morita categories to exist,
and our functors give symmetric monoidal functors between them. In
particular, we get \emph{higher BV quantization} functors
$\mathfrak{Alg}_{n}(\QQuad_{1}(X)) \to
\mathfrak{Alg}_{n}(\AAlg_{\rE_{0}}(\QCoh(X)))$.

We expect that the $i$-morphisms in
$\mathfrak{Alg}_{n}(\QQuad_{1}(X))$ have an interesting
interpretation, in the same way as we conjectured above for the
objects. Specifically, for $i = 1$ they should be a cospan analogue of
the linear version of the Poisson morphisms studied by
\cite{CPTVV,SafronovPoisson}, and for $i > 1$ we should have iterated
versions of this notion. We have learned from Nick Rozenblyum that
such results should be provable using the same techniques as in his
unpublished proof that $\mathrm{E}_{n}$-algebras in
$\mathrm{P}_{k}$-algebras are $\mathrm{P}_{k+n}$-algebras.

It is then attractive to guess that
$\mathfrak{Alg}_{n}(\QQuad_{1}(A))$ receives a symmetric monoidal
functor from an $(\infty,n)$-category
$\txt{Lag}_{(\infty,n)}^{n-1,\txt{lin}}(A)$ where
\begin{itemize}
\item objects are $(n-1)$-shifted symplectic $A$-modules,
\item 1-morphisms are Lagrangian correspondences,
\item $i$-morphisms for $i > 1$ are iterated Lagrangian correspondences.
\end{itemize}
Heuristically, this functor simply takes duals --- \eg{} it would take
an $(n-1)$-shifted symplectic $A$-module $M$ to its dual $M^{\vee}$
with its induced $(1-n)$-shifted pairing. 

The $(\infty,n)$-categories $\txt{Lag}_{(\infty,n)}^{n-1,\txt{lin}}(A)$,
or rather the more general version
$\txt{Lag}_{(\infty,n)}^{k}$ whose objects are $k$-shifted
symplectic derived Artin stacks, will be constructed in forthcoming work
of the second author together with Damien Calaque and Claudia
Scheimbauer. Moreover, it will be shown there that the AKSZ
construction, implemented in this algebro-geometric setting in
\cite{PTVV}, gives for every $k$-symplectic derived Artin stack $X$ 
 an extended oriented topological quantum field theory (TQFT) \[ \txt{AKSZ}_{X} \colon \txt{Bord}^{\txt{or}}_{(\infty,n)} \to
\txt{Lag}_{(\infty,n)}^{k} \]
 (where the dimension $n$ is
arbitrary). Combining this with our hypothetical dualizing functor, we
would have for every $(n-1)$-shifted symplectic cochain complex $X$ a
chain of symmetric monoidal functors of $(\infty,n)$-categories
\[ \txt{Bord}^{\txt{or}}_{(\infty,n)} \to
\txt{Lag}_{(\infty,n)}^{n-1,\txt{lin}}(k) \to
\mathfrak{Alg}_{n}(\QQuad_{1}(k)) \to
\mathfrak{Alg}_{n}(\AAlg_{\rE_{0}}(k)).\]
The resulting TQFT can be interpreted as a \emph{quantization} of the AKSZ
field theory with target $X$.

\subsection{Linear BV Quantization and Integration}
\label{div}

Let's turn here to a more traditional presentation of linear BV quantization and 
interpret it as a hidden version of the de Rham complex 
providing the connection to homological perspectives on integration.
(For further discussion of BV quantization, we recommend \cite{Fiorenza,CMR,Witten,Cos,CG}.)

Let $V$ be a finite-dimensional vector space over, say, the real numbers. 
A standard way to encode the algebraic relations among integrals 
(more accurately, integrands) is via the de Rham complex. 
In particular, if $K \subset V$ is a smooth compact region 
--- like a closed ball of codimension 0 --- 
then two top forms $\alpha$ and $\alpha' = \alpha + \d \beta$ 
satisfy
\[ \int_K \alpha - \int_K \alpha' = \int_K \d \beta = \int_{\partial K} \beta.\]
Hence top forms modulo exact terms encodes integrals modulo boundary integrals.
We want to see what this perspective tells us about the integrals 
that represent toy models of the path integral.

The model case for physics is to fix a quadratic form $Q$ on $V$ with
a global minimum at the origin and consider the integrand $e^{-Q(x)}
\d^nx$.  Up to scale, this provides a probability measure on $V$ whose
average is the origin and which extends around it as a ``Bell curve''.
(On $\RR$, this might be $Q(x) = x^2$, which gives the Gaussian
measure as the integrand.)  It is a toy model of a free theory in
physics, with $Q$ the action functional. (As $Q$ is quadratic, its
equations of motion are linear and so ``free.'')  It is reasonable, in
order to explore this measure, to focus on its \emph{moments}, \ie{} to
understand the integrals
\[ \int_V p(x) e^{-Q(x)} \d^nx,\]
with $p \in \Sym(V^*)$ a polynomial. 
Indeed, the perturbative machinery of Feynman diagrams can be understood as 
formally extending these computations to formal power series.
Our BV approach to this free case then similarly extends and provides
another perspective on the origin of Feynman diagrams as ``homotopy transfer''.
(See \cite{JFG, OGthesis} for more.)
In other words, we want to understand the expected value map
\[ 
\begin{array}{cccc}
\mathbb{E}: & \Sym(V^*) & \to & \RR \\
& p & \mapsto & \int_V p(x) e^{-Q(x)} \d^nx/\int_V e^{-Q(x)} \d^nx.
\end{array}
\]
Note that these integrands decay very fast at infinity and hence are integrable.

Observe that this map factors as a composition
\[  \Sym(V^*)  \xto{e^{-Q(x)}\d^n x}  \Omega^{\dim V}(V)  \xto{\int_V}  \RR. \]
The kernel of $\mathbb{E}$ can be identified with those integrands $p(x) e^{-Q(x)}\d^n x$ 
that are exact, \ie{} in the image of the de Rham differential $\d$.
In fact, the kernel of $\mathbb{E}$ is the image of the 
``divergence against the volume form $\mu = e^{-Q(x)}\d^n x$'' operator
\[
\begin{array}{cccc}
\div_\mu: & \Sym(V^*) \otimes V & \to & \Sym(V^*) \\
&X = \sum_i p_i(x) \frac{\partial}{\partial x_i} & \mapsto & \mathcal{L}_X \mu = \d(\iota_X \mu) = -\sum_i \frac{\partial Q}{\partial x_i} e^{-Q(x)}\d^n x,
\end{array}
\]
sending a vector field $X$ with polynomial coefficients to the Lie derivative of $\mu$ along $X$.
More generally we can use the volume form $\mu = e^{-Q(x)}\d^n x$ to produce an injection 
\[ \iota_\mu: \txt{PV}_{poly}^d(V) = \Sym(V^*) \otimes \Lambda^d V \to \Omega^{\dim V - d}(V) \]
from the polyvector fields on $V$ with polynomial coefficients into de Rham forms.
The de Rham differential preserves the image 
(just note that the derivative of $e^{-Q(x)}$ is a polynomial times itself)
and hence pulls back to a ``divergence operator'' $\div_\mu$ on $\txt{PV}^*_{poly}(V)$.

By construction, this \emph{divergence complex} $\txt{Div} = (\txt{PV}^*_{poly}(V), \div_\mu)$
encodes the moments of the measure $\mu$ in the map 
\[ 
\begin{array}{ccc}
\Sym(V^*) & \to & H^0(\txt{Div}) \cong \RR \\
p & \mapsto & [p] = \mathbb{E}(p).
\end{array}
\]
(In this situation, the rest of the cohomology vanishes, by a Poincar\'e lemma argument.)
Hence it captures the information we most want from the measure.
But this construction has several features that make it possible 
to generalize this approach to infinite-dimensional vector spaces and manifolds 
(\ie{} to actual field theories)
and also to derived settings, where the usual approaches to integration 
do not always work.
Two aspects are:
\begin{itemize}
\item It replaces the measure by the divergence operator, 
and so one can try to axiomatize the properties of divergence complexes
and then search for new examples. In particular, it replaces 
questions about integration by examining relations between integrands.
\item It focuses on functions and their expected values --- 
\ie{} integration against a fixed volume form ---
rather than a general theory of integration. 
Thus, by contrast to the de Rham complex, it makes sense in infinite dimensions,
whereas top forms make no sense there.
\end{itemize}
Here we will focus on the first aspect, 
using an operad introduced by Beilinson and Drinfeld \cite{BD} for the axiomatization,
and introduce a wealth of examples from higher algebra and derived geometry.
The second aspect is pursued wherever BV quantization is used in field theory, 
such as \cite{CMR,Cos,CG}.
Of course, this BV approach to the path integral does not resolve
all challenges! Viewing integration this way loses some of the advantages
of other perspectives and introduces new puzzles and challenges.

Let us now rapidly sketch the algebraic features of the divergence complex that we will focus on.
First, polyvector fields have a natural graded-commutative product by wedging
(in parallel with de Rham forms, but \emph{not} preserved by the map $\iota_\mu$!).
Second, polyvector fields have a shifted Poisson bracket, known as the Schouten bracket,
which is defined by extending the natural action of vector fields on functions and vector fields.
Explicitly, we define
\[ \{X,f\} = \mathcal{L}_X f \quad \txt{and} \quad \{X,Y\} = [X,Y] \]
for $f$ a function and $X,Y$ vector fields. 
(In general, there are signs to keep track of, due to the Koszul sign rule,
but we will not focus on that in this introduction.)
Finally, the divergence operator is a derivation with respect to the bracket 
(\ie{} with respect to the shifted Lie algebra structure) 
but it is \emph{not} a derivation with respect to the commutative product.
Instead, it satisfies the relation
\[
\div(\alpha \beta) = \div(\alpha)\beta + (-1)^{\alpha} \alpha\,\div(\beta) + \{\alpha,\beta\}
\]
for any $\alpha, \beta$ polyvector fields. 
(These features do not depend on the coefficients being polynomial 
and hold for holomorphic or smooth coefficients too.)
This relation says that the bracket encodes the failure of $\div$ to be a derivation.
It should be seen as analogous to deformation quantization, 
where the failure to be commutative is encoded in the commutator bracket
and, to first order in $\hbar$, this failure is the Poisson bracket.
This perspective leads directly to the  definition of a Beilinson-Drinfeld algebra
(see Definition \ref{def bd}).

In our model case, the shifted Poisson algebra is
\[ (\txt{PV}_{poly}(V), \{Q,-\}). \]
(Note that the zeroth cohomology is precisely functions on the critical set of $Q$,
which fits nicely with the fact that observables in a classical theory should be
functions on the critical points of the action.)
The BV quantization is
\[ (\txt{PV}_{poly}(V)[\hbar], \{Q,-\}+\hbar \triangle), \]
where $\triangle = \div_{\txt{Leb}}$ is divergence against the Lebesgue measure $\d^n x$.
(In formulas, one usually sees $\triangle = \sum_i \partial^2/\partial x_i \partial \xi_i$,
where the $x_i$ are a basis for $V^*$ and the $\xi_i$ are the dual basis for $V[1]$.)
In this example, we explicitly see that the deformation of the differential
amounts to taking into account the relations among integrands.

There is one final thing to note about this model example,
which makes manifest the analogy with Weyl quantization.
Observe that the shifted Poisson bracket $\{-,-\}$ is linear in nature.
If we fix a basis $\{x_i\}$ for $V^*$ and a dual basis $\{\xi_i\}$ for $V$,
then 
\[ \txt{PV}_{poly}(V) \cong \RR[x_1,\ldots,x_n,\xi_1,\ldots,\xi_n], \]
with $\dim(V) = n$ and the $x_i$'s in degree zero and the $\xi_i$'s in degree one.
The bracket is
\[ \{x_i,\xi_j\} = \delta_{ij} \quad \txt{and} \quad \{x_i,x_j\} = 0 = \{\xi_i,\xi_j\},\]
which looks just like a shifted version of the Poisson bracket on the symplectic vector space $T^* \RR^n$.
Indeed, we can view this shifted bracket as arising from a shifted skew-symmetric pairing
\[ \omega: (V^* \oplus V[1])^{\otimes 2} \to \RR[1], \]
which is simply the restriction of $\{-,-\}$ to the linear space 
generating the graded-symmetric algebra of polyvector fields.
There is then a shifted Lie algebra $\g$ given by centrally extending the abelian Lie algebra $V^* \oplus V[1]$ 
by $\RR \c$, \ie{} 
\[ \RR\c \to \g \to V^* \oplus V[1], \]
where the shifted Lie bracket is
\[ [p,q] = \c \omega(p,q). \]
Thus $\g$ is clearly a kind of shifted Heisenberg Lie algebra.
To obtain the BV quantization, we do not take the universal enveloping algebra,
which would produce an associative algebra, 
but instead take the enveloping BD algebra $U_\BD(\g)$.
(We construct this enveloping algebra functor in the text.) 
The quotient $U_\BD(\g)/(\c = \hbar)$ recovers the standard BV quantization on the nose.

\subsection{Notations and Conventions}
Throughout this paper, we work in the setting of cochain complexes
over a field $k$ of characteristic zero.  In other words, everything
is differential graded, aside from the occasional motivational remark.
Hence, when we speak about an algebra, we always mean an algebra
object in some category (or higher category) with a forgetful functor
to cochain complexes.  After the introduction, we will simply speak
about commutative or Lie algebras and not differential graded
commutative algebras or differential graded Lie algebras.
Notationally, $A$ typically denotes a commutative algebra in cochain
complexes over the field $k$ (\ie{} a cdga), and $\g$ typically denotes
a Lie algebra in cochain complexes over $k$ (\ie{} a dgla).  A module
over an algebra always means a module object and we will not use the
term differential graded module.  Thus we write $A$-module rather that
differential graded $A$-module and so~on.

To construct our \icats{} and functors we will also need to work with
both model categories and simplicial categories. To distinguish the
three kinds of mapping objects that arise we adopt the convention that
for an ordinary category $\mathbf{C}$ we write
$\Hom_{\mathbf{C}}(x,y)$ for the set of maps between objects $x$ and
$y$, for a simplicial category $\mathbf{C}$ we write
$\sHom_{\mathbf{C}}(x,y)$ for the simplicial set of maps, and for an
\icat{} $\mathcal{C}$ we write $\Map_{\mathcal{C}}(x,y)$ for the space
of maps.

In many cases, we will have to work with a model category, a
simplicial category and an \icat{} that encode the same homotopy
theory, and we use a typographical convention to distinguish
these. For instance, there is a category $\Mod(A)$ of $A$-modules in
$\Mod(k)$, the category of cochain complexes over $k$.  There is also
a simplicial category $\sMod(A)$ of (cofibrant) $A$-modules, and there
is an \icat{} $\MMod(A)$ of $A$-modules.  Similarly, for $\mathbf{O}$
an operad in the category $\Mod(A)$ of $A$-modules, there is a
category $\Alg_{\mathbf{O}}(A)$ of $\mathbf{O}$-algebras in $\Mod(A)$,
there is a simplicial category $\sAlg_{\mathbf{O}}(A)$ of (cofibrant)
$\mathbf{O}$-algebras in the simplicial category $\sMod(A)$, and there
is an \icat{} $\AAlg_{\mathbf{O}}(A)$ of $\mathbf{O}$-algebras in the
\icat{} $\MMod(A)$.

There are two exceptions to the convention we just described. 
When $\mathbf{O}$ is the commutative operad $\Comm$, we use the
abbreviated notations $\Comm(A)$, $\sComm(A)$, and $\CComm(A)$, and
when $\mathbf{O}$ is the Lie operad $\Lie$, or more generally the
$n$-shifted Lie operad $\Lie_n$, we use $\Lie_n(A)$, $\sLie_{n}(A)$,
and $\LLie_{n}(A)$.

We write $\Lambda^{2}_{[n]}X$ for the shifted antisymmetric square
\[ \Lambda^{2}_{[n]}X := \Lambda^{2}(X[n])[-2n].\]
In other words,
\[ \Lambda^{2}_{[n]}X \cong
\begin{cases}
  \Lambda^{2}X, & \txt{$n$ even},\\
  \Sym^{2}X, & \txt{$n$ odd}.
\end{cases}
\]

\subsection{Acknowledgments}

OG thanks Kevin Costello for teaching him the BV formalism and
pointing out that it behaves like a determinant, an idea he pursued in
his thesis and that prompted this collaboration. He also thanks Nick
Rozenblyum, Toly Preygel, and Thel Seraphim for many helpful conversations around
quantization and higher categories.
RH thanks Irakli Patchkoria for help with model-categorical
technicalities and Dieter Degrijse for some basic homological
algebra.
Together we thank Theo Johnson-Freyd, David Li-Bland, and Claudia Scheimbauer
for stimulating conversations around these topics, 
particularly the pursuit of higher Weyl quantization.

Finally, this work was begun at the Max Planck Institute for Mathematics
when RH and OG were both postdocs there, and we deeply appreciate the 
open and stimulating atmosphere of MPIM that made it so easy to begin our collaboration.
Moreover, it is through the MPIM's great generosity that we were able to continue
work and finish the paper during several visits by RH.

\section{Operads and Enveloping Algebras}\label{sec:opd}

Our goal in this section is to introduce the operads that play a
central role in BV quantization and to construct a collection of
functors between their \icats{} of algebras.  To do so, we first
explain what we mean by the \icat{} of algebras over a $k$-linear
operad $\mathbf{O}$ --- although notions of enriched \iopds{} have
been introduced in \cite{enropd,symmseq}, their theory is not yet
sufficiently developed for our purposes.  Thus, the beginning of this
section is devoted to higher-categorical machinery: we draw together
results from the literature in order to
\begin{enumerate}
\item produce a model category $\Alg_{\mathbf{O}}(A)$ of $\mathbf{O}$-algebras in $\Mod(A)$,
where $A$ is a commutative algebra in cochain complexes over $k$ and $\Mod(A)$ is
a model category of $A$-modules, and then
\item extract a simplicial category $\sAlg_{\mathbf{O}}(A)$ of $\mathbf{O}$-algebras in a simplicial category 
$\sMod(A)$ of  $A$-modules, and finally 
\item provide  an \icat{} $\AAlg_{\mathbf{O}}(A)$ of $\mathbf{O}$-algebras in the \icat{} $\MMod(A)$ of $A$-modules.
\end{enumerate}
With these tools available, we turn to our problem of interest.

The main result of this section can then be summarized in the following commuting diagram of symmetric
monoidal functors:
\[
\begin{tikzcd}
{} & & {\LLie_1}(A)^\oplus \arrow[bend
right]{ddl}[swap]{U_{\PZ}} \arrow[bend left]{ddrr}{\CL}
\arrow{d}{U_{\BD} \circ (A[\hbar] \otimes_A -)} \arrow[bend right]{dll}{\txt{forget}} & & \\
\MMod(A)^{\oplus} \arrow{d}{\Sym} & &  \AAlg_{\BD}(A[\hbar])^{\otimes_{A[\hbar]}}
 \arrow{dl}[swap]{\ev_{\hbar = 0}} \arrow{dr}{\ev_{\hbar = 1}} & &\\
\CComm(A)^{\otimes_{A}} &\AAlg_{\PZ}(A)^{\otimes_{A}}
\arrow{l}{\txt{forget}} & & \AAlg_{\teo}(A)^{\otimes_{A}} \arrow{r}{\txt{forget}} & \MMod(A)^{\otimes_{A}}
\end{tikzcd}
\]
which says in essence that 
\begin{enumerate}
\item every shifted Lie algebra $\g$ in $\MMod(A)$ generates a shifted Poisson algebra $U_{\PZ}(\g) = \Sym(\g)$ 
that admits a natural BV quantization by its $\BD$-enveloping algebra $U_{\BD}(\g[\hbar])$, and
\item when $\hbar$ is specialized to $1$, this quantization reduces to
  $\CL(\g) = \mathrm{C}^{\Lie}(\g[-1])$ (\ie{} the derived
  coinvariants, or Chevalley-Eilenberg chains, of the {\it un}shifted
  Lie algebra) .
\end{enumerate}
These relationships certainly seem to be folklore among the community
who work with the BV formalism, but we need the result in this higher-categorical setting and so provide proofs. 
(See, for instance, \cite{BD, BashkirovVoronov, BraunLazarev}.) 
We will begin by proving everything in the setting of model categories and 
then apply our machinery to obtain the desired statements for \icats{}.

\subsection{Model Categories of Modules and Operad Algebras}
\label{sec: MCMOA}

Let $k$ be a field of characteristic~$0$. We write $\Mod(k)$ for the
category of (unbounded) cochain complexes of $k$-modules, equipped with the
standard projective model structure:
\begin{propn}[Hinich, Hovey]
  The category $\Mod(k)$ has a left proper combinatorial model
  structure where
  \begin{itemize}
  \item the weak equivalences are the quasi-isomorphisms,
  \item the fibrations are the levelwise surjective maps.
  \end{itemize}
  Moreover, this is a symmetric monoidal model category with respect to
  the usual tensor product of cochain complexes.
\end{propn}
\begin{proof}
  The model structure is constructed in \cite{HinichHAHA}*{Theorem
    2.2.1}; see also \cite{HoveyModCats}*{Theorem 2.3.11} for a more
  detailed construction that works over an arbitrary ring. It is a
  symmetric monoidal model category by
  \cite{HoveyModCats}*{Proposition 4.2.13}.
\end{proof}

If $A$ is a commutative algebra over $k$, \ie{} a commutative algebra object in
$\Mod(k)$, then we can lift this model structure to the category
$\Mod(A)$ of $A$-modules in $\Mod(k)$:
\begin{propn}[Hinich, Schwede-Shipley]\ Let $A$ be a commutative
  algebra over $k$. Then the category $\Mod(A)$ has a left proper
  combinatorial model structure where the weak equivalences and
  fibrations are the maps whose underlying maps of cochain complexes
  are weak equivalences and fibrations in $\Mod(k)$. If the underlying
  cochain complex of $A$ is cofibrant, then the forgetful functor also
  preserves cofibrations. Moreover, this is a symmetric monoidal model
  category with respect to~$\otimes_{A}$.
\end{propn}
\begin{proof}
  This is \cite{HinichHAHA}*{\S 3} or
  \cite{SchwedeShipleyAlgMod}*{Theorem 4.1}.
\end{proof}

\begin{remark}
  \cite{BarthelMayRiehlDG}*{Theorems 9.10 and 9.12} give an explicit
  characterization of the cofibrant objects and cofibrations in
  $\Mod(A)$. 
\end{remark}

Since $\Mod(A)$ is a symmetric monoidal model category, if $M$ is a
cofibrant object then the functor $M \otimes_{A} \blank$ preserves
quasi-isomorphisms between cofibrant objects. In fact, slightly more
is true:

\begin{lemma}\label{lem:cofmodtensor}
  If $M$ is a cofibrant object of $\Mod(A)$, then the functor $M \otimes_{A}
  \blank$ preserves quasi-isomorphisms.
\end{lemma}

This fact is standard; we include a short proof for
completeness. 

\begin{proof}
  $\Mod(A)$ is a cofibrantly generated model category, with the set
  $\mathcal{I}$ of generating cofibrations being $S^{n}_{A} := A
  \otimes S^{n}_{k} \to A \otimes D^{n+1}_{k} =: D^{n+1}_{A}$, where
  $S^{n}_{k} := k[n]$ is the cochain complex with $k$ in degree $-n$ and
  $0$ elsewhere, and $D^{n+1}_{k}$ is that with $k$ in degrees $-n$ and
  $-n-1$, with differential $\id_{k}$, and $0$ elsewhere
  (\cf{} \cite{BarthelMayRiehlDG}*{Theorem 3.3}). It follows that the
  cofibrant $A$-modules are the objects that are retracts of
  $\mathcal{I}$-cell complexes, where the latter are the objects $X$
  that can be written as colimits of a sequence of maps $0=F_{0} \to
  F_{1} \to F_{2} \to \cdots$, with each $F_{n-1} \to F_{n}$ obtained
  as a pushout \nolabelcsquare{\coprod_{i \in T_{n}}
    S^{d_{i}}_{A}}{F_{n-1}}{\coprod_{i \in T_{n}}
    D^{d_{i}+1}_{A}}{F_{n},} where $T_{n}$ is a set.  Note in
  particular that each filtration quotient $F_{n}/F_{n-1}$ is of the
  form $\coprod_{i \in T_{n}} S^{n+1}_{A}$, \ie{} it is a sum of copies
  of shifts of $A$.

  Now suppose $X$ is a cofibrant $A$-module, and $f \colon M \to M'$
  is a quasi-isomorphism. We wish to prove that $X \otimes f$ is a
  quasi-isomorphism. Since quasi-isomorphisms are closed under
  retracts, it suffices to prove this under the assumption that $X$ is
  an $\mathcal{I}$-cell complex; we therefore fix a filtration $F_{n}$
  of $X$ as above. We claim that the induced maps $F_{n} \otimes_{A} M
  \to F_{n+1} \otimes_{A} M$ are injective, so that we get a
  filtration of $X \otimes_{A}M$. Assuming this, we have short exact
  sequences of cochain complexes over $k$,
  \[ 
  0 \to F_{n-1} \otimes_{A} M \to F_{n} \otimes_{A} M \to F_{n}/F_{n-1}
  \otimes_{A} M \to 0,
  \] 
  and using the associated long exact sequence, we see by induction 
  that $F_{n} \otimes_{A} M \to F_{n} \otimes_{A} M'$ is a quasi-isomorphism, 
  since $F_{n}/F_{n-1} \otimes_{A} M \to F_{n}/F_{n-1} \otimes_{A} M'$ is a
  quasi-isomorphism (being a sum of shifts of $f$). As quasi-isomorphisms are closed under filtered
  colimits, it follows that $X \otimes_{A} M \to X \otimes_{A} M'$ is
  also a quasi-isomorphism.
  
  To prove injectivity for $F_{n}\otimes_{A} M \to F_{n+1} \otimes_{A}
  M$, observe that we can prove this on the level of underlying graded
  $k$-modules. The freeness of $F_{n+1}/F_{n}$ implies that we can
  choose a splitting of $F_{n+1} \to F_{n+1}/F_{n}$, which gives a
  splitting of $F_{n+1} \otimes_{A} M \to F_{n+1}/F_{n} \otimes_{A}
  M$. Thus we have for every $i \in \mathbb{Z}$ a split
  short exact sequence \[0 \to (F_{n} \otimes_{A} M)_{i}\to (F_{n+1}
  \otimes_{A} M)_{i} \to (F_{n+1}/F_{n} \otimes_{A} M)_{i} \to 0\] of
  $k$-modules, which in particular implies that the map $(F_{n} \otimes_{A}
  M)_{i}\to (F_{n+1} \otimes_{A} M)_{i}$ is injective.
\end{proof}

For later use, we note a useful consequence of this:
\begin{lemma}\label{lem:Symnqiso}
  Suppose $A$ is a commutative algebra over $k$. Then the $n$th
  symmetric power functor
  $\Sym^{n}_{A} \colon \Mod(A) \to \Mod(A)$
  preserves quasi-isomorphisms between cofibrant $A$-modules.
\end{lemma}
\begin{proof}
  The functor $\Sym^{n}_{A}$ is defined by the tensor product $A
  \otimes_{A[\Sigma_{n}]} (\blank)^{\otimes_{A} n}$ where $A$ has the
  trivial $\Sigma_{n}$-action and $(\blank)^{\otimes_{A} n}$ has the
  obvious action by permuting the factors. Since $k$ is a field of
  characteristic zero, every module over $k[\Sigma_{n}]$ is
  projective. In particular, $k$ is a projective
  $k[\Sigma_{n}]$-module, and hence it is cofibrant in
  $\Mod(k[\Sigma_{n}])$. Since $A[\Sigma_{n}] \otimes_{k[\Sigma_{n}]}
  \blank$ is a left Quillen functor, this implies that $A$ is
  cofibrant in $\Mod(A[\Sigma_{n}])$. It therefore follows from
  Lemma~\ref{lem:cofmodtensor} that the functor $A
  \otimes_{A[\Sigma_{n}]} (\blank)$ preserves quasi-isomorphisms. We
  are left with showing that if $M \to N$ is a quasi-isomorphism of
  cofibrant $A$-modules, then $M^{\otimes_{A} n} \to N^{\otimes_{A}
    n}$ is a quasi-isomorphism, which follows from $\blank \otimes_{A}
  \blank$ being a left Quillen bifunctor.
\end{proof}

It will also be useful to know that in the case of a field we can relax the
assumption that $M$ is cofibrant:

\begin{lemma}\label{lem:unbddflat}
  For every $X \in \Mod(k)$, the functor $X \otimes \blank$ preserves
  quasi-isomorphisms.
\end{lemma}

\begin{proof}
  By \cite{HoveyModCats}*{Lemma 2.3.6}, any bounded-above cochain complex
  of $k$-modules is cofibrant, so the result holds in this case by
  Lemma~\ref{lem:cofmodtensor}. But any cochain complex $X$ is a
  filtered colimit of bounded-above cochain complexes. Since the tensor
  product commutes with colimits in each variable and
  quasi-isomorphisms are closed under filtered colimits, we obtain
  the result.
\end{proof}

\begin{propn}\label{propn:commutative algebraQAdj}
 Any map of commutative algebras $\phi \colon A \to B$ induces a Quillen adjunction
 \[ 
 \phi_{!} := B \otimes_{A} \blank : \Mod(A) \rightleftarrows \Mod(B) : \phi^{*}. 
 \] 
 If $\phi$ is a quasi-isomorphism, then this adjunction is a Quillen equivalence.
\end{propn}

\begin{proof}
  It is a Quillen adjunction because weak equivalences and fibrations are detected in $\Mod(k)$.
  It is a Quillen equivalence for $\phi$ a quasi-isomorphism by
  \cite{HinichHAHA}*{Theorem 3.3.1} or
  \cite{SchwedeShipleyAlgMod}*{Theorem 4.3}, together with
  Lemma~\ref{lem:cofmodtensor}.
\end{proof}

If $\mathbf{O}$ is an operad in $\Mod(A)$, we can  lift the model
structure on $\Mod(A)$ to the category $\Alg_{\mathbf{O}}(A)$ of
$\mathbf{O}$-algebras in $A$-modules:

\begin{propn}[Pavlov--Scholbach]\label{propn:opdch}\ 
  \begin{enumerate}[(i)]
  \item The category $\Alg_{\mathbf{O}}(A)$ has a model structure
    where the weak equivalences and fibrations are the maps whose
    underlying maps of $A$-modules are weak equivalences and
    fibrations in $\Mod(A)$.
  \item If $\mathbf{O}(n)$ is cofibrant in $\Mod(A)$ for all $n$ and the unit $A
    \to \mathbf{O}(1)$ is a cofibration, then the forgetful functor
    from $\Alg_{\mathbf{O}}(A)$ to
    $\Mod(A)$ also preserves cofibrations.
  \item Any map $f \colon \mathbf{O} \to \mathbf{P}$ of operads in
    $\Mod(A)$ gives rise to a Quillen adjunction
    \[ f_{!} \colon \Alg_{\mathbf{O}}(A) \rightleftarrows
    \Alg_{\mathbf{P}}(A) : f^{*}.\] If $f$ is a weak equivalence then
    this is a Quillen equivalence.
  \item Any map of commutative algebras $\phi \colon A \to B$ gives rise to a Quillen
    adjunction
    \[ 
    (\phi_{!})_{*} : \Alg_{\mathbf{O}}(A) \rightleftarrows \Alg_{\phi_{!}\mathbf{O}}(B) : (\phi^{*})_{*},
    \]
    where $\phi_{!}\mathbf{O}$ denotes the base-changed operad $B
    \otimes_{A} \mathbf{O}$. 
    This adjunction is a Quillen equivalence if $\phi$ is a quasi-isomorphism 
    and one of the following holds:
    \begin{enumerate}[(a)]
    \item $\mathbf{O}$ is cofibrant,
    \item $\mathbf{O}$ is $A \otimes \mathbf{O}'$ for some operad
      $\mathbf{O}'$ in $\Mod(k)$,
    \item $A$ is an $R$-algebra for some commutative algebra $R$, $\mathbf{O}$ is $A
      \otimes_{R} \mathbf{O}'$ for some operad $\mathbf{O}'$ in
      $\Mod(R)$, and the underlying $R$-modules of $A$ and $B$ are
      cofibrant.
    \end{enumerate}
  \end{enumerate}
\end{propn}

\begin{remark}
  Over $k$, most of these results are due to Hinich~\cite{HinichHAHA}:
  (i) is \cite{HinichHAHA}*{Theorem 4.1.1} and (iii) is
  \cite{HinichHAHA}*{Theorem 4.7.4}.
\end{remark}

As special cases, we have the model categories $\Comm(A)$ of
commutative algebras and $\Lie_{n}(A)$ of $n$-shifted Lie algebras
in~$\Mod(A)$.

\begin{proof}
We use results from \cite{PavlovScholbachOpd}, 
whose  hypotheses hold for $\Mod(k)$ by \cite{PavlovScholbachSymm}*{\S 7.4} 
and hold for $\Mod(A)$ for any commutative algebra $A$ over $k$ by \cite{PavlovScholbachSymm}*{Theorem 5.3.1}. 
Then (i) follows from
  \cite{PavlovScholbachOpd}*{Theorem 5.10} and  (ii) from
  \cite{PavlovScholbachOpd}*{Theorem 6.6}. The adjunctions in (iii)
  and (iv) are obviously Quillen adjunctions, and the adjunction in
  (iii) is a Quillen equivalence for
  $f$ a weak equivalence by \cite{PavlovScholbachOpd}*{Theorem
    7.5}. The adjunction in (iv) is a Quillen equivalence in case (a)
  by \cite{PavlovScholbachOpd}*{Theorem 8.10}. 
  To prove case (b),
  let $r: \mathbf{O}'' \to \mathbf{O}'$ be a cofibrant
  replacement in operads in $\Mod(k)$. We then have a commutative
  square of left Quillen functors
  \csquare
  {\Alg_{\mathbf{O}''}(A)}{\Alg_{\mathbf{O}''}(B)}{\Alg_{\mathbf{O}'}(A)}{\Alg_{\mathbf{O}'}(B)}
  {(\phi_{!})_{*}}{r_{!}}{r_{!}}{(\phi_{!})_{*}}
  Since $k$ is a field, Lemma~\ref{lem:unbddflat} tells us that
  $R \otimes r: R \otimes \mathbf{O}'' \to R \otimes \mathbf{O}'$ 
  is again a weak equivalence for any commutative algebra $R$. By (iii) this
  implies that both vertical morphisms are left Quillen
  equivalences. We also know that the top horizontal map is a left
  Quillen equivalence by (iv)(a), so it follows that the bottom horizontal
  map must be one too. Case (c) is proved similarly, taking a
  cofibrant replacement $\mathbf{O}'' \to \mathbf{O}'$ in operads in
  $\Mod(R)$ and using Lemma~\ref{lem:cofmodtensor} to conclude that $A
  \otimes_{R} \mathbf{O}'' \to B \otimes_{A} \mathbf{O}'$ is a weak
  equivalence.
\end{proof}

\begin{remark}
The results of Pavlov and Scholbach encompass a broader class of examples,
including model categories of cochain complexes of bornological and convenient vector spaces constructed in \cite{Wallbridge}.
These would form a natural context for many examples coming from field theory
where our formulation of functorial BV quantization would apply,
but we will restrict our efforts here to an algebraic setting.
\end{remark}

\subsection{$\infty$-Categories of Modules and Operad Algebras}

From the model categories discussed in \S\ref{sec: MCMOA}, we can obtain \icats{} by
inverting the weak equivalences, \ie{} the quasi-isomorphisms. 
Let $A$ be a commutative algebra over $k$.
We write
\begin{itemize}
\item $\MMod(A)$ for the \icat{} obtained from~$\Mod(A)$, 
\item $\AAlg_{\mathbf{O}}(A)$ for the \icat{} obtained from
  $\Alg_{\mathbf{O}}(A)$, with $\mathbf{O}$ an operad in~$\Mod(A)$.
\end{itemize}
Here we will use the results of \S\ref{sec: SEMC} to show that these
\icats{} can alternatively be described using the standard simplicial
category structures defined by tensoring with the algebras of
polynomial differential forms:

\begin{defn}
  Let $\Omega(\Delta^{n})$ denote the commutative differential graded
  $k$-algebra of polynomial differential forms on $\Delta^{n}$.
  That is, $\Omega(\Delta^{n}) = k[x_{1},\ldots,x_{n},dx_{1},\ldots,dx_{n}]$ 
  where each $x_i$ has degree 0 and $dx_i$ has degree 1 and 
  the differential is the derivation determined by $d(x_j) = dx_j$.  
  This construction extends to a unique limit-preserving functor $\Omega \colon \sSet^{\op} \to \Comm(k)$.
\end{defn}

For all the categories $\mathbf{C}$ considered above, we can use the
simplicial object $\Omega(\Delta^{\bullet})$ to define a simplicial
enrichment, by taking the mapping spaces to be $\mathbf{C}(X,
\Omega(\Delta^{\bullet}) \otimes Y)$. We will denote the simplicial categories
obtained in this way from the \emph{cofibrant} objects in the model
categories above by $\sMod(A)$, and
$\sAlg_{\mathbf{O}}(A)$.

\begin{lemma}[{Bousfield-Gugenheim \cite{BousfieldGugenheim}*{\S 8}}]
The functor  $\Omega \colon \sSet^{\op} \to \Comm(k)$ is a right Quillen functor.
\end{lemma}
\begin{proof}
  The functor $\Omega$ has a left adjoint by
  \cite{BousfieldGugenheim}*{8.1}, and this is a left Quillen functor
  by \cite{BousfieldGugenheim}*{Lemma 8.2, Proposition 8.3}.
\end{proof}

\begin{lemma}\label{lem:OmegaReedy}
  For every cochain complex $X$, the simplicial cochain complex
  $\Omega(\Delta^{\bullet}) \otimes X$ is Reedy fibrant. Moreover, the maps
  $\Omega(\Delta^{n}) \otimes X \to \Omega(\Delta^0) \otimes X \cong X$ are all quasi-isomorphisms.
\end{lemma}

\begin{proof}
  The $n$th matching object for $\Omega(\Delta^{\bullet})$ is
  $\Omega(\partial \Delta^{n})$, and since $\Omega$ is a right Quillen
  functor, the map $\Omega(\Delta^{n}) \to \Omega(\partial \Delta^{n})$
  is a fibration and hence $\Omega(\Delta^{\bullet})$ is Reedy fibrant. We can
  moreover identify the matching object for $\Omega(\Delta^{\bullet}) \otimes
  X$ with $\Omega(\partial \Delta^{n}) \otimes X$ --- this boils down
  to the fact that over a field the tensor product preserves finite
  limits in each variable. The levelwise surjectivity of
  $\Omega(\Delta^{n})\to \Omega(\partial \Delta^{n})$ gives levelwise
  surjectivity of $\Omega(\Delta^{n}) \otimes X \to \Omega(\partial
  \Delta^{n}) \otimes X$, so this is again a fibration, as
  required. The second point follows from Lemma~\ref{lem:unbddflat}.
\end{proof}

\begin{lemma}\label{lem:Chkcohfr}\ 
  \begin{enumerate}[(i)]
  \item The simplicial monad $\Omega(\Delta^{\bullet}) \otimes \blank$ gives a
    coherent right framing on $\Mod(k)$ (in the sense of Definition~\ref{defn:cohfr}).
  \item More generally $(A \otimes \Omega(\Delta^{\bullet})) \otimes_{A}
    \blank$ gives a coherent right framing of $\Mod(A)$ for
    $A$ any commutative algebra over~$k$.
  \item If $\mathbf{O}$ is an operad in $\Mod(k)$ and $A$ is a
    commutative algebra over $k$, then $(A \otimes
    \Omega(\Delta^{\bullet})) \otimes_{A} \blank$ (with
    $\mathbf{O}$-algebra structure from the base change adjunction) is
    a coherent right framing on $\Alg_{\mathbf{O}}(A)$.
  \end{enumerate}
\end{lemma}
\begin{proof}
  Monadicity is clear since these functors come from adjunctions. The
  remaining conditions can be checked in $\Mod(k)$, where we proved
  them in Lemma~\ref{lem:OmegaReedy}.
\end{proof}

Combining this with Proposition~\ref{propn:cohfrloc}, we get:

\begin{cor}\label{cor:Chkicat}\ 
  \begin{enumerate}[(i)]
  \item The simplicial category $\sMod(A)$ is fibrant for every $A \in
    \Comm(k)$, and its
    coherent nerve is equivalent to the \icat{} $\MMod(A)$.
  \item The simplicial category $\sAlg_{\mathbf{O}}(A)$ is fibrant for
    every $A \in \Comm(k)$ and every operad $\mathbf{O}$ in $\Mod(A)$,
    and its coherent nerve is equivalent to the \icat{} $\AAlg_{\mathbf{O}}(A)$.
  \end{enumerate}
\end{cor}

The Quillen adjunctions induced by maps of algebras and operads of
Proposition~\ref{propn:commutative algebraQAdj} and
Proposition~\ref{propn:opdch}(iii--iv) induce adjunctions of \icats{}
(as proved in \cite{MazelGeeQAdj} for not necessarily simplicial model
categories such as these). However, since tensor products are not
strictly associative, the left adjoints are only
\emph{pseudo}functorial in the commutative algebra variable.  Since
these functors, unlike their right adjoints, are compatible with the
simplicial categories we have just described, we quickly point out how
to obtain a functor of \icats{}:

\begin{lemma}\label{lem:AlgOftr}
  Let $R$ be a commutative algebra over $k$ and $\mathbf{O}$ an operad in
  $\Mod(R)$. There is a functor $\AAlg_{\mathbf{O}}(R \otimes \blank):
  \CComm(k) \to \LCatI$ taking $A$ to $\AAlg_{\mathbf{O}}(R \otimes
  A)$.
\end{lemma}

\begin{proof}
  The proof follows that of \cite{freepres}*{Lemma A.24}, and we will freely use
  notation and ideas from there in the proof here (but nowhere else in this paper).  
  We have a normal pseudofunctor from
  commutative algebras over $k$ to (fibrant) simplicial categories taking $A$ to
  $\sAlg_{\mathbf{O}}(R \otimes A)$. Using the Duskin nerve
  \cite{Duskin} of 2-categories as in \cite{freepres}*{\S A} this
  gives a functor of quasicategories $\mathrm{N}\Comm(k) \to
  \mathrm{N}_{(2,1)}\LCAT_{\Delta}$.  If we restrict to
  \emph{cofibrant} commutative algebras, then this functor takes quasi-isomorphisms
  of commutative algebras to weak equivalences of simplicial categories by
  Proposition~\ref{propn:opdch}(iv)(c); it thus induces a functor from
  the localization of $\mathrm{N}\Comm(k)^{\txt{cof}}$ at the
  quasi-isomorphisms, which is $\CComm(k)$, to the localization of
  $\mathrm{N}_{(2,1)}\LCAT_{\Delta}$ at the weak equivalences of
  simplicial categories, which is $\LCatI$ since by \cite{HA}*{Theorem
    1.3.4.20} it is equivalent to the localization of the 1-category
  of simplicial categories at the weak equivalences.
\end{proof}

We also note a useful technical result:

\begin{propn}[Pavlov-Scholbach, \cite{PavlovScholbachOpd}*{Proposition 7.8}]\label{propn:opdforgetsifted}
  Let $\mathbf{O}$ be an operad in $\Mod(A)$ such that the unit map $A \to
  \mathbf{O}(1)$ is a cofibration and $\mathbf{O}(n)$
  is a cofibrant $A$-module for every $n$. Then the forgetful functor
  $\AAlg_{\mathbf{O}}(A) \to \MMod(A)$ detects sifted colimits. \qed
\end{propn}

\subsection{Some Operads}
In this section, we introduce the operads relevant to our construction: 
$\Lie_{n}$, $\PZ$, and $\teo$, which live in cochain complexes over $k$, and
$\BD$, which lives in cochain complexes over the algebra $k[\hbar]$, where $\hbar$ has degree zero.

Before defining these operads, we need to review some material, for which we use \cite{LV} as a convenient reference.

\begin{defn}
The \emph{Hadamard tensor product} $\htimes$ of operads (see Section 5.3 of \cite{LV})
has $n$-ary operations
\[
(\mathbf{O} \htimes \mathbf{P})(n) = \mathbf{O}(n) \otimes \mathbf{P}(n),
\]
where the permutation group $\Sigma_{n}$ acts diagonally on the tensor
product, and the composition of operations is in $\mathbf{O}$
and $\mathbf{P}$ independently.  For instance, composition $\circ_i$
in the $i$th input is given by
\begin{align*}
(\mathbf{O} \htimes \mathbf{P})(n) \otimes (\mathbf{O} \htimes \mathbf{P})(m) 
&\cong \mathbf{O}(n) \otimes \mathbf{O}(m) \otimes \mathbf{P}(n) \otimes \mathbf{P}(m)\\
&\quad  \quad \quad\quad\quad\quad\quad \downarrow\\
(\mathbf{O} \htimes \mathbf{P})(n+m-1)&= \mathbf{O}(n+m-1) \otimes \mathbf{P}(n+m-1), \\
\end{align*}
where the vertical map is $\circ_i^{\mathbf{O}} \otimes \circ_i^{\mathbf{P}}$.
\end{defn}

Note that given an $\mathbf{O}$-algebra $A$ and a $\mathbf{P}$-algebra $B$, 
the tensor product $A \otimes B$ possesses a natural structure of an $\mathbf{O} \htimes \mathbf{P}$-algebra.

\begin{defn}[\cite{LV}*{\S 5.3.5}]
A \emph{Hopf operad} is an operad  $\mathbf{O}$ that is a counital coassociative coalgebra in operads, with respect to $\htimes$. 
Equivalently, it is an operad in the symmetric monoidal category of counital coassociative coalgebras.
\end{defn}

By the preceding remark, we see that for a Hopf operad $\mathbf{O}$, 
the category of $\mathbf{O}$-algebras possesses a natural monoidal structure.
When the Hopf operad is cocommutative --- as in our examples --- 
$\mathbf{O}$-algebras form a symmetric monoidal category.

We also need to discuss \emph{shifts} of operations, particularly
shifted Lie brackets, for which we follow the treatment of \cite{LV}
(notably Section 7.2). In the setting of cochain complexes, it is
convenient to view shifting a complex as tensoring with the complex
$k[1]$, the one-dimensional vector space placed in degree
$-1$. Similarly, shifting an operad amounts to tensoring with a
distinguished, simple operad, namely the the endomorphism operad
$\txt{End}_{\txt{Op}}(k[1])$ of $k[1]$.  The $\Sigma_n$-module of
$n$-ary operations $\txt{End}_{\txt{Op}}(k[1])(n)$ is the sign
representation placed in degree $1-n$.

\begin{defn}
For $\mathbf{O}$ an operad, its \emph{operadic suspension} is $\txt{End}_{\txt{Op}}(k[1]) \htimes \mathbf{O}$.
\end{defn}

Note that for any cochain complex $V$, its suspension $k[1] \otimes V$ is an algebra over $\txt{End}_{\txt{Op}}(k[1])$.
Hence, an $\txt{End}_{\txt{Op}}(k[1]) \htimes \mathbf{O}$-algebra structure on $V$ is equivalent to an $\mathbf{O}$-algebra
structure on $k[-1] \otimes V$.

\begin{defn}
  Let $\Lie_n$ denote the $n$-shifted Lie operad $\txt{End}_{\txt{Op}}(k[n]) \otimes_H
  \Lie$.
\end{defn}

\begin{remark}\label{rmk:shiftedliesign}
  Giving a $\Lie_{n}$-algebra structure on $V$ is by construction
  equivalent to giving a Lie algebra structure on $V[-n]$. An
  $n$-shifted Lie algebra therefore has a shifted Lie bracket
  \[ \Lambda^{2}(V[-n]) \to V[-n],\]
  or 
  \[ \Lambda^{2}_{[-n]}(V) \to V[n],\]
  where we write $\Lambda^{2}_{[-n]}(V) :=
  \Lambda^{2}(V[-n])[2n]$. Since
\[ \Lambda^{2}_{[-n]}V \cong
\begin{cases}
  \Lambda^{2}V, & \txt{$n$ even},\\
  \Sym^{2}V, & \txt{$n$ odd},
\end{cases}
\]
we see that a shifted Lie algebra has a skew-symmetric shifted pairing
for $n$ even, but a \emph{symmetric} shifted pairing for $n$ odd.
\end{remark}

As we will primarily be interested in $\Lie_1$, we will call an
algebra over $\Lie_1$ a \emph{shifted} Lie algebra, only mentioning
the level of shifting when it is not 1. Note in particular that the
chain complex $\Lie_1(2)$ of binary operations is the trivial
$\Sigma_2$-representation placed in degree~1.

\begin{defn}
  The operad $\PZ$ is generated by two binary operations: $\bullet$ in
  degree 0 called ``multiplication'' and $\{\}$ in degree 1 called
  ``bracket.'' The operation $\bullet$ satisfies the relations for a
  commutative algebra, and the operation $\{\}$ satisfies the
  relations for a shifted Lie algebra. The remaining ternary relation is
  that the bracket acts as a biderivation for multiplication. 
  (The operad $\PZ$ is also known as the Poisson$_0$ or Gerstenhaber or -1-braid
  operad. For a description of the operad using generators and relations, 
  see Section 13.3.4 of \cite{LV} .) 
\end{defn}

Note that each space of $n$-ary operations $\PZ(n)$ has zero
differential. As remarked in Section 13.3.4 of \cite{LV}, $\PZ$ is an
``extension'' of the shifted Lie operad $\Lie_1$ by the commutative
operad $\Comm$, and hence there are canonical operad maps $\Comm \to
\PZ \to \Lie_1$. There is also a natural operad map $\Lie_1 \to \PZ$.

\begin{defn}
The operad $\mathrm{E}_{0}$ is the operad with just a single nullary
operation. Its algebras in a symmetric monoidal category
$\mathbf{C}$ are therefore just objects of $\mathbf{C}$ equipped with
a map from the monoidal unit. 
\end{defn}

We construct now an operad quasi-isomorphic to $\mathrm{E}_{0}$ as a variant of the operad $\PZ$:

\begin{defn}
  The operad $\teo$ is a modification of $\PZ$ by changing the
  differentials. Let the binary operations $\teo(2)$ be $\PZ(2)$ with
  differential $\d(\bullet) = \{\}$. Let $\teo(n)$ denote $\PZ(n)$
  equipped with the differential induced by the differential on binary
  operations.
\end{defn}

By construction, the cohomology operad $H^* \teo$ has trivial $n$-ary operations for $n > 1$. 
As $\teo$ is cochain homotopic to $H^* \teo$, this operad provides a model for the $\rE_0$-operad.

Note that there is a map of operads $\Lie_1 \to \teo$, induced by
the map $\Lie_1 \to \PZ$. In contrast, 
the map $\Comm \to \PZ$ does not lift to a map $\Comm \to \teo$ as 
such a map would not respect the differential on $\teo$.

Finally, we introduce an operad interpolating between $\PZ$ and $\teo$; it is a kind of ``Rees operad.''

\begin{defn}
The operad $\BD$ is a modification of $\PZ \otimes k[\hbar]$ by changing the differentials. Let the binary operations $\BD(2)$ be $\PZ(2) \otimes k[\hbar]$ with differential $\d(\bullet) =\hbar \{\}$. Let $\BD(n)$ denote $\PZ(n) \otimes k[\hbar]$ equipped with the differential induced by the differential on binary operations. 
\end{defn}

This definition implies that for a $\BD$-algebra $A$,
\[
\d(a \cdot b) = (\d a) \cdot b + (-1)^a a (\d b) + \hbar\{a,b\},
\]
so that $\d$ is a second-order differential operator on the underlying graded algebra $A^\sharp$.
Thus, modulo $\hbar$, the differential $\d$ is a derivation, so that the bracket measures
the failure of $A$ to be a commutative algebra in cochain complexes.

Observe that $\PZ$ is isomorphic to
\[
\BD_{\hbar = 0} := \BD \otimes_{k[\hbar]} k[\hbar]/(\hbar)
\]
and that $\teo$ is isomorphic to
\[
\BD_{\hbar = 1} := \BD \otimes_{k[\hbar]} k[\hbar]/(\hbar -1).
\]
Thus lifting a $\PZ$-algebra to a $\BD$-algebra produces an $\mathrm{E}_0$-algebra by setting $\hbar =1$ in the algebra. In this sense, a $\BD$-algebra ``quantizes'' a $\PZ$-algebra to an $\mathrm{E}_0$-algebra.

\begin{remark}
  The operad $\PZ$ is a cocommutative Hopf operad, just as the Poisson operad is.
  The coproduct $\Delta: \PZ \to \PZ \htimes \PZ$ is given by
  \[
  \Delta(\bullet) = \bullet \otimes \bullet \quad \text{and} \quad \Delta(\{\}) = \{\} \otimes \bullet + \bullet \otimes \{\},
  \]
  which is the direct analogue for the Poisson operad.
  One simply checks directly that this choice works.
  The same coproduct works for the operads $\BD$ and
  $\teo$, which are thus also Hopf. 
  \end{remark}

\subsection{Enveloping Algebras on the Model Category Level}

We now wish to analyze the relationship between algebras over the
three operads $\PZ$, $\BD$, and~$\teo$.  As we remarked above, we have
a map of $k[\hbar]$-operads $\Lie_1[\hbar] \to \BD$ that induces both
the standard inclusion $\Lie_1 \to \PZ$ when we set $\hbar = 0$ and
also a map $\Lie_1 \to \teo$ when we set $\hbar = 1$. Combining these
with the right Quillen functors induced by the algebra maps $k[\hbar]
\to k[\hbar]/(\hbar) \cong k$, $k[\hbar] \to k[\hbar]/(\hbar-1) \cong
k$, and the inclusion $k \to k[\hbar]$, we get a commutative diagram
of right Quillen functors:
\[
\begin{tikzcd}
{} & {\Lie_1}(A)  \arrow{dl}{\id} \arrow{d} & \Alg_{\teo}(A) \arrow{l} \arrow{d}\\
{\Lie_1}(A)  & {\Lie_1}(A[\hbar]) \arrow{l}&
\Alg_{\BD}(A[\hbar]) \arrow{l}\\
 & {\Lie_1(A)} \arrow{ul}{\id}\arrow{u}  & \Alg_{\PZ}(A). \arrow{l}\arrow{u}
\end{tikzcd}
\]
We will give explicit descriptions of the corresponding
\emph{left} adjoints to the horizontal morphisms, which can be thought
of as ``enveloping algebras'':
\begin{itemize}
\item the $\PZ$-enveloping functor $U_\PZ$ is left adjoint to the ``forgetful'' functor from $\Alg_\PZ(A)$ to ${\Lie_1}(A)$,
\item the $\BD$-enveloping functor $U_{\BD}$ is left adjoint to
  the ``forgetful'' functor from $\Alg_\BD(A[\hbar])$ to
  ${\Lie_1}(A[\hbar])$,
\item the $\teo$-enveloping functor $U_\teo$ is left adjoint to
  the ``forgetful'' functor from $\Alg_\teo(A)$ to
  ${\Lie_1}(A)$.
\end{itemize}
From the explicit descriptions it will be clear that these enveloping
functors interact well with the natural symmetric monoidal structures
on these categories of algebras. Note that $\PZ$, $\BD$, and $\teo$
are all Hopf operads, and so the natural monoidal structures amount, on the
level of the underlying modules, to just tensor product $\otimes$
(over $A$ for $\PZ$ and $\teo$ or $A[\hbar]$ for $\BD$). By contrast,
we equip Lie algebras with the monoidal structure given by the
Cartesian product, which is the direct sum $\oplus$ on the level of
underlying modules.

Let $\ev_{\hbar = 0} \colon \Mod(A[\hbar]) \to \Mod(A)$ be the left adjoint
functor induced by the map of algebras $A[\hbar]\to A[\hbar]/(\hbar)
\cong A$, sending $M$ to $M \otimes_{A[\hbar]} A[\hbar]/(\hbar)$. It
is naturally symmetric monoidal, intertwining $\otimes_{A[\hbar]}$ and
$\otimes_A$.  Likewise, let $\ev_{\hbar = 1} \colon \Mod(A[\hbar]) \to
\Mod(A)$ denote the symmetric monoidal functor induced by $A[\hbar]
\to A[\hbar]/(\hbar-1) \cong A$. Then replacing the right adjoints in
the diagram above with their left adjoints, we get a commutative
diagram of symmetric monoidal categories and strong symmetric monoidal
functors:
\[
\begin{tikzcd}
{} & {\Lie_1}(A)^{\oplus} \arrow{r}{U_{\teo}}   & \Alg_{\teo}(A)^{\otimes_{A}} \\
{\Lie_1}(A)^{\oplus} \arrow{ur}{\id} \arrow{r}{\blank \otimes_{A} A[\hbar]}
\arrow{dr}{\id} & {\Lie_1}(A[\hbar])^{\oplus} \arrow{u}{\ev_{\hbar = 1}}
\arrow{d}{\ev_{\hbar = 0}} \arrow{r}{U_{\txt{BD}}} &
\Alg_{\BD}(A[\hbar])^{\otimes_{A[\hbar]}} \arrow{u}{\ev_{\hbar = 1}}
\arrow{d}{\ev_{\hbar = 0}} \\
 & {\Lie_1(A)}^{\oplus} \arrow{r}{U_{\PZ}}  & \Alg_{\PZ}(A)^{\otimes_{A}}. 
\end{tikzcd}
\]

\begin{defn}
Let $\dq: \Alg_\BD(A[\hbar]) \to \Alg_\PZ(A)$ denote the {\em
  dequantization} functor sending $R$ to $R \otimes_{A[\hbar]}
A[\hbar]/(\hbar)$. Thus, given a $\PZ$-algebra $R^{cl}$,  a
\emph{BD-quantization} of $R^{cl}$ is any $R \in \Alg_\BD(A[\hbar])$
such that $R^{cl} \simeq \dq(R)$.
\end{defn}

In this terminology, we have shown that $U_{\BD}(\g \otimes_A
A[\hbar])$ is a functorial BD quantization of $U_{\PZ}(\g)$ for any
shifted Lie algebra $\g$ in $\Mod(A)$.

\begin{remark}
  In the setting of deformation quantization, people require that a
  quantization is flat over $\hbar$ or topologically free. Since the
  functors involved are left Quillen, our construction always produces
  a module that is nicely behaved with respect to $\hbar$ provided the
  input is cofibrant.
\end{remark}

The $\PZ$-enveloping functor is explicitly provided by the following
construction, which should seem obvious: if we have a shifted Lie
algebra and we want a $\PZ$-algebra, all we need to do is freely
construct the commutative algebra structure.

\begin{lemma}
For a $\Lie_1$-algebra $\g$ in $\Mod(A)$, the $\PZ$-enveloping algebra 
$U_\PZ(\g)$ is $\Sym_A(\g)$ with the commutative multiplication of the symmetric algebra 
and with the bracket
\[
\{x,y\} = [x,y]
\]
where $x,y \in \g$. Thus $U_\PZ$ is a strong symmetric monoidal functor:
\[
U_\PZ( \g \oplus \g') \cong U_\PZ(\g) \otimes_A U_\PZ (\g')
\]
for any shifted Lie algebras $\g$ and $\g'$.
\end{lemma}

\begin{proof}
  Let $\g$ be a shifted Lie algebra, and let us write
  $U_{\PZ}(\mathfrak{g})$ for the explicit $\PZ$-algebra above; we
  will then show that this gives a left adjoint to the forgetful
  functor. Observe that the inclusion $\g \hookrightarrow \Sym_{A}(\g)$ of the
  commutative algebra generators is a map of Lie algebras, if we equip
  $\Sym_{A}(\g)$ with the bracket that defines $U_{\PZ}(\g)$. We want
  to show that composing with this map induces for every $\PZ$-algebra
  $R$ an isomorphism
  \[ \Hom_{\Alg_{\PZ}(A)}(U_{\PZ}(\g), R) \to
  \Hom_{{\Lie_1}(A)}(\g, R)\]
  (where we have not explicitly denoted the forgetful functor to Lie algebras).

  To see this, observe that a Lie algebra map
  $\g \to R$ induces a unique map of commutative algebras $\Sym_{A}(\g) \to
  R$, and this respects the Lie bracket giving $U_{\PZ}(\g)$ its
  Poisson structure, \ie{} it is a map of $\PZ$-algebras. By
  inspection, this construction provides the desired inverse.

  The fact that the functor is strong symmetric monoidal is then an
  immediate consequence of the fact that $\Sym$ is.
\end{proof}

By a completely parallel argument, we obtain an analogous description
of the $\teo$-enveloping functor, except that the construction of the enveloping
algebra looks slightly more complicated than in the $\PZ$ case, since we
need to describe the differential explicitly. Recall that for a commutative algebra $A$, we use 
$A^\sharp$ to denote the underlying commutative graded algebra.

\begin{lemma}
For a $\Lie_1$-algebra $\g$ in $\Mod(A)$, the $\teo$-enveloping
  algebra $U_\teo(\g)$ has underlying $A^\sharp$-module
  $\Sym_{A^\sharp}(\g)$ with the commutative multiplication of
  the symmetric algebra, with the bracket
  \[
  \{x,y\} = [x,y]
  \]
  where $x,y \in \g$, and with differential $\d_{\teo}$ determined by
  \[
  \d_{\teo}(x)= \d_\g x
  \]
  for $x \in \g$ and
  \[
  \d_{\teo}(x \cdot y) = (\d_\g x) \cdot y + (-1)^{x} x \cdot (\d_\g y) +  \{x,y\}
  \]
  for $x, y \in \g$. Thus $U_\teo$ is a strong symmetric monoidal functor:
  \[
  U_\teo( \g \oplus \g') \cong U_\teo(\g) \otimes_{A} U_\teo(\g')
  \]
  for any shifted Lie algebras $\g$ and $\g'$.\qed
\end{lemma}

The fact that an $\teo$-algebra satisfies
\[
\d(a \cdot b) = (\d a) \cdot b + (-1)^a a \cdot (\d b) + \{a,b\}
\]
for any elements $a$ and $b$ means that we can inductively define the differential on higher symmetric
powers in $U_\teo(\g)$, as we have specified it on its $\Sym^{\leq 2}$ summand.

The situation with $\BD$ is parallel, after adjoining $\hbar$ everywhere.

\begin{lemma}
  For a $\Lie_1$-algebra $\g$ in $\Mod(A[\hbar])$, the $\BD$-enveloping
  algebra $U_\BD(\g)$ has underlying $A^\sharp[\hbar]$-module
  $\Sym_{A^\sharp[\hbar]}(\g)$ with the commutative multiplication of
  the symmetric algebra, with the bracket
  \[
  \{x,y\} = [x,y]
  \]
  where $x,y \in \g$, and with differential $\d_{\BD}$ determined by
  \[
  \d_{\BD}(x)= \d_\g x
  \]
  for $x \in \g$ and
  \[
  \d_{\BD}(x \cdot y) = (\d_\g x) \cdot y + (-1)^{x} x \cdot (\d_\g y) + \hbar \{x,y\}
  \]
  for $x, y \in \g$. Thus $U_\BD$ is a strong symmetric monoidal functor:
  \[
  U_\BD( \g \oplus \g') \cong U_\BD(\g) \otimes_{A[\hbar]} U_\BD(\g')
  \]
  for any shifted Lie algebras $\g$ and $\g'$.\qed
\end{lemma}

\subsection{Relationship with Lie Algebra Homology}\label{sec:CE}
The enveloping algebra constructions described above may seem
reminiscent of the Chevalley-Eilenberg chains of a Lie algebra, since
the differential is determined by the Lie bracket in a similar way.
We now pin down a precise relationship.

Let $\txt{Cocomm}(A)$ denote the category of cocommutative coalgebras in $\Mod(A)$.
Let $\Sym^c_A(V)$ denote the symmetric coalgebra on the $A$-module $V$, 
whose underlying $A$-module is $\bigoplus_{n \geq 0} \Sym^n_A(V)$ and
whose coproduct satisfies
\[ \Delta(x) = 1 \otimes x + x \otimes 1 \]
for every $x \in \Sym^1_A(V)$.

\begin{defn}\label{CL}
  Let $\CL: {\Lie_1}(A) \to \txt{Cocomm}(A)$ denote the functor sending
  $\g$ to the  cocommutative coalgebra $\Sym^c_{A^\sharp}(\g^\sharp)$ 
  over $A^\sharp$ equipped with the differential $\d_{\CL}$, 
  which is the degree 1 coderivation such that
  for $x \in \Sym^1_{A^\sharp}(\g^\sharp)$,  $$\d_{\CL}(x) =\d_\g(x)$$ 
  and for $x y \in \Sym^2_{A^\sharp}(\g^\sharp)$, 
  \[
  \d_{\CL}(xy) = (\d_\g x)y + (-1)^x x(\d_\g y) + [x,y].
  \]
  (In other words, this functor agrees with the Chevalley-Eilenberg
  chains functor after shifting $\g$ to an unshifted Lie algebra $\g[-1]$.)
\end{defn}

\begin{propn}\label{enviscl}
  For a $\Lie_1$-algebra $\g$ in $\Mod(A)$, the underlying cochain
  complex of the $\teo$-enveloping algebra $U_\teo(\g)$ is naturally
  isomorphic to the underlying cochain complex of~$\CL(\g)$.
\end{propn}

In other words, if $\theta$ denotes the forgetful functor from $\Alg_\teo(A)$
to $\Mod(A)$, then there is a natural isomorphism $\theta \circ \CL
\Rightarrow \theta \circ U_\teo$.

\begin{remark}
  This relationship should not seem implausible. Consider the
  underived setting of Lie algebras in vector spaces.  The inclusion
  of $\txt{Vect}$ into $\Lie$ as \emph{abelian} Lie algebras is right
  adjoint to the functor $\g \mapsto \g/[\g,\g]$ that ``abelianizes''
  a Lie algebra (or takes its \emph{coinvariants}). Hence the functor
  $\g \mapsto \txt{C}^\Lie_*(\g,\g)$ (whose cohomology is the Lie algebra
  cohomology groups $H^{\Lie}_{*}(\g, \g)$) should provide a model for the
  derived left adjoint of the abelian Lie algebra functor. 
  Now let us turn to our situation of shifted Lie algebras.  An
  $\rE_0$-algebra is simply a ``pointed'' module $A \to M$, so we see
  that the functor $\g \mapsto A \oplus \CL(\g,\g)$ --- where the
  first summand is the ``pointing'' --- provides a derived left
  adjoint to the functor $(A \to M) \mapsto M/A$, with $M/A$ an
  abelian shifted Lie algebra.  But the composite $\theta \circ
  \CL(\g)$ is isomorphic to $A \oplus \CL(\g,\g)$.  In short, $\theta
  \circ \CL$ should be a derived left adjoint to the ``forgetful''
  functor (\ie{} inclusion functor) from $\Alg_{\rE_0}(A)$ to
  ${\Lie_1}(A)$.
\end{remark}

\begin{remark}
  The result also fits nicely with the perspective of derived
  deformation theory: if we view a differential graded Lie algebra
  $\g$ as presenting a formal moduli space, then $\CL(\g)$ describes
  the coalgebra of \emph{distributions} on this space.  As
  distributions are a natural home for ``things that integrate,'' it
  is not surprising that this derived version exhibits the formal,
  algebraic properties axiomatized by physicists in $\BD$-algebras
  when they sought to formalize properties of the putative path
  integral.
\end{remark}

\begin{proof}[Proof of Proposition~\ref{enviscl}]
Both $U_{\teo}$ and $\CL$ assign to $\g$ the same underlying $A^\sharp$-module
$\Sym_{A^\sharp}(\g^\sharp)$. Moreover, the differentials on both modules
respect the filtration by symmetric powers:
\[
\d_{\teo}(\Sym^{\leq n}_{A^\sharp}(\g^\sharp)) \subset \Sym^{\leq n}_{A^\sharp}(\g^\sharp)
\]
and
\[
\d_{\CL}(\Sym^{\leq n}_{A^\sharp}(\g^\sharp)) \subset \Sym^{\leq n}_{A^\sharp}(\g^\sharp).
\]
By definition, the differentials agree on $\Sym^{\leq 2}_{A^\sharp}(\g^\sharp)$. The key difference is that
\begin{itemize}
\item $\d_{\teo}$ is extended (uniquely) to higher symmetric powers as a differential operator on the symmetric algebra, whereas
\item $\d_{\CL}$ is extended (uniquely) to higher symmetric powers as a coderivation on the symmetric coalgebra.
\end{itemize}
Hence we must show these conditions coincide, which follows
immediately from Lemma \ref{lemma: coder}. (In fact, the differential
of the BD-enveloping algebra is a linear-coefficient second-order
differential operator, which matches the fact that the
Chevalley--Eilenberg differential arises from a Lie bracket.)
\end{proof}

Let $R$ be a graded commutative algebra, and let $V$ be in $\Mod(R)$.
The symmetric coalgebra $\Sym^c_R(V)$ and the symmetric algebra $\Sym_R(V)$
are manifestly isomorphic as underlying $R$-modules.

\begin{lemma}\label{lemma: coder}
Under this isomorphism of $R$-modules, there is a bijection between coderivations on the symmetric coalgebra $\Sym^c_R(V)$ and linear-coefficient differential operators on the symmetric algebra~$\Sym_R(V)$.
That is, every coderivation $\delta$ of $\Sym^c_R(V)$ determines a linear-coefficient differential operator on $\Sym_R(V)$, and {\em vice versa},
which can be identified by looking at the underlying $R$-linear endomorphism.
\end{lemma}

\begin{remark}
This lemma is the coalgebraic twin to a familiar fact about the symmetric algebra:
there is a bijection between derivations and first-order differential operators on $\Sym_R(V)$.
In fact, linear duality recovers this lemma from that fact if $V$ is finitely-generated and projective as an $R$-module.
\end{remark}

\begin{proof}
Every coderivation $\delta$ is determined by its post-composition with projection onto the cogenerators $\Sym^1_R(V)$ of $\Sym^c_R(V)$,
just as a derivation is determined by its behavior on generators of $\Sym_R(V)$.
Likewise, every linear-coefficient differential operator is determined by the same information.
Hence, to prove the lemma, we will demonstrate a class of coderivations that manifestly correspond to linear-coefficient differential operators and then we will observe that every coderivation is a linear combination of elements of this class.

Consider the multiplication map $m_x: p \mapsto xp$ given by multiplying in $\Sym_R(V)$ by a linear element $x \in \Sym^1_R(V)$.
It determines a coderivation: 
\[
\Delta(m_x(p)) = \Delta(x)\Delta(p) = (x \otimes 1 + 1 \otimes x) \Delta(p) = (m_x \otimes \id + \id \otimes m_x)(\Delta(p)).
\]
(Note that this operation is the linear dual to~$\partial/\partial x$.)

Consider now a constant-coefficient derivation $\partial$ of the
form $f \mapsto \iota_\lambda f$, where $\lambda \in \Hom_R(V,R)$ and
$\iota_\lambda$ denotes contraction with $\lambda$. 
Thus $\partial$ is the derivation on $\Sym_R(V)$ obtained by extending $\lambda$ from
$\Sym^1_R(V) \cong V$ by the Leibniz rule.
From the perspective of $\Sym^c_R(V)$, this map $\partial$ is a comodule map, as we show by direct computation.
Let $x_1 \cdots x_n$ be a pure product in $\Sym^n_R(V)$, and compute
\begin{align*}
\Delta(\partial(x_1 \cdots x_n)) &= \Delta \left( \sum_{j = 1}^n (-1)^{|\partial|(|x_1| + \cdots |x_{j-1}|)} x_1 \cdots (\partial x_j) \cdots x_n \right) \\
&= \sum_{j = 1}^n (-1)^{|\partial|(|x_1| + \cdots |x_{j-1}|)} \prod_{i = 1}^{j-1} \Delta(x_i) \cdot (\partial x_j)\cdot \prod_{i = j+1}^{n}\Delta(x_i)
\end{align*}
and then compute
\begin{align*}
\id \otimes \partial(\Delta (x_1 \cdots x_n)) &= \id \otimes \partial \left( \prod_{i =1}^n \Delta(x_i)\right)\\
&= \sum_{j = 1}^n (-1)^{|\partial|(|x_1| + \cdots |x_{j-1}|)} \prod_{i = 1}^{j-1} \Delta(x_i) \cdot (0 + 1 \otimes \partial x_j)\cdot \prod_{i = j+1}^{n} \Delta(x_i).
\end{align*}
Comodule maps are closed under composition, so any constant-coefficient differential operator $D = \partial_1 \cdots \partial_n$ is also a comodule map.
(Note that this operation is linear dual to multiplication by a monomial.)

The composite of a coderivation following a comodule map is a coderivation.
Thus the composition $m_x D$ is a coderivation.
By construction, we have seen that for this coderivation of $\Sym^c_R(V)$, 
the underlying linear endomorphism can be read as a differential operator on $\Sym_R(V)$ with linear coefficients.

It is straightforward to see that linear combinations of such coderivations span all coderivations.
Given a coderivation, one reconstructs the expression as a differential operator by postcomposing the coderivation with projection onto cogenerators.
\end{proof}

\begin{remark}
As noted in \cite{BashkirovVoronov}, this lemma implies that the Chevalley-Eilenberg chains of an $\rm{L}_\infty$-algebra $\g$ is the specialization to $\hbar = 1$ of a kind of $\BD_\infty$-algebra $U_{\BD_\infty}(\g[1])$. Here one weights the $k$th Taylor coefficient of the Chevalley-Eilenberg differential by~$\hbar^{k-1}$.
\end{remark}

\subsection{Enveloping Algebras on the $\infty$-Category Level}\label{subsec:icatenv}
As a special case of Lemma~\ref{lem:AlgOftr}, we know that the
enveloping algebra functors described above give
functors of \icats{}, compatible with base change in the commutative
algebra variable. What is a bit less straightforward is showing that
these functors are symmetric monoidal at the \icat{} level. The
issue is that our model categories are not \emph{monoidal} model
categories: in particular, the tensor products do not preserve
cofibrant objects, so our simplicial categories will not be symmetric
monoidal. We therefore have to do a bit more work to see we have
symmetric monoidal structures on the \icats{} at all; we will proceed
analogously to the proof of \cite{HA}*{Proposition 4.1.3.10}: we
enhance our simplicial category to a simplicial operad
and check that the associated \iopd{} is actually a symmetric monoidal
\icat{}.  Once this is done, it is straightforward to see that our
enveloping functors give maps between these simplicial operads,
(pseudofunctorially) compatible with base change, and
these induce symmetric monoidal functors on the \icat{} level.

We focus on the case of Lie algebras; the same idea works for the
other operads. 
\begin{defn}
  We define a simplicial (coloured) operad structure on the simplicial
  category ${\sLie}(A)^{\op}$ by defining the multimorphism spaces
  as
  \[ \sHom((X_{1},\ldots,X_{k}), Y) :=
  \sHom_{{\sLie}(A)^{\op}}(X_{1} \oplus \cdots \oplus X_{k}, Y) =
  \sHom_{{\sLie}(A)}(Y, X_{1} \oplus \cdots \oplus X_{k}),\] with
  composition induced from that in ${\sLie}(A)$. This is compatible
  with the simplicial enrichment, given by tensoring with
  $\Omega(\Delta^{\bullet})$, since tensoring commutes with direct sums. 
\end{defn}

To get from this simplicial operad to an \iopd{} we need to pass
through its simplicial \emph{category of operators}. Recall that any
simplicial operad $\mathbf{O}$ has a simplicial category of operators
$\mathbf{O}^{\otimes}$. This has objects pairs $(\angled{n}, (X_{1},
\ldots, X_{n}))$, where $\angled{n}$ is an object of $\bbGamma^{\op}$
--- the category of finite pointed sets --- and the $X_{i}$ are
objects of $\mathbf{O}$. A morphism $(\angled{n},
(X_{1},\ldots,X_{n})) \to (\angled{m}, (Y_{1},\ldots,Y_{m}))$ is given
by a morphism $\phi \colon \angled{n} \to \angled{m}$ in
$\bbGamma^{\op}$ and for each $i \in \angled{m}$ a multimorphism
$(X_{j})_{j \in \phi^{-1}(i)} \to Y_{i}$ in $\mathbf{O}$. If
$\mathbf{O}$ is a fibrant simplicial operad (meaning each simplicial
set of multimorphisms $\mathbf{O}((X_{1},\ldots,X_{k}),Y)$ is a Kan
complex), then the coherent nerve $\mathrm{N}\mathbf{O}^{\otimes} \to
\bbGamma^{\op}$ of the obvious projection to $\bbGamma^{\op}$ is an
\iopd{}, in the sense of \cite{HA}*{\S 2.1.1}, by
\cite{HA}*{Proposition 2.1.1.27}.

Let ${\sLie}(A)^{\op,\oplus}$ denote the simplicial category of
operators of the simplicial operad of cofibrant Lie algebras we just
defined. The simplicial sets $\sHom((X_{1},\ldots,X_{k}), Y)$ are all
Kan complexes, since $Y$ is cofibrant, so the nerve
$\mathrm{N}({\sLie}(A)^{\op,\oplus}) \to \bbGamma^{\op}$ is an \iopd{}.

Recall that a symmetric monoidal \icat{} can be defined as an \iopd{}
$\mathcal{O}$ such that the projection $\mathcal{O} \to
\bbGamma^{\op}$ is a coCartesian fibration. This holds for our \iopd{} $\mathrm{N}({\sLie}(A)^{\op,\oplus})$:
\begin{propn}
  The projection $\pi \colon \mathrm{N}({\sLie}(A)^{\op,\oplus}) \to
  \bbGamma^{\op}$ is a coCartesian fibration. That is,
  $\pi$ is a symmetric monoidal \icat{}.
\end{propn}

\begin{proof}
  It suffices to show that for every object $(X_{1},\ldots,X_{n})$ of
  ${\sLie}(A)^{\op,\oplus}$ and every map $\phi \colon \angled{n}
  \to \angled{m}$ in $\bbGamma^{\op}$, there exists a morphism
  $(X_{1},\ldots,X_{n}) \to (X'_{1},\ldots,X'_{m})$ over $\phi$ such
  that for any $(Y_{1},\ldots,Y_{k})$, the square
  \nolabelcsquare{\Map_{\mathrm{N}({\sLie}(A)^{\op,\oplus})}((X'_{1},\ldots,X'_{m}),
    (Y_{1},\ldots,Y_{k}))}{\Map_{\mathrm{N}({\sLie}(A)^{\op,\oplus})}((X_{1},\ldots,X_{n}),
    (Y_{1},\ldots,Y_{k}))}{\Hom_{\bbGamma^{\op}}(\angled{m},
    \angled{k})}{\Hom_{\bbGamma^{\op}}(\angled{n}, \angled{k})} is homotopy
  Cartesian. Choose a weak equivalence $X'_{i} \to \bigoplus_{j,
    \phi(j) = i} X_{j}$ in $\Lie_{n}(A)$ with $X'_{i}$ cofibrant.
  We claim the resulting map $(X_{1},\ldots,X_{n}) \to
  (X'_{1},\ldots,X'_{m})$ in ${\sLie}(A)^{\op,\oplus}$ has this property. 
  To see this it suffices to show that we have a weak
  equivalence on fibres over each $\psi \colon \angled{m} \to
  \angled{k}$, since the objects in the bottom row are discrete. 
  These fibres decompose as products, so it is enough to show that
  \[ \Map_{\mathrm{N}({\sLie}(A)^{\op,\oplus})}((X'_{i})_{\psi(i) = j}, Y_{j}) \to \Map_{\mathrm{N}({\sLie}(A)^{\op,\oplus})}((X_{k})_{\psi\phi(k)
    = j}, Y_{j})\]
  is a weak equivalence for all $j$. We can identify this map with
  \[ \sHom_{{\sLie}(A)}(Y_{j}, \bigoplus_{i} X'_{i}) \to \sHom_{{\sLie}(A)}(Y_{j},
  \bigoplus_{j} X_{j}).\] Since $Y_{j}$ is cofibrant, to see that this map
  is a weak equivalence of simplicial sets, it suffices to show that
  $\bigoplus_{i} X'_{i} \to \bigoplus_{j} X_{j}$ is a weak equivalence
  in Lie algebras. But this map is the product over $i$ of the maps
  $X'_{i} \to \bigoplus_{j, \phi(j) =i} X_{j}$, which are weak
  equivalences, and since weak equivalences are detected in $\Mod(k)$,
  it is clear that direct sums of weak equivalences are again weak
  equivalences.
\end{proof}

We thus have a symmetric monoidal \icat{} with underlying \icat{}
$\LLie_{n}(A)^{\op}$. This induces a symmetric monoidal structure
on $\LLie_{n}(A)$, \ie{} we have:
\begin{cor}
  The Cartesian product $\oplus$ of Lie algebras induces a symmetric
  monoidal structure on the \icat{} $\LLie_{n}(A)$.
\end{cor}

\begin{propn}\ 
  \begin{enumerate}[(i)]
  \item The tensor product of $A$-modules induces symmetric monoidal
    structures on the \icats{} $\AAlg_{\teo}(A)$ and $\AAlg_{\PZ}(A)$,
    and the tensor product of $A[\hbar]$-modules induces a symmetric
    monoidal structure on the \icat{} $\AAlg_{\BD}(A[\hbar])$.
  \item The enveloping algebra functors $U_{\BD}$, $U_{\teo}$ and
    $U_{\PZ}$, as well as the functors $\txt{ev}_{\hbar=0}$ and
    $\txt{ev}_{\hbar=1}$ induce symmetric monoidal functors of
    \icats{}.
  \end{enumerate}
\end{propn}
\begin{proof}
  (i) follows from the same argument as for Lie algebras. We only need
  to check that if $X$ is a cofibrant algebra then $X \otimes_{A}
  \blank$ preserves quasi-isomorphisms. (In the case of
  BD-algebras, we use $X \otimes_{A[\hbar]} \blank$ instead.) This claim 
  follows from Lemma~\ref{lem:cofmodtensor}, since by
  Proposition~\ref{propn:opdch}(ii), the underlying $A$-module of a
  cofibrant algebra is cofibrant.

  For (ii), if $U$ is either $U_{\teo}$ or $U_{\PZ}$, we must show
  that if $L$ and $L'$ are cofibrant Lie algebras and
  $L'' \to L \oplus L'$ is a cofibrant replacement, then
  $U(L'') \to U(L \oplus L')$ is a weak equivalence. It suffices to
  check weak equivalences at the level of the underlying modules, and
  there we have a natural filtration:
  $U(\g) \cong \colim_n U^{\leq n}(\g)$ for any Lie algebra $\g$, with
  $U^{\leq n}(\g)$ being the subcomplex whose underlying graded module
  is $\Sym^{\leq n}(\g)$.  It is manifest that
  $U^{\leq 0}(L'') = A = U^{\leq 0}(L \oplus L')$.  Now consider the
  map of cofiber sequences
  \[\begin{tikzcd} 
    U^{\leq n-1}(L'') \arrow{r} \arrow{d} & U^{\leq n}(L'') \arrow{r} \arrow{d} & \Sym^n_A(L'') \arrow{d} \\
    U^{\leq n-1}(L \oplus L') \arrow{r} & U^{\leq n}(L \oplus L') \arrow{r} & \Sym^n_A(L \oplus L')
  \end{tikzcd}\] 
  By Lemma~\ref{lem:Symnqiso} the functor $\Sym^n_A$ preserves
  quasi-isomorphisms for cofibrant modules. Since the underlying
  $A$-module of a cofibrant Lie algebra is cofibrant by
  Proposition~\ref{propn:opdch}(ii), the rightmost vertical map in the
  diagram is a quasi-isomorphism. Inducting on $n$, it follows that
  $U^{\leq n}(L'') \to U^{\leq n}(L \oplus L')$ is a quasi-isomorphism
  for all $n$. As quasi-isomorphisms commute with filtered colimits,
  we conclude that $U(L'') \to U(L \oplus L')$ is a
  quasi-isomorphism. The proof for $U_{\BD}$ is the same, except with
  some $\hbar$'s.
   
  For the functors induced by the two maps $A[\hbar] \to A$, it again
  suffices to show that we get a quasi-isomorphism of underlying
  modules, which is true since $A \otimes_{A[\hbar]} \blank$ is a left
  Quillen functor, the underlying module of a cofibrant algebra is
  cofibrant, and the tensor product of cofibrant modules is again
  cofibrant.
\end{proof}

Taking the base change (pseudo)functors into account, we have:
\begin{lemma}\label{lem:Liesymmmonftr}
  There are functors ${\LLie_{n}}(\blank)^{\oplus}$,
  $\AAlg_{\txt{BD}}(k[\hbar] \otimes \blank)^{\otimes}$,
    $\AAlg_{\teo}(\blank)^{\otimes}$ and
    $\AAlg_{\PZ}(\blank)^{\otimes}$ from $\CComm(k)$ to the \icat{}
  $\CComm(\LCatI)$ of (large) symmetric monoidal \icats{} taking $A$
  to the symmetric monoidal \icats{} constructed above. The enveloping
  algebra functors $U_{\BD}$, $U_{\teo}$ and $U_{\PZ}$, as well as the
  functors $\txt{ev}_{\hbar=0}$ and $\txt{ev}_{\hbar=1}$ induce
  natural symmetric monoidal functors between these.
\end{lemma}
\begin{proof}
  As Lemma~\ref{lem:AlgOftr}, just replacing simplicial categories
  with simplicial operads (and passing to opposite \icats{}).
\end{proof}

\subsection{Aside: An $\mathrm{E}_n$-Enveloping Algebra Functor}\label{sec:Enenv}

In this section, we will describe an ``enveloping algebra'' adjunction
\[ {\LLie_{1-n}}(A) \rightleftarrows \AAlg_{\mathrm{E}_{n}}(A),\] 
where $\mathrm{E}_{n}$ is the ``little $n$-discs'' \iopd{}. 
(This section is something of a digression from our main objective, although it is
relevant as motivation for Conjecture~\ref{conj:QuadEn}.) 
We expect that this construction agrees with the enveloping functor 
for the map of operads constructed
by Fresse~\cite{FresseKoszul} using Koszul duality as well as that
recently constructed (in greater generality) by
Knudsen~\cite{KnudsenEn} using factorization algebras. However,
although we will show that our functor satisfies some of the same
formal properties as Knudsen's, we will not attempt to compare them
here.

By our work in the preceding sections, we have a symmetric monoidal left adjoint
functor 
\[ U_{\teo}: {\LLie_{1}}(A) \to \AAlg_{\teo}(A) \simeq \AAlg_{\mathrm{E}_{0}}(A),\] 
where the second equivalence follows from the quasi-isomorphism $\teo \simeq \mathrm{E}_0$
of operads. By \cite{HA}*{Corollary 7.3.2.7}, 
the right adjoint to the $\teo$-enveloping functor is lax monoidal, and so the resulting relative adjunction over 
$\bbGamma^{\op}$ induces an adjunction
\[ \AAlg_{\mathcal{O}}({\LLie_{1}}(A)) \rightleftarrows
\AAlg_{\mathcal{O}}(\AAlg_{\teo}(A)) \]
for every $\infty$-operad $\mathcal{O}$. Taking $\mathcal{O}$ to be the \iopd{}
$\mathrm{E}_{n}$, we have
\[ \AAlg_{\mathrm{E}_{n}}(\AAlg_{\teo}(A)) \simeq
\AAlg_{\mathrm{E}_{n}}(A)\] since the Boardman-Vogt tensor $\mathrm{E}_{n} \otimes
\mathrm{E}_{0}$ is equivalent to $\mathrm{E}_{n}$; thus we get an adjunction
\[ \AAlg_{\mathrm{E}_{n}}({\LLie_{1}}(A)) \rightleftarrows
\AAlg_{\mathrm{E}_{n}}(A).\]
To get the enveloping functor we want, we combine this construction
with a result due to To\"en (though we learned the argument
from Nick Rozenblyum). Before we state it, we must recall the
\emph{bar/cobar adjunction}, as set up for \icats{} by Lurie in
\cite{HA}*{\S 5.2.2}. If $\mathcal{C}$ is a monoidal \icat{} with 
simplicial colimits and cosimplicial limits,  there is an adjunction 
\[ \txt{Bar} : \AAlg_{\txt{Ass}}^{\txt{aug}}(\mathcal{C}) \rightleftarrows
\mathcal{C}\txt{oAlg}^{\txt{coaug}}_{\txt{Ass}}(\mathcal{C}) : \txt{Cobar} \]
between augmented associative algebras and coaugmented coassociative
coalgebras. If $\mathcal{C}$ has a zero object and the monoidal
structure is the Cartesian product, then this simplifies to an
adjunction
\[ \txt{Bar} : \AAlg_{\txt{Ass}}(\mathcal{C}) \rightleftarrows \mathcal{C} :
\txt{Cobar}.\]

\begin{propn}\label{propn:barcobareq}
  Suppose $\mathcal{C}$ is a presentable stable \icat{}, $\mathcal{D}$
  is a presentable \icat{}, and $U \colon \mathcal{D} \to \mathcal{C}$
  is a functor that detects equivalences and preserves limits and
  sifted colimits. Then, regarding $\mathcal{D}$ as a monoidal \icat{}
  via the Cartesian product, the bar/cobar adjunction
  \[ \txt{Bar} : \AAlg_{\txt{Ass}}(\mathcal{D}) \rightleftarrows \mathcal{D} :
  \txt{Cobar}\]
  is an equivalence.
\end{propn}
\begin{proof}
  Since $\mathcal{C}$ is stable, the Cartesian product in
  $\mathcal{C}$ is also the coproduct, and hence
  it commutes with sifted colimits and cosifted limits in each variable. 
  As $U$ detects equivalences and preserves limits and sifted colimits, 
  we find that the Cartesian product in $\mathcal{D}$ also preserves
  sifted colimits and cosifted limits in each variable.
  Thus using \cite{HA}*{Example 5.2.2.3}, for any $X \in
  \AAlg_{\txt{Ass}}(\mathcal{D})$, 
  we can identify $U\txt{Bar}(X)$ with the suspension $\Sigma UX$.
  Dually, if $U'$ denotes the forgetful functor 
  \[ \AAlg_{\txt{Ass}}(\mathcal{D}) \to \mathcal{D} \to \mathcal{C},\] 
  then for $Y \in \mathcal{D}$, we can identify
  $U'\txt{Cobar}(Y)$ with the loop object $\Omega U(Y)$.

  To show that the bar/cobar functors are an adjoint equivalence, it
  suffices to show that the unit and counit transformations are
  natural equivalences. But since the functors $U$ and $U'$ detect
  equivalences, we are finished because suspension/loops is an adjoint
  equivalence on the stable \icat{} $\mathcal{C}$.
\end{proof}

\begin{cor}[To\"en]
  The bar functor is an equivalence
  \[\AAlg_{\txt{Ass}}(\LLie(A)) \isoto \LLie(A),\]
  given by $X \mapsto X[1]$ on underlying $A$-modules.
\end{cor}
\begin{proof}
  The forgetful functor $\LLie(A) \to \MMod(A)$ satisfies the
  assumptions of Proposition~\ref{propn:barcobareq}.
\end{proof}

\begin{remark}
  This result is also found, with essentially the same proof, as
  \cite{ToenBranes}*{Lemma 5.3}. Note that we did not use any special
  property of the Lie operad: the same argument works for associative
  algebras with respect to the Cartesian product in
  $\AAlg_{\mathbf{O}}(A)$ for any $\mathbf{O}$.
\end{remark}

Iterating this equivalence, we get:
\begin{cor}
  By $n$-fold application of the bar construction, we get an equivalence
  \[ \AAlg_{\mathrm{E}_{n}}(\LLie(A)) \isoto \LLie(A) \]
  given on underlying $A$-modules by $X \mapsto X[n]$.
\end{cor}

We can interpret this result as an equivalence
$\AAlg_{\mathrm{E}_{n}}({\LLie_{k}}(A)) \isoto {\LLie_{k-n}}(A)$ given by
the identity on underlying $A$-modules. Combining this result with our
functor \[\AAlg_{\mathrm{E}_{n}}({\LLie_{1}}(A)) \to \AAlg_{\mathrm{E}_{n}}(A)\] 
gives an ``enveloping algebra''
\[ U_{n} \colon {\LLie_{1-n}}(A) \to \AAlg_{\mathrm{E}_{n}}(A)\]
that is left adjoint to a ``forgetful functor'' $\AAlg_{\mathrm{E}_{n}}(A) \to {\LLie_{1-n}}(A)$.

\begin{remark}
  It follows from the proof that under the equivalence
  $\AAlg_{\mathrm{E}_{n}}({\LLie_{1}}(A)) \isoto {\LLie_{1-n}}(A)$, the
  forgetful functor ${\LLie_{1-n}}(A) \to \MMod(A)$ is identified with
  the forgetful functor $$\AAlg_{\mathrm{E}_{n}}({\LLie_{1}}(A)) \to
  {\LLie_{1}}(A) \to \MMod(A).$$ Thus we have a commutative diagram of
  right adjoints
  \opctriangle{\AAlg_{\mathrm{E}_{n}}(A)}{{\LLie_{1-n}}(A)}{\MMod(A),}{}{}{}
  which implies that the corresponding diagram of left adjoints also
  commutes. This observation implies that $U_{n}$ takes the free
  $\Lie_{1-n}$-algebra on an $A$-module $M$ to the free
  $\mathrm{E}_{n}$-algebra on $M$, as in \cite{KnudsenEn}*{Theorem A}.
\end{remark}

\section{The Heisenberg Functor}\label{sec:Heis}

The usual Heisenberg Lie algebra of a symplectic vector space
$(V,\omega: \Lambda^2 V \to k)$ is the vector
space $V \oplus k\c$ equipped with the Lie bracket 
\[
[x + \alpha \c, y + \beta \c] = \omega(x,y) \c.
\]
In other words, it is a central extension of the abelian Lie algebra
$V$ by the one-dimensional abelian Lie algebra $k \c$. Specializing
$\c$ to $i\hbar$, one recovers Heisenberg's celebrated relation $[x,p] = i\hbar$. 
Note that the pairing $\omega$ need not be non-degenerate, so
the construction works even for ``presymplectic'' vector spaces.

Our goal in this section is to articulate a version of this
construction where the input is a \emph{quadratic module of degree
  $1$} --- a module $V$ over a commutative algebra $A$ equipped with a shifted
symmetric pairing $\omega: \Sym^2_A V \to A[1]$ --- and the
output is a shifted Lie algebra given by centrally extending the
abelian Lie algebra $V$ by $A \c$, with $\c$ in degree zero. 
This construction makes sense ``on the nose'' for
objects of the natural category of quadratic modules: 
given a quadratic module $(V,\omega)$, we can define 
the $1$-shifted Heisenberg Lie algebra $\txt{Heis}_{1}(V,\omega)$ 
as $V \oplus A\c$, where $\c$ has degree zero, 
with shifted Lie bracket
\[
[x + \alpha \c, y + \beta \c] = \omega(x,y) \c.
\]
If we have a map $f \colon V \to V'$ such that $f^{*}\omega' = \omega$, then
we get a Lie algebra map $\text{Heis}_{1}(f): \txt{Heis}_{1}(V,\omega) \to \txt{Heis}_{1}(V,\omega')$ by 
\[
x + \alpha \c \mapsto f(x) + \alpha \c
\]
since
\[
[f(x) + \alpha \c, f(y) + \beta \c] = \omega'(f(x),f(y)) \c = \omega(x,y) \c = [x + \alpha \c, y + \beta \c] .
\]
This definition, however, has two issues. Firstly, the Lie
algebra $\txt{Heis}_{1}(V)$ is not cofibrant, and so 
it needs to be cofibrantly replaced in order to get a homotopically
meaningful answer when we apply the enveloping
functors described above. Secondly, a more subtle
issue is that, as we will see in \S\ref{subsec:quadmod}, the obvious
way to make a simplicial category of quadratic modules does not define
the correct \icat{}. It turns out that we can fix the second issue by
taking the maps of quadratic modules to be maps that preserve the
pairings only up to a specific cochain homotopy.  Unfortunately, this notion of map does
not give maps between Heisenberg Lie algebras: if $F: (V,\omega) \to
(V',\omega')$ is a map of quadratic modules in this sense, given by a
map of $A$-modules $f: V \to V'$ and a homotopy $\eta$ between
$\omega$ and $f^*\omega$, meaning
\[
\d_A \circ \eta + \eta \circ \d_{\Lambda^2 V} = \omega - \omega' \circ (f \otimes f),
\]
then we see that
\[
[f(x) + \alpha \c, f(y) + \beta \c] = \omega'(f(x),f(y)) \c = \omega(x,y) \c - \left( \d_A(\eta(x,y)) + \eta(\d_V x, y) +(-1)^x \eta( x, \d_V y) \right).
\] 
In other words, $f$ produces a Lie algebra map only \emph{up to homotopy}.

For this reason we take a technical detour through
\emph{$L_{\infty}$-algebras}, as the formalism of $L_\infty$-algebras 
provides a convenient tool for working with Lie algebras up
to homotopy. The key advantage is that one works with 
the coalgebra of \emph{Chevalley-Eilenberg chains}
$\CL(\g)$ of a Lie algebra $\g$ --- its \emph{bar construction} $\mathbb{B}(V)$ ---
rather than directly with $\g$. In particular, maps between bar constructions 
capture the notion of ``maps of Lie algebras up
to homotopy.''

Using the flexibility of $L_\infty$-algebras, we will see in \S\ref{subsec:HeisLI}
that the corrected maps of quadratic modules induce natural maps on
the bar constructions $\mathbb{B}\txt{Heis}_{1}(V)$. We can then
apply the cobar construction $\bbOmega$ to get a
functor to shifted Lie algebras; this approach also fixes the first issue mentioned
above, since $\bbOmega \mathbb{B}\txt{Heis}_{1}(V)$ is a natural cofibrant
replacement of $\txt{Heis}_{1}(V)$. Passing to \icats{}, we produce a functor
\[
\HH\colon \QQuad_1(A) \to {\LLie_{1}}(A).
\]

\subsection{Quadratic Modules}\label{subsec:quadmod}

In this section we introduce the \icat{} of quadratic modules. This
admits a simple description: it is the pullback of \icats{}
\csquare{\QQuad_{n}(A)}{\MMod(A)_{/A[n]}}{\MMod(A)}{\MMod(A),}{}{}{}{\Lambda^{2}_{[-n]}}
where the right vertical functor is the forgetful functor that takes a
morphism to $A[n]$ to its domain, and $\Lambda^{2}_{[-n]} :=
\Lambda^{2}(\blank[-n])[2n]$.

\begin{remark}
  This sign convention is justified by
  Remark~\ref{rmk:shiftedliesign}: we want an $n$-shifted quadratic
  module $(X, \omega)$ over $A$ to determine an $n$-shifted Lie
  algebra structure on $X \oplus A$ with the Lie bracket of $x,y \in
  X$ given by $\omega(x,y)$ in $A$. Thus we want a skew-symmetric
  pairing on $X[-n]$, which corresponds to a pairing
  $\Lambda^{2}_{[-n]}X \to k[n]$; note that this pairing is symmetric for $n$
  odd and skew-symmetric for $n$ even.
\end{remark}

\begin{remark}
  The closely related situation of modules equipped with (shifted)
  \emph{symmetric} pairings has been studied by Vezzosi~\cite{VezzosiQuad}.
\end{remark}

For our purposes, it will be convenient to have a simplicial category
that models the \icat{} $\QQuad_{n}(A)$ of such quadratic modules,
so that we can give explicit constructions that play nicely with the
enveloping algebra functors.  As a first attempt, let us try to mimic
the pullback construction above in the setting of simplicial
categories.

By Corollary~\ref{cor:Chkicat}(ii) the \icat{} $\MMod(A)$ is modelled
by the usual simplicial enrichment $\sMod(A)$ of the category
$\Mod(A)^{\txt{cf}}$ of fibrant-cofibrant $A$-modules in cochain
complexes.  Similarly, the slice
\icat{} $\MMod(A)_{/A[n]}$ can be modelled by the corresponding
simplicial enrichment of the fibrant-cofibrant objects in the slice
category $\Mod(A)^{\txt{cf}}_{/A[n]}$. Na\"ively we might therefore
try to model the \icat{} $\QQuad_{n}(A)$ by the pullback of simplicial
categories
\csquare{\mathbf{C}}{\sMod(A)_{/A[n]}}{\sMod(A)}{\sMod(A),}{}{}{}{\Lambda^{2}_{[-n]}}
as we did with the \icats{}.
This pullback gives a simplicial enrichment of the obvious strict category
of quadratic modules, but it is not a homotopy pullback diagram of simplical categories: the right vertical functor is
not a fibration in the model category of simplicial categories. We
therefore need to replace it with a map that is a fibration, which we
do as follows:

\begin{defn}
  For $X \in \Mod(A)$, let $\Mod(A)'_{/X}$ be the category in which an
  object is a fibrant-cofibrant object of $\Mod(A)_{/X}$,
  namely a pair \[(C \in \Mod(A), f \colon C \to X),\] with $C$
  cofibrant in $\Mod(A)$ and $f$ a fibration, and in which a
  morphism $(C,f) \to (C',f')$ is a map $\phi \colon C \to C'$
  together with a cochain homotopy $\eta \colon C \to \Omega(\Delta^{1}) \otimes X$ from $f$
  to $f' \circ \phi$. This category has an obvious simplicial enrichment
  $\sMod(A)'_{/X}$,  defined as usual by tensoring with
  $\Omega(\Delta^{\bullet})$.
\end{defn}

Note that a morphism between such objects respects the maps down to $X$ only up to homotopy.

\begin{lemma}
  The inclusion $i \colon \sMod(A)_{/X} \to \sMod(A)'_{/X}$ is a weak
  equivalence of simplicial categories, and the projection
  $p \colon \sMod(A)'_{/X} \to \sMod(A)$ is a fibration of simplicial
  categories.
\end{lemma}

\begin{proof}
  Recall that a functor of simplicial categories is a fibration \IFF{}
  it is an isofibration on homotopy categories (\ie{} every isomorphism
  in the target can be lifted to one in the source) and it is given by
  Kan fibrations on the mapping spaces.  Let $(Y, f \colon Y \to X)$
  and $(Z, g \colon Z \to X)$ be two objects of $\sMod(A)'_{/X}$; for
  brevity we will refer to these objects as just $f$ and $g$. Then the
  simplicial set of maps between them is given by the pullback square
  \csquare{\sHom'_{A/X}(f,g)}{\sHom_{A}(Y, Z)}{\sHom_{A}(Y,
    \Omega(\Delta^{1})\otimes X)}{\sHom_{A}(Y, X) \times\sHom_{A}(Y,
    X),} {}{}{\{f\} \times (g \circ -)}{\ev_0 \times \ev_1} where the
  bottom horizontal map evaluates a map in $\Omega(\Delta^{1})\otimes
  X$ at the two endpoints of the 1-simplex.  Since
  $\Omega(\Delta^{\bullet})$ is Reedy fibrant, the bottom horizontal
  map is a Kan fibration, and hence so is the top horizontal map. To
  see that $p$ is an isofibration on homotopy categories, observe that
  since $\Mod(A)$ is a model category and $\sMod(A)$ contains only the
  fibrant-cofibrant objects, the isomorphisms in the homotopy category
  of $\sMod(A)$ are precisely the cochain homotopy equivalences. Given
  a cochain homotopy equivalence $\phi \colon Y \to Y'$ and a map $f
  \colon Y' \to X$, a trivial cochain homotopy $\eta$ from $f \circ
  \phi$ to itself gives a map $\tilde{\eta}: (Y, f \phi) \to (Y', f)$
  in $\sMod(A)'_{/X}$ over $\phi$. This map induces a simplicial
  homotopy equivalence on all mapping spaces, and thus becomes an
  isomorphism in the homotopy category.

  Since the simplicial categories $\sMod(A)_{/X}$ and $\sMod(A)'_{/X}$
  have the same objects, the functor $i$ is obviously essentially
  surjective on the homotopy categories. To see that it is a weak
  equivalence, it therefore only remains to show that for any two
  objects $f \colon Y \to X$, $g \colon Z \to X$, the map of
  simplicial sets
  \[ \sHom_{A/X}(f, g) \to \sHom'_{A/X}(f, g) \]
  is a weak equivalence. To prove this, we consider the commutative
  diagram of simplicial sets
\[  
\begin{tikzcd}
    \sHom_{A/X}(f, g) \arrow[twoheadrightarrow]{d}\arrow{r} & \sHom'_{A/X}(f, g)
    \arrow[twoheadrightarrow]{d} \arrow{rr} & & \sHom_{A}(Y, Z) \arrow[twoheadrightarrow]{d}{g \circ -} \\
    \{f\} \arrow{r}{\sim} & \sHom_{A}(Y, \Omega(\Delta^{1}) \otimes X)_{f} \arrow{r} \arrow[twoheadrightarrow]{d}{\sim}& \sHom_{A}(Y,
    \Omega(\Delta^{1}) \otimes X) \arrow{r} \arrow[twoheadrightarrow]{d}{\ev_0}& \sHom_{A}(Y, X) \\
    & \{f\} \arrow{r} & \sHom_{A}(Y, X).
  \end{tikzcd}
\]
Note that by definition the bottom square, the upper right square, and
the outer composite square in the top row are all pullback squares.  Hence
the top left square is also a pullback.  The top right vertical arrow is a Kan
fibration because $g : Z \to X$ is a fibration, and so as indicated in
the diagram, it follows that all three top vertical arrows are Kan
fibrations.  The bottom right arrow $\ev_0$ is a trivial Kan fibration
and so the bottom left arrow is a trivial Kan fibration.  By the
2-out-of-3 property, we thus deduce that the bottom map in the upper
left square is a weak equivalence.  Hence the top left horizontal map
is also a weak equivalence, as required, as simplicial sets form a
right proper model category.
\end{proof}

We then define our simplicial category of quadratic modules by:
\begin{defn}
  Let $\sQuad_{n}(A)$ be the simplicial category defined by
  the pullback square
  \csquare{\sQuad_{n}(A)}{\sMod(A)'_{/A[n]}}{\sMod(A)}{\sMod(A),}{}{}{}{\Lambda^{2}_{[-n]}}
  which is also a homotopy pullback square.
  We also write $\Quad_{n}(A)$ for the underlying category of
  $\sQuad_{n}(A)$, which sits in a pullback square
  \csquare{\Quad_{n}(A)}{\Mod(A)'_{/A[n]}}{\Mod(A)^{\txt{cf}}}{\Mod(A)^{\txt{cf}}.}{}{}{}{\Lambda^{2}_{[-n]}}  
\end{defn}

We now turn to equipping these categories with a symmetric monoidal structure. 
On the underlying modules, we simply use $\oplus$, but we need to describe how the quadratic forms are combined.
Given $(X, \omega)$ and $(X', \omega')$, let $\omega+\omega'$ on the direct sum $X \oplus X'$ be given by the composite map
\[
\Lambda^{2}_{[-n]}(X \oplus X') \xto{\pi} \Lambda^{2}_{[-n]}X \oplus \Lambda^{2}_{[-n]}X'  \xto{\omega \oplus \omega'} A[n] \oplus A[n] \xto{+} A[n],
\]
where $\pi$ is projection. We then define $(X,\omega) \oplus (X',
\omega')$ to be $(X \oplus X', \omega + \omega')$. Given two maps $(f,
\eta) \colon (X, \omega) \to (X', \omega')$ and $(g, \psi) \colon (Y, \nu) \to (Y', \nu')$, their
tensor product $(f, \eta) \oplus (g, \psi) \colon (X, \omega) \oplus
(X', \omega') \to (Y, \nu) \oplus (Y', \nu')$ is defined to be $f
\oplus g \colon X \oplus X' \to Y \oplus Y'$ together with the
homotopy
\[ \Lambda^{2}_{[-n]}(X \oplus X') \xto{\pi} \Lambda^{2}_{[-n]}X \oplus
\Lambda^{2}_{[-n]}X'  \xto{\eta \oplus \psi} \Omega(\Delta^{1}) \otimes A[n]
\oplus \Omega(\Delta^{1}) \otimes A[n] \xto{+} \Omega(\Delta^{1}) \otimes A[n].\]

\begin{propn}
  This definition extends naturally to a symmetric monoidal structure
  on the simplicial category $\sQuad_{n}(A)$. Moreover, this symmetric
  monoidal structure is pseudonatural in $A$.
\end{propn}

\begin{proof}
  Since $\sQuad_{n}(A)$ is the simplicial category associated to
  $\Quad_{n}(A \otimes \Omega(\Delta^{\bullet}))$ by the construction
  of Proposition~\ref{propn:pseudosimp}, it suffices to show that the
  categories $\Quad_{n}(A)$ are symmetric monoidal, pseudonaturally in
  $A$. It is easy to see that our definition does indeed give such a
  pseudonatural symmetric monoidal structure.
\end{proof}

\begin{cor}
The \icat{} $\QQuad_{n}(A)$ has a natural symmetric monoidal structure.
\end{cor}

\subsection{$L_\infty$-Algebras}\label{subsec:Linftyalg}

In this section we introduce a version of $L_\infty$ algebras
well-suited to our purposes. A crucial requirement is that our notion
must work over any commutative algebra $A$ over a field $k$ of
characteristic zero and must play nicely with base-change. We will not
develop a general framework, but rather proceed in a somewhat ad hoc
fashion that produces the limited results we need.

Recall that $\txt{Cocomm}(A)$ denotes the category of cocommutative
coalgebras in $\Mod(A)$.  A cocommutative coalgebra $C$ is
\emph{coaugmented} if there is a retract of coalgebras $A \xto{\eta} C
\to A$.  Its \emph{reduced coalgebra} $\overline{C}$ is the kernel of
the counit map, so that $C = A \oplus \overline{C}$, and
$\overline{C}$ inherits a coproduct $\bar{\Delta}$ by
\[ \bar{\Delta}(c) = \Delta(c) - c \otimes 1 - 1 \otimes c. \] For us,
a coaugmented cocommutative coalgebra $A \xto{\eta} C$ is
\emph{conilpotent} if for any element $c$ in the reduced coalgebra
$\overline{C}$, there is some integer $n$ such that $\bar{\Delta}^n(c)
= 0$.  The key example is the symmetric coalgebra $\Sym^c_A(V)$, whose
reduced coalgebra $\Sym^{\geq 1}_A(V)$ is manifestly conilpotent as
$\ker(\bar{\Delta}^n) = \bigoplus_{j = 1}^n \Sym^{i}_A(V)$.  We write
$\CNCoalg(A)$ for the category of conilpotent coaugmented
cocommutative coalgebras in $\Mod(A)$, where we require maps to
preserve the coaugmentations.
 
A morphism of commutative algebras $f \colon A \to B$ induces a base
change functor \[f_{!} \colon \CNCoalg(A) \to \CNCoalg(B)\] by
tensoring with $B$ over $A$. Hence we can define a simplicial category
$\sCNCoalg(A)$ by taking the simplicial set of morphisms to be
\[\sHom_{\CNCoalg(A)}(C, C')_{k}= 
\Hom_{\CNCoalg(\Omega(\Delta^{k}) \otimes A)}(\Omega(\Delta^{k}) \otimes C, \Omega(\Delta^{k}) \otimes C'),\]
in parallel with our construction of simplicial categories of algebras over operads.

\begin{defn}
  For any commutative algebra $A$, there is a cobar-bar adjunction
  \[ \bbOmega : \CNCoalg(A) \rightleftarrows \Lie_1(A) : \mathbb{B}.\]
  The \emph{bar construction} $\mathbb{B}$ is given by the  functor $\CL$ of Definition \ref{CL}.
  (Recall this is the usual Chevalley-Eilenberg chains, after shifting.)
  The \emph{cobar construction} $\bbOmega$ assigns to $C \in \CNCoalg(A)$, the semi-free $\Lie_1$-algebra
  \[
  ({\rm Free}_{\Lie_1}(\overline{C}), \d_{\bbOmega})
  \]
  whose differential is the Lie algebra derivation of degree 1 determined by the shift of the coproduct on $\overline{C}$.
  This adjunction is natural in $A$, so in particular it gives a
  simplicial adjunction between the associated simplicial categories.
\end{defn}

Note that by working with shifted Lie algebras, 
we obviate the need to shift in constructing the Chevalley-Eilenberg chains.

\begin{remark}
  For a field $k$ of characteristic zero, Hinich~\cite{HinichDgCoalg}
  constructs a model structure on $\CNCoalg(k)$ where all objects are
  cofibrant and the fibrant objects are the \emph{semi-free}
  coalgebras, \ie{} those whose underlying graded coalgebra is $\Sym^c(V)$ for some graded vector space $V$. 
  The cobar-bar adjunction is a Quillen equivalence between this model category and
  that of Lie algebras. In particular, for any $L \in \Lie_1(A)$, the
  adjunction counit $\bbOmega \mathbb{B} L \to L$ is a cofibrant
  replacement of $L$. We do not know if an analogous model structure
  exists on $\CNCoalg(A)$ when $A$ is not a field, but it turns out
  that $\bbOmega \mathbb{B} L$ is still often a cofibrant replacement,
  which is enough for our purposes:
\end{remark}

\begin{lemma}\label{lem:cofcobareq}
  Let $L$ be a shifted Lie algebra over $A$ whose underlying
  $A$-module is cofibrant. Then 
  \begin{enumerate}[(i)]
  \item the counit map $\bbOmega \mathbb{B} L \to L$ is a
  weak equivalence, and 
  \item the shifted Lie algebra $\bbOmega \mathbb{B} L$ is cofibrant.
  \end{enumerate}
\end{lemma}

\def\BB{\mathbb{B}}

\begin{proof}
The bar coalgebra is the colimit of a sequence of coalgebras
\[
A \to \BB^1(\g) \to \BB^2(\g) \to \cdots, 
\]
where
\[
\BB^k(\g) := \left(\bigoplus_{n=0}^k\Sym^n_A(\g), \d_{\BB(\g)} \right), 
\]
since the coproduct on the symmetric coalgebra decreases symmetric powers and the differential preserves and lowers the symmetric powers. 
Note that the cokernel of the map $\BB^{k-1}(\g) \to \BB^k(\g)$ is simply $\Sym_A^k(\g)$.

Consider the following pushout square in $A$-modules:
\[
\begin{tikzcd}
\Sym_A^k(\g)[-1]   \arrow{r}{\d_{\BB^k(\g)}} \arrow{d} & \BB^{k-1}(\g) \arrow{d} \\
C(\id) \arrow{r} & \BB^k(\g),
\end{tikzcd}
\] 
where 
\begin{itemize}
\item the top horizontal map is the differential on $\BB^k(\g)$, restricted to $\Sym_A^k(\g)[-1]$, and viewed as a degree zero map,  and
\item the bottom left corner $C(\id)$ denotes the cone of the identity map from $\Sym_A^k(\g)[-1]$ to itself.
\end{itemize}
This is a pushout square because it is a pushout on the underlying graded vector spaces 
and the maps are also compatible with the differentials. 
The left vertical map is a cofibration in $A$-modules, and hence the right vertical map is a cofibration of $A$-modules.

We can also view this square as a commutative diagram of cocommutative coalgebras, 
where the two coalgebras on the left-hand side have zero coproduct. 
Apply the cobar functor to this square. 
On the left-hand side it reduces to the free $\Lie_1$-algebra functor, and 
hence we have a cofibration of $\Lie_1$-algebras in $A$-modules. 
As it is a pushout square of $\Lie_1$-algebras, the right vertical map is also a cofibration.

The base case $\BB^1(\g)$ is a free $\Lie_1$-algebra on a cofibrant $A$-module and 
hence a cofibrant $\Lie_1$-algebra. 
Hence $\BB^k(\g)$ is also cofibrant as a $\Lie_1$-algebra, 
since it is the $k$-iterated pushout along a cofibration.
\end{proof}

\begin{defn}
  Let $\LI(A)$ denote the full subcategory of $\CNCoalg(A)$ spanned by
  the objects $\mathbb{B}L$ where $L$ is a $\Lie_1$-algebra over $A$ whose
  underlying $A$-module is cofibrant. Let $\sLI(A)$ denote the
  analogous simplicial category.
\end{defn}

\begin{lemma}
  The simplicial category $\sLI(A)$ is fibrant.
\end{lemma}
\begin{proof}
  Given objects $\mathbb{B}L$ and $\mathbb{B}L'$ in $\sLI(A)$, the
  simplicial set of maps $\sHom_{\sLI(A)}(\mathbb{B}L, \mathbb{B}L')$
  is isomorphic to $\sHom(\bbOmega\mathbb{B}L, L')$, since
  the cobar-bar adjunction is simplicial. Since $\bbOmega\mathbb{B}L$
  is cofibrant by Lemma~\ref{lem:cofcobareq}, this simplical set is a
  Kan complex by Proposition~\ref{propn:cohframe}(i).
\end{proof}

We write $\LLI(A)$ for the \icat{} obtained as the coherent nerve of
the fibrant simplicial category $\sLI(A)$.

\begin{lemma}
  The simplicial functor $\bbOmega \colon \sLI(A) \to \sLie_1(A)$ is a
  weak equivalence of simplicial categories.
  Hence there is an induced equivalence of \icats{} $\bbOmega \colon \LLI(A) \isoto \LLie_1(A)$.
\end{lemma}

\begin{proof}
  Given objects $\mathbb{B}L$ and $\mathbb{B}L'$ in $\sLI(A)$, we have
  a commutative diagram of simplicial sets
  \nolabelopctriangle{\sHom_{\sLI(A)}(\mathbb{B}L,
    \mathbb{B}L')}{\sHom_{\sLie_1(A)}(\bbOmega \mathbb{B} L, \bbOmega
    \mathbb{B} L')}{\sHom_{\sLie_1(A)}(\bbOmega \mathbb{B} L, L').}
  Here the left diagonal map is an isomorphism, since the cobar-bar
  adjunction is simplicial, and the right diagonal map is a weak
  equivalence by Proposition~\ref{propn:cohframe} since $\bbOmega
  \mathbb{B} L$ is cofibrant and $\bbOmega \mathbb{B} L' \to L'$ is a
  weak equivalence by Lemma~\ref{lem:cofcobareq}. Thus $\bbOmega$ is
  weakly fully faithful.

  It remains to show that $\bbOmega$ is essentially surjective on the
  homotopy category. Given $L \in \sLie_1(A)$, then by definition
  $L$ is a cofibrant $\Lie_1$-algebra, so by
  Proposition~\ref{propn:opdch}(ii) its underlying $A$-module is also
  cofibrant. Lemma~\ref{lem:cofcobareq} therefore implies that the
  counit $\bbOmega \mathbb{B} L \to L$ is a weak equivalence of
  cofibrant Lie algebras, and hence an equivalence in the simplicial
  category $\sLie_1(A)$.  
\end{proof}

We need to know that this equivalence respects symmetric monoidal structures 
coming from the Cartesian
product on both sides. For ${\LLie_1}(A)$, we constructed this product in
\S\ref{subsec:icatenv}. The case of $\LLI(A)$ is easy: The product
of conilpotent cocommutative coalgebras is given by the tensor product
over $A$, and thus $\sLI(A)$ is closed under products --- 
$\mathbb{B}$ is a right adjoint and so $\mathbb{B}L \otimes
\mathbb{B}L' \cong \mathbb{B}(L \oplus L')$. This construction is also
compatible with base change, so the simplicial category $\sLI(A)$
inherits a symmetric monoidal structure, and hence so does $\LLI(A)$.

Since this right adjoint $\mathbb{B}$ preserves products, its left
adjoint $\bbOmega$ is oplax monoidal. We thus have a lax monoidal
functor on the opposite categories. This functor is compatible with the
simplicial enrichments, so we get a map of simplicial operads from
$\sLI(A)^{\op}$ (which is a symmetric monoidal simplicial category) to
$\sLie_1(A)^{\op}$ (with the simplicial operad structure described in
\S\ref{subsec:icatenv}). Taking coherent nerves, we get a lax symmetric monoidal
functor of \icats{} $\LLI(A)^{\op} \to {\LLie_1}(A)^{\op}$.

\begin{lemma}
  The lax monoidal functor $\bbOmega \colon \LLI(A)^{\op} \to
  {\LLie_1}(A)^{\op}$ induced by $\bbOmega$ is, in fact, symmetric
  monoidal.
\end{lemma}

\begin{proof}
  We must show that for any objects $\mathbb{B}L$ and $\mathbb{B}L'$
  in $\LI(A)$, the oplax structure map \[\bbOmega(\mathbb{B}L \otimes
  \mathbb{B}L') \cong \bbOmega\mathbb{B}(L \oplus L')
  \to\bbOmega\mathbb{B}L \oplus \bbOmega\mathbb{B}L'\] is a weak
  equivalence of (cofibrant) Lie algebras. The counit transformation
  gives a commutative diagram
  \nolabelopctriangle{\bbOmega\mathbb{B}(L \oplus
    L')}{\bbOmega\mathbb{B}L \oplus \bbOmega\mathbb{B}L'}{L \oplus L'}
  where the diagonal maps are weak equivalences (for the right
  diagonal map, this holds since weak equivalences are closed under
  $\oplus$). By the 2-out-of-3 property the horizontal map is hence
  also a weak equivalence.
\end{proof}

The symmetric monoidal functor $\LLI(A)^{\op} \to {\LLie_1}(A)^{\op}$
then induces a symmetric monoidal functor on opposite \icats{}, $\LLI(A)
\to \LLie_{1}(A)$. Moreover, it is easy to see (using the same argument as
in Lemma~\ref{lem:Liesymmmonftr}) that this construction is natural in the
commutative algebra variable.

\subsection{The Heisenberg $L_{\infty}$-Algebra}\label{subsec:HeisLI}

We construct here a symmetric monoidal functor $$\H_\infty \colon
\QQuad_1(A) \to \LLI(A)$$ that produces a Heisenberg
$L_\infty$-algebra from a quadratic module of degree 1. As earlier, we
begin by constructing a 1-categorical functor, upgrade it to a functor
of simplicial categories, and then take the coherent nerves.

We will construct a functor $$\txt{H}_\infty: \Quad_1(A) \to \LI(A)$$
that sends $(V,\omega)$ to $\mathbb{B}\Heis_{1}(V,\omega) =
\CL(\Heis_{1}(V,\omega))$.  We then need to associate functorially
to each map $F: (V,\omega) \to (V',\omega')$ in $\Quad_1(A)$, a map of
cocommutative coalgebras
\[
\txt{H}_\infty(F) \colon \mathbb{B}(\Heis_{1}(V,\omega)) \to \mathbb{B}(\Heis_{1}(V',\omega')).
\]
This map $\txt{H}_\infty(F)$ is easy to describe once we have some elementary results about coalgebras in hand.

\begin{lemma}\label{gcoalgmap}
Let $D$ be a degree zero coderivation of a conilpotent graded coalgebra $C$ (\ie{} with trivial differential). Then $\exp(D)$ is a coalgebra automorphism of $C$.
\end{lemma}

\begin{proof}
We compute
\begin{align*}
\Delta \circ \exp(D) 
&= \sum_{n \geq 0} \frac{1}{n!} \Delta \circ D^n \\
&= \sum_{n \geq 0} \frac{1}{n!} (\Delta \circ D) \circ D^{n-1} \\
&= \sum_{n \geq 0} \frac{1}{n!} (D \otimes \id + \id \otimes D) \circ \Delta \circ D^{n-1} \\
& \quad\quad \vdots \\
&= \sum_{n \geq 0} \frac{1}{n!} (D \otimes \id + \id \otimes D)^n \circ \Delta \\
&= \sum_{n \geq 0} \frac{1}{n!} \sum_{m=0}^n {n \choose m} D^m \otimes D^{n-m} \circ \Delta \\
&= \sum_{p,q \geq 0} \frac{1}{p!q!} {p+q \choose p} D^p \otimes D^q \circ \Delta \\
&= (\exp(D) \otimes \exp(D)) \circ \Delta
\end{align*}
as desired. The inverse is clearly $\exp(-D)$.
\end{proof}

If $(C, \d_C)$ is a differential graded coalgebra and $\delta$ is a
\emph{Maurer-Cartan element} in the Lie algebra of coderivations
$\rm{Coder}(C)$, \ie{} a degree one element such that
\[
[\d_C, \delta] + \delta^2 = 0,
\]
then $(C,\d_C + \delta)$ defines another coalgebra (with the same coproduct).

\begin{lemma}\label{dgcoalgmap}
Let $(C, \d_C)$ be a differential graded coalgebra. 
Let $\delta_1$ and $\delta_2$ be Maurer-Cartan elements in the Lie algebra $\rm{Coder}(C)$. 
If there exists a degree zero coderivation $D$ such that 
\begin{enumerate}[(i)]
\item $[\d_C, D] = \delta_1 - \delta_2$ and 
\item $[D,\delta_1] = 0 = [D,\delta_2]$, 
\end{enumerate}
then $\exp(D)$ provides a coalgebra automorphism from $(C,\d_C + \delta_1)$ to $(C,\d_C + \delta_2)$.
\end{lemma}

\begin{proof}
We compute
\begin{align*}
\d_C \circ \exp(D)
&= \sum_{n \geq 0} \frac{1}{n!} \d_C \circ D^n \\
&= \sum_{n \geq 0} \frac{1}{n!} (\d_C \circ D) \circ D^{n-1} \\
&= \sum_{n \geq 0} \frac{1}{n!} (D \circ \d_C +  \delta_1 - \delta_2) \circ D^{n-1} \\
& \quad\quad \vdots \\
&= \sum_{n \geq 0} \frac{1}{n!} (D^n \circ \d_C + n (\delta_1 - \delta_2) \circ D^{n-1})\\
&= \sum_{n \geq 0} \frac{1}{n!} D^n \circ \d_C + \sum_{n \geq 0} \frac{1}{(n-1)!} (\delta_1 - \delta_2) \circ D^{n-1}) \\
&= \exp(D) \circ \d_C + \exp(D) \circ \delta_1 - \delta_2 \circ \exp(D).
\end{align*}
In short,
\[
(\d_C +\delta_2) \circ \exp(D) = \exp(D) \circ (\d_C + \delta_1),
\]
as desired.
\end{proof}

Now we can prove the key lemma:

\begin{lemma}
Let $F \colon (V,\omega) \to (V',\omega')$ be a map in $\Quad_1(A)$ where $f \colon V \to V'$ is a map of $A$-modules and $\eta \colon \Lambda^2_A V \to A[1]$ is a homotopy from $\omega$ to $f^* \omega'$. Let $D_\eta$ denote the degree zero coderivation on $\CL(\Heis_{1}(V,\omega))$ determined by $\eta$. Then the composite map
\[
\Sym^c_A(f) \circ \exp(D_\eta)
\]
is a map in $\CNCoalg(A)$ from $\mathbb{B}(\Heis_{1}(V,\omega))$ to $\mathbb{B}(\Heis_{1}(V',\omega'))$.
\end{lemma}

\begin{proof}
Recall that the differential on $\CL(\Heis_{1}(V,\omega))$ is a
sum of degree 1 coderivations $\d_V + \delta_\omega$ where $\d_V$ denotes
the differential on $V$ extended to $\Sym^c_A(V)$ as a coderivation
and $\delta_\omega$ likewise denotes $\omega$ --- viewed as a degree 1 map
from $\Sym^2_A(V)$ to $A$ --- extended as a coderivation. (The obvious
analogues hold for the other coalgebras, such as $\CL(\Heis_{1}(V',\omega'))$.) 
Thus, $\Sym^c_A(f)$ is a coalgebra map from $\CL(\Heis_{1}(V,f^*\omega'))$ to $\CL(\Heis_{1}(V',\omega'))$ 
since $f$ naturally provides a map of shifted Lie algebras from $\Heis_{1}(V,f^*\omega)$ to $\Heis_{1}(V',\omega')$. 
It remains to show that $\exp(D_\eta)$ is a map of coalgebras.

This claim follows from Lemma \ref{dgcoalgmap} once we verify that $[D_\eta, \delta_{\omega}] = 0 = [D_\eta, \delta_{f^*\omega'}]$. Without loss of generality, we simply verify that the commutator with $\delta_{\omega}$ vanishes. Note that it suffices to compute the commutator $[D_\eta, \delta_{\omega}]$ just on the second stage of the filtration
\[
F^2\; \CL(\Heis_{1}(V,\omega)) = \Sym^{\leq 2}(V \oplus A\c),
\] 
since any coderivation preserves the filtration by symmetric powers and a coderivation is determined by its behavior on cogenerators. But on this stage of the filtration, both $\eta$ and $\omega$ map into $A\c$, and they both vanish on $A\c$, so their commutator vanishes.
\end{proof}

We need to show that this construction respects composition.

\begin{lemma}
Let $F \colon (V,\omega) \to (V',\omega')$ and $G \colon (V',\omega') \to (V'',\omega'')$ be maps in $\Quad_1(A)$, with $f \colon V \to V'$ and $g \colon V' \to V''$ the maps of $A$-modules and $\eta\colon \Lambda^2_A V \to A[1]$ and $\gamma\colon \Lambda^2_A V' \to A[1]$ the respective homotopies. Then
\[
\Sym^c_A(g\circ f) \circ \exp(D_{f^* \gamma + \eta}) = \Sym^c_A(g) \circ \exp(D_\gamma) \circ \Sym^c_A(f) \circ \exp(D_\eta),
\]
where $f^*\gamma + \eta \colon \Lambda^2_A V \to A[1]$ is the homotopy from $\omega$ to $f^* g^* \omega''$ obtained by composing $\eta$ and $f^* \gamma$ is the natural way.
\end{lemma}

\begin{proof}
Observe that 
\[
\exp(D_\gamma) \circ \Sym^c_A(f)= \Sym^c_A(f)\circ \exp(D_{f^*\gamma}),
\]
essentially by the definition of $D_{f^*\gamma}$. Next observe that $D_\eta$ and $D_{f^*\gamma}$ commute, by the argument in the preceding lemma:  they are determined by their behavior on cogenerators and that is defined on the second stage of the filtration, but both have image in $A\c$ and vanish on $A\c$. Hence
\[
\exp(D_{f^*\gamma}) \exp(D_{\eta}) = \exp(D_{f^*\gamma+ \eta}).
\]
Thus
\begin{align*}
\Sym^c_A(g\circ f) \circ \exp(D_{f^* \gamma + \eta}) &= \Sym^c_A(g) \circ \Sym^c_A(f) \circ \exp(D_{f^*\gamma}) \exp(D_{\eta}) \\
&=  \Sym^c_A(g) \circ \exp(D_\gamma) \circ \Sym^c_A(f) \circ \exp(D_{\eta}),
\end{align*}
as claimed.
\end{proof}

Putting these results together, we can make the following definition:
\begin{defn}
Let $\txt{H}_\infty \colon \Quad_1(A) \to \LI(A)$ denote the functor sending $(V,\omega)$ to $\mathbb{B}\Heis_{1}(V,\omega)$ and sending a map $F = (f,\eta) \colon (V,\omega) \to (V',\omega')$ to $\Sym^c_A(f) \circ \exp(D_\eta)$. 
\end{defn}

\begin{propn}
The functor $\txt{H}_\infty$ is lax symmetric monoidal, sending $\oplus$ to $\otimes_{A}$.
\end{propn}

The laxness is a consequence of the fact that each Heisenberg algebra contributes a central element. Thus, $\txt{H}_\infty(V,\omega) \otimes_A \txt{H}_\infty(V',\omega')$ has two central elements $\c$ and $\c'$. By contrast, if we take a direct sum before constructing the Heisenberg algebra, we only have one central element $\c$. We identify these two central elements with one another to produce a map
\[
\txt{H}_\infty(V,\omega) \otimes_A \txt{H}_\infty(V',\omega') \to \txt{H}_\infty(V \oplus V',\omega+\omega').
\] 
This construction provides the natural transformation making $\txt{H}_\infty$ lax symmetric monoidal. 

The functor $\txt{H}_\infty$ is natural in $A$, and we now want to use
this naturality, applied to $A \otimes \Omega(\Delta^{\bullet})$, to
get a functor of simplicial categories that is again natural in
$A$. However, the naturality in $A$ is not strict: since the tensor
product is not strictly associative, but only associative up to
isomorphism, it is only \emph{pseudo}natural. The following is
therefore not entirely obvious, but uses some pseudofunctorial
observations we have delegated to the appendix.

\begin{cor}
  The functor $\mathrm{H}_{\infty}$ induces a lax symmetric monoidal
  functor of simplicial categories $\mathbf{H}_{\infty} \colon
  \sQuad_{1}(A) \to \sLI(A)$. Moreover, this functor is
  pseudonatural in the commutative algebra $A$.
\end{cor}
\begin{proof}
  The pseudonaturality in $A$ means that $\mathrm{H}_{\infty}$ is a
  natural transformation of pseudofunctors $\Comm(k) \to \Cat$. For
  any commutative algebra $A$, tensoring with
  $\Omega(\Delta^{\bullet})$ gives a functor $\simp^{\op} \to
  \Comm(k)$, so composing with this we have a natural transformation
  of pseudofunctors $\simp^{\op} \to \Cat$. By the results of
  \S\ref{sec:pseudo} this induces a functor of simplicial categories
  $\sQuad_{1}(A) \to \sLI(A)$,
  as both the simplicial categories $\sQuad_{1}(A)$ and
  $\sLI(A)$ arise as in Proposition~\ref{propn:pseudosimp}.

  Tensoring an arbitrary commutative algebra with
  $\Omega(\Delta^{\bullet})$ gives a functor $\Comm(k) \times
  \simp^{\op} \to \Comm(k)$, and composing with this we get by
  adjunction a pseudofunctor $\Comm(k) \times [1] \to
  \Fun^{\txt{Ps}}(\simp^{\op}, \Cat)$, where the target denotes the
  $2$-category of pseudofunctors. The observations of
  \S\ref{sec:pseudo} give a functor from $\Fun^{\txt{Ps}}(\simp^{\op},
  \Cat)$ to the $2$-category $\CAT_{\Delta}$ of simplicial categories, and composing
  these we end up with a pseudofunctor $\Comm(k) \times [1] \to \CAT_{\Delta}$
  that exhibits the pseudonaturality of $\mathbf{H}_{\infty}$. A
  similar argument for the associated simplicial operads gives these
  functors lax monoidal structures, also pseudonatural in $A$.
\end{proof}

\begin{cor}
  The functor $\mathbf{H}_{\infty}$ induces a lax symmetric monoidal
  functor of \icats{} \[\HH_\infty \colon \QQuad_1(A) \to \LLI(A)\]
  via the coherent nerve. Moreover, this functor is natural in $A \in \CComm(k)$.
\end{cor}
\begin{proof}
  We saw above that $H_{\infty}$ determines a pseudofunctor $\Comm(k)
  \times [1] \to \CAT_{\Delta}$. If we restrict to cofibrant
  commutative algebras, then this functor takes weak equivalences of
  commutative algebras to weak equivalences of simplicial
  categories. Proceeding as in the proof of Lemma~\ref{lem:AlgOftr} we
  get from this the desired functor of \icats{} $\CComm(k) \times \Delta^{1} \to
  \CatI$. A similar argument with simplicial operads gives the lax
  monoidal structure.
\end{proof}
 
The composite $\bbOmega \circ \HH_\infty$ is then a lax symmetric
monoidal functor $\HH \colon \QQuad_{1}(A) \to {\LLie_{1}}(A)$.  It
takes the unit $0$ in $\QQuad_{1}(A)$ to a 1-dimensional abelian Lie
algebra we'll denote by $A\c$, and therefore factors through a lax
monoidal functor $\widetilde{\mathcal{H}} \colon \QQuad_{1}(A) \to
\MMod_{A\c}({\LLie_{1}}(A))$.

\begin{lemma}\label{lem:Hsymmon}
  The lax symmetric monoidal functor $\widetilde{\mathcal{H}}$ is symmetric monoidal.
\end{lemma}

\begin{proof}
  Let us write $X \otimes_{A\c} Y$ for the tensor product in
  $\MMod_{A\c}({\LLie_{1}}(A))$, which is by definition the geometric
  realization $|X \oplus A\c \oplus \cdots \oplus Y|$ in the \icat{}
  ${\LLie_{n}}(A)$. Then we must show that the natural map
  $\mathcal{H}(V, \omega) \otimes_{A\c} \mathcal{H}(V', \omega') \to
  \mathcal{H}(V \oplus V', \omega + \omega')$ is an equivalence. Since
  the forgetful functor to $\MMod(k)$ detects equivalences and
  preserves sifted colimits by Proposition~\ref{propn:opdforgetsifted},
  it suffices to check that the underlying map in $\MMod(k)$ is an
  equivalence. But in $\MMod(k)$ we can identify the image of
  $\mathcal{H}(V, \omega) \otimes_{A\c} \mathcal{H}(V', \omega')$ with
  the pushout $\mathcal{H}(V, \omega) \amalg_{A\c} \mathcal{H}(V',
  \omega')$. It therefore suffices to show that $\mathcal{H}(V \oplus
  V', \omega + \omega')$ is correspondingly a homotopy pushout in
  $\Mod(k)$. To see this we can replace $\mathcal{H}(V
  \oplus V', \omega + \omega')$ with the quasi-isomorphic cochain
  complex $V \oplus V' \oplus A\c$, which is clearly the pushout of $V
  \oplus A\c$ and $V' \oplus A\c$ over $A\c$. Since the inclusions $A\c
  \to V \oplus A\c, V' \oplus A\c$ are cofibrations (as $V$ and
  $V'$ are cofibrant, and cofibrations are closed under pushouts) this
  is a homotopy pushout in $\Mod(k)$, which completes the proof.
\end{proof}

\section{Linear BV Quantization}

Combining the constructions of the previous sections, we get a
symmetric monoidal functor of \icats{}
\[ \bvq: \QQuad_{1}(A) \xto{\HH} \MMod_{A\mathbf{c}}({\LLie_{1}}(A))
\xto{U_{\BD}} \MMod_{A[\mathbf{c},\hbar]}(\AAlg_{\BD}(A[\hbar]))\]
that we call \emph{linear BV quantization}.  (As explained in
Section~\ref{div}, setting $\c = \hbar$ recovers the construction
typically seen in the literature.)  In this section we will explore
some properties of this functor and its close cousin $$\quant :=
\ev_{\hbar = \c = 1} \circ U_\BD \circ \HH \colon \QQuad_{1}(A) \to
\AAlg_{\rE_{0}}(A),$$ which we call simply \emph{linear quantization}.
In \S\ref{subsec:BVstack} we show that there is a natural extension
from modules to quasicoherent sheaves on derived stacks, so that
linear BV quantization is a well-posed construction in derived
geometry.  Then in \S\ref{subsec:quantfmp} and
\S\ref{subsec:quantsympvb} we show that this functor behaves like a
determinant when restricted either to symplectic modules on a formal
moduli problem or to symplectic vector bundles on a derived stack,
after dealing with the base case of symplectic $k$-modules in
\S\ref{subsec:quantk}.  Finally, in \S\ref{subsec:higherbv} we combine
our functors with the higher Morita category construction of
\cite{nmorita} to get symmetric monoidal functors of
$(\infty,n)$-categories.

\subsection{Linear BV Quantization as a Map of Derived
  Stacks}\label{subsec:BVstack}

We will show here that our BV quantization functor extends for formal reasons
from commutative algebras to derived stacks. 

Recall that a $\mathcal{C}$-valued \emph{\'{e}tale sheaf} is a
presheaf $\mathcal{F} \colon \CComm(k) \to \mathcal{C}$ that satisfies
\emph{\'{e}tale descent}: it preserves finite products and takes
\emph{derived \'{e}tale covers} (which we will not define here,
\cf{} \cite{HAG2}*{Definition 2.2.2.12} or \cite{DAG5}*{Definition
  4.3.13}) to cosimplicial limits. A \emph{derived stack} (in the most
general sense) is an $\widehat{\mathcal{S}}$-valued \'{e}tale sheaf,
where $\widehat{\mathcal{S}}$ is the \icat{} of large spaces. We use
$\dSt_{k}$ to denote the full subcategory of $\Fun(\CComm(k),
\widehat{\mathcal{S}})$ spanned by the derived stacks. It is then a
formal consequence of the definition (\cf{} \cite{DAG5}*{Proposition
  5.7}) that for any (very large) presentable \icat{}
$\mathcal{C}$, the \icat{} $\Fun^{R}(\dSt_{k}^{\op}, \mathcal{C})$ of
limit-preserving functors is equivalent to the \icat{} of functors
$\CComm(k) \to \mathcal{C}$ that are \'{e}tale sheaves, via
restricting along the  Yoneda embedding
$\CComm(k)^{\op} \to \dSt_{k}$.  The inverse functor is given simply
by taking right Kan extensions.

We have constructed natural transformations $\QQuad_{1}(\blank) \to
\LLie_{1}(\blank)$, $\LLie_{1}(\blank) \to \AAlg_{\BD}(\blank)$, etc.,
of functors $\CComm(k) \to \LCatI$. To see that these extend to
natural transformations of functors on derived stacks, it suffices to
show that the functors in question are \'{e}tale sheaves. This claim
will follow quite straightforwardly from Lurie's descent theorem for
modules:

\begin{thm}[Lurie \cite{DAG7}*{Theorem 6.1}]\label{LurieDesc}
  The functor $\MMod(\blank) \colon \CComm(k) \to \LCatI$ is an
  \'{e}tale sheaf.
\end{thm}

\begin{remark}
  In fact, Lurie's result is substantially more general: he shows that
  $\MMod(\blank)$ is a hypercomplete sheaf in the flat topology, and
  that this holds for modules over commutative ring spectra.
\end{remark}

As a trivial consequence we have:
\begin{lemma}
  The functor $\QQuad_{n}(\blank)$ satisfies \'etale descent, and so
  has a natural extension to a limit-preserving functor $\QQuad_{n}
  \colon \dSt_{k}^{\op} \to \CatI$.
\end{lemma}
\begin{proof}  
  Immediate from Lurie's descent theorem and the description of
  $\QQuad_{n}(A)$ as a pullback of \icats{}.
\end{proof}

For any commutative algebra $R \in \CComm(k)$ and operad $\mathbf{O}$
in $\Mod(R)$, we have a functor $\AAlg_{\mathbf{O}}(R \otimes \blank)
\colon \CComm(k) \to \LCatI$ (\cf{} Lemma~\ref{lem:AlgOftr}). We now
prove that this functor also satisfies descent; this is no doubt
well-known to the experts --- in particular, in the case of Lie
algebras it is stated by Hennion as \cite{Hennion}*{Proposition
  2.1.3}.

\begin{propn}\label{propn:AlgOdesc}
  Let $R$ be a commutative algebra over $k$ and let $\mathbf{O}$ be an operad in
  $\Mod(R)$.  Then the functor $\AAlg_{\mathbf{O}}(R \otimes \blank)$
  satisfies \'etale descent, and so determines a limit-preserving
  functor
  \[
  \AAlg_{\mathbf{O}}(R \otimes \blank) \colon \txt{dSt}_{k}^{\op} \to \PrL,
  \]
  where $\PrL$ is the \icat{} of presentable \icats{} and left adjoint functors.
\end{propn}

\begin{remark}
  It is easy to enhance this to get, for instance, a functor $\OOpd(k)
  \times \txt{dSt}_{k}^{\op} \to \PrL$, where $\OOpd(k)$ is the
  \icat{} of operads in $\MMod(k)$, \eg{} obtained by inverting the
  weak equivalences between (cofibrant) operads in $\Mod(k)$. To see
  this, observe that such a functor is equivalent to a functor
  $\CComm(k) \to \Fun(\OOpd(k), \PrL)$ that satisfies \'{e}tale
  descent, which it does \IFF{} it does so when evaluated at each
  operad (since limits in functor \icats{} are computed
  objectwise). The simplicial categories $\sAlg_{(\blank)}(\blank)$ are
  naturally pseudofunctorial in both variables, so they determine a
  pseudofunctor $\Opd(k) \times \Comm(k) \to \CAT_{\Delta}$. By
  Proposition~\ref{propn:opdch}(iii--iv), if we restrict to cofibrant
  objects of $\Comm(k)$ then this functor takes quasi-isomorphisms in
  both variables to weak equivalences of simplicial
  categories. Localizing, we obtained the required functor $\OOpd(k)
  \times \CComm(k) \to \PrL$. 
\end{remark}

\begin{remark}
  If we had a good theory of enriched $\infty$-operads, we would be
  able to formally identify $\AAlg_{\mathbf{O}}(X)$, for $X$ a derived
  stack, with the \icat{} of $\mathbf{O}$-algebras in the \icat{}
  $\QCoh(X)$ of quasicoherent sheaves on $X$, regarded as enriched
  over $k$-modules.
\end{remark}

For any derived stack $X$, we thus obtain natural functors
\[ \bvq(X) : \QQuad_{1}(X) \to \MMod_{\mathcal{O}_{X}[\mathbf{c},\hbar]}\AAlg_{\BD}(X[\hbar])\]
and
\[ \quant(X) : \QQuad_{1}(X) \to \MMod_{\mathcal{O}_{X}}(\AAlg_{\rE_{0}}(X))\]
from our earlier work.

For the proof of this proposition, we need a technical result:

\begin{propn}\label{propn:limcartcocart}
  Let $p \colon \mathcal{E} \to \mathcal{C}^{\triangleleft}$ and $q
  \colon \mathcal{F} \to \mathcal{C}^{\triangleleft}$ be Cartesian and
  coCartesian fibrations. Suppose the functor $\phi \colon
  \mathcal{C}^{\triangleleft} \to \CatI$ associated to $q$ is a limit
  diagram.  If a functor $F \colon \mathcal{E} \to \mathcal{F}$ over
  $\mathcal{C}^{\triangleleft}$ satisfies
  \begin{enumerate}[(a)]
  \item the functor $F$ preserves both Cartesian and coCartesian morphisms,
  \item for every $x \in \mathcal{C}^{\triangleleft}$, the functor $F_{x} \colon \mathcal{E}_{x} \to \mathcal{F}_{x}$ detects equivalences and preserves limits,
  \item the \icats{} $\mathcal{E}_{x}$ are complete for all $x \in
    \mathcal{C}^{\triangleleft}$,
  \end{enumerate}
then the functor $\epsilon \colon \mathcal{C}^{\triangleleft} \to \CatI$ associated to $p$ is also a limit diagram.
\end{propn}

\begin{proof}
  For $c \in \mathcal{C}$, let $e_{c}$ denote the unique map $-\infty
  \to c$ in $\mathcal{C}^{\triangleleft}$, where $-\infty$ denotes the
  initial object (or cone point) of $\mathcal{C}^{\triangleleft}$. By \cite{DAG7}*{Lemma
    5.17}, we know that $\epsilon$ is a limit diagram \IFF{} the
  following conditions hold:
  \begin{enumerate}[(1)]
  \item The functors $e_{c,!} \colon \mathcal{E}_{-\infty} \to
    \mathcal{E}_{c}$ are jointly conservative, \ie{} if $f$ is a
    morphism in $\mathcal{E}_{-\infty}$ such that $e_{c,!}f$ is an
    equivalence in $\mathcal{E}_{c}$ for all $c \in \mathcal{C}$, then
    $f$ is an equivalence.
  \item If $G \colon \mathcal{C} \to \mathcal{E}$ is a coCartesian
    section of $p$ over $\mathcal{C}$, then $G$ can be extended to a
    $p$-limit diagram $\bar{G} \colon \mathcal{C}^{\triangleleft} \to
    \mathcal{E}$, such that $\bar{G}$ carries $e_{c}$ to a coCartesian
    morphism for all $c \in \mathcal{C}$.
  \end{enumerate}
  Let us first prove (1). Suppose $f$ is a morphism in
  $\mathcal{E}_{-\infty}$ such that $e_{c,!}f$ is an equivalence in
  $\mathcal{E}_{c}$ for all $c \in \mathcal{C}$. Since $F_{-\infty}$
  detects equivalences, to show that $f$ is an equivalence it suffices
  to prove that $F_{-\infty}f$ is an equivalence. But as $\phi$ is
  a limit diagram, $F_{-\infty}f$ is an equivalence \IFF{}
  $e_{c,!}F_{-\infty}f$ is an equivalence for all $c \in
  \mathcal{C}$. And since $F$ preserves coCartesian morphisms, we have
  natural equivalences $e_{c,!}F_{-\infty}f \simeq F_{c}e_{c,!}f$,
  hence these maps are indeed equivalences.

  Now we prove (2). The functor $p$ is a Cartesian fibration, its
  fibres are complete, and the Cartesian pullback functors preserve
  limits since they are right adjoints. Therefore the $p$-limit of any
  diagram $G$ exists by \cite{HTT}*{Corollary 4.3.1.11}. Moreover, by
  \cite{HTT}*{Proposition 4.3.1.10} we know that the limit is given by
  the limit in $\mathcal{E}_{-\infty}$ of the Cartesian pullback of
  the diagram $G$ to this fibre. Given the $p$-limit diagram $\bar{G}$
  we are left with proving that the maps $G(e_{c}) \colon
  \bar{G}(-\infty) \to G(c)$ are all coCartesian, \ie{} that the
  induced maps $e_{c,!}\bar{G}(-\infty) \to G(c)$ are equivalences in
  $\mathcal{E}_{c}$. Since the functors $F_{c}$ detect equivalences,
  it suffices to show that the maps
  $e_{c,!}F_{-\infty}\bar{G}(-\infty) \simeq
  F_{c}e_{c,!}\bar{G}(-\infty) \to F_{c}G(c)$ are equivalences in
  $\mathcal{F}_{c}$, \ie{} that the maps $F\bar{G}(e_{c})$ are
  $q$-coCartesian. But since $F_{-\infty}$ preserves limits and $F$
  preserves Cartesian morphisms, we know that
  $F_{-\infty}\bar{G}(-\infty)$ is the limit of the Cartesian pullback
  of $FG$ to $\mathcal{F}_{-\infty}$, which is the $q$-limit of $F
  \circ G$. Since $\phi$ is a limit diagram, we know that (2) holds
  for $q$, \ie{} that $F\bar{G}(e_{c})$ is coCartesian for all $c$.
\end{proof}

\begin{cor}\label{cor:flatdesccartcocart}
  Let $q \colon \MMod \to \CComm(k)$ be the coCartesian (and
  Cartesian) fibration associated to the functor $\MMod(\blank) \colon
  \CComm(k) \to \CatI$, and suppose $p \colon \mathcal{E} \to
  \CComm(k)$ is a Cartesian and coCartesian fibration. If $F \colon
  \mathcal{E} \to \MMod$ is a functor over $\CComm(k)$ such that
  \begin{enumerate}[(a)]
  \item $F$ preserves Cartesian and coCartesian morphisms,
  \item for every $A \in \CComm(k)$, the functor $F_{A} \colon
    \mathcal{E}_{A} \to \MMod(A)$ detects equivalences and preserves
    limits,
  \item the \icats{} $\mathcal{E}_{A}$ are complete for all $A \in \CComm(k)$,
  \end{enumerate}
  then the functor $\epsilon \colon \CComm(k) \to \CatI$ associated to
  $p$ satisfies \'etale descent, and so determines a
  limit-preserving functor $\epsilon \colon \txt{dSt}_{k}^{\op} \to
  \CatI$ from the \icat{} of derived stacks over $k$.
\end{cor}

\begin{proof}
  Combine Proposition~\ref{propn:limcartcocart} with Theorem~\ref{LurieDesc}.
\end{proof}

\begin{proof}[Proof of Proposition~\ref{propn:AlgOdesc}]
  The conditions of Corollary~\ref{cor:flatdesccartcocart} clearly
  hold in this situtation.
\end{proof}

\subsection{Quantization over $k$}\label{subsec:quantk}

Over the field $k$, the linear quantization functor $\quant$ 
is particularly well-behaved on quadratic $k$-modules 
that are non-degenerate and have finite-dimensional cohomology: 
the quantization has one-dimensional cohomology. 
In fact, we will see that for a cohomologically finite $W$, 
there is a numerical factor $d_W$ depending on the Betti numbers $\dim
H^{*}W$ of $W$
such that
\[
\quant(T^*[1]W) \simeq \det(W)[-d_W],
\]
where $T^*[1]W$ denotes $W \oplus W^*[1]$ with symplectic pairing 
given by the skew-symmetrization of the evaluation pairing. 
In other words, $\quant$ is a determinant-type functor, 
which illuminates one sense in which it provides 
a homological approach to integration, 
as the determinant of a vector space is the natural home for volume forms on it.

Every cochain complex is quasi-isomorphic to its cohomology, 
so it suffices to verify this invertibility property on such modules.

\begin{lemma}\label{detlemma}
Let $V$ be a cochain complex with zero differential such that $\dim_k V^d < \infty$ for all $d$ and it vanishes for $d \gg 0$ and $d \ll 0$.
Suppose $V$ is equipped with a non-degenerate pairing $\omega \colon \Lambda^2 V \to k[1]$.
Then
\[
H^d( \quant(V,\omega)) \cong \begin{cases} k, & d = \sum_{n} (2n+1)\dim_k V^{2n+1}, \\ 0, & \txt{else}.\end{cases}
\]
\end{lemma}

In short, 
\[
\quant(V,\omega) \simeq \det(\bigoplus_n V^{2n+1})[-m],
\]
where $m = \sum_{n} (2n+1)\dim_k V^{2n+1}$ is the index from the lemma.
In particular, when $V = T^*[1]W = W \oplus W^*[1]$ with the natural pairing $\omega_\ev$, we find that 
\[
\quant(T^*[1]W) \simeq \det(W)[-d_W]
\]
where $d_W = \sum_n (2n+1)(\dim_k W^{2n+1} - \dim_k W^{2n})$.

\begin{proof}
The vector space $V$ is a direct sum of atomic components of the following form: Let $V_n$ denote the graded vector space with a copy of $k$ in degree $n$ and a copy of $k$ in degree $-1-n$, and the obvious pairing $\omega_n$. Denote the degree $n$ generator by $x$ and the dual generator in degree $-1-n$ by $\xi$. Then $\omega_n(x,\xi) = 1$. Observe that
\[
\quant(V_n,\omega_n) = (k[x,\xi], \triangle = \partial^2/\partial x \partial \xi),
\]
by unraveling the definitions. Without loss of generality, assume that $\xi$ has odd degree. (Otherwise, just swap the labels on $x$ and $\xi$.) Then compute that $\triangle(x^{a+1} \xi) = \pm (a+1) x^a$, and so every monomial $x^a$ is exact and only the monomial $\xi$ is closed. Hence
\[
H^*\quant(V_n,\omega_n) \cong k \,\xi \cong k[1+n],
\]
when $n$ is even. If $x_1,\ldots,x_N$ is a set of even degree basis elements for $V$ and $\xi_1,\ldots,\xi_N$ the dual set of odd degree basis elements, then 
\[
H^*\quant(V,\omega) \cong k \,\xi_1 \cdots \xi_N \cong k[-m],
\]
where $m = \sum_i |\xi_i|$. In short, $H^*\quant(V,\omega)$ is isomorphic to the top exterior power (or determinant) of the odd-degree components of $V$, but placed in degree $m$.
\end{proof}

\subsection{Quantization over Formal Moduli Stacks}\label{subsec:quantfmp}

The main result of this section is that this determinant-type behavior
extends to formal moduli problems. We carefully state the result here and 
spend the rest of the section working through the proof.

To do this, we need to introduce the notion of a symplectic module.

\begin{defn}
  A quadratic module $(V,\omega) \in \QQuad_n(A)$ is \emph{symplectic}
  of degree $n$ if the pairing $\omega$ is non-degenerate, \ie{} the
  associated map $\omega^* \colon V \to \Hom_A(V,A[n])$ is an
  equivalence. Let $\SSymp_n(A)$ denote the full subcategory of
  $\QQuad_n(A)$ whose objects are symplectic.
\end{defn}

Following Lurie, a commutative algebra $A$ is \emph{small} if it sits in a finite sequence
\[
A = A_0 \to A_1 \to \cdots \to A_n = k
\]
of algebras where each map $A_i \xto{f_{i}} A_{i+1}$ is an \emph{elementary extension}, \ie{}
sits in a pullback square of commutative algebras
\[
\begin{tikzcd}
A_i \arrow{d}{f_{i}} \arrow{r} & k \arrow{d} \\
A_{i+1} \arrow{r} & k \oplus k[n]
\end{tikzcd}
\]
with $n \geq 0$, where $k \oplus k[n]$ is the trivial square-zero
extension of $k$ by the $k$-module $k[n]$.
Note that $A$ is naturally augmented.
The \icat{} $\AAlg^{sm}(k)$ of small algebras is the full subcategory of
augmented connective commutative algebras $\CComm(k)^{\leq 0}_{/k}$.
(Lurie uses ``small'' where many people use ``Artinian.'')

\begin{defn}\cite{SAG}*{Chapter 13}
A \emph{formal moduli problem} is a functor $X$ from $\AAlg^{sm}(k)$ to $\mathcal{S}$ 
such that 
\begin{enumerate}
\item $X(k) \simeq \ast$ and
\item given a pullback square $\sigma$ of small algebras
\csquare{A'}{B'}{A}{B}{}{}{\phi}{}
where $\phi$ is elementary, then its image $X(\sigma)$ is a pullback square in $\mathcal{S}$.
\end{enumerate}
Let $\Moduli$ denote the \icat{} of formal moduli problems over the field $k$. 
\end{defn}

Let $\QCoh$ denote the functor 
\[
X \in \Moduli \mapsto \lim_{\Spec(A) \to X} \MMod(A),
\]
as defined in \cite{SAG}*{\S13.4.6}.
Likewise, let $\QQuad_1$ denote the analogous functor for 1-shifted quadratic modules, 
let $\SSymp_1$ denote the analogous functor for 1-shifted symplectic modules, and
let $\PPic$ denote the analogous functor for $\otimes$-invertible objects. 

\begin{thm}
The functor $\quant: \SSymp_1 \to \QQCoh$ factors through $\PPic$ when
restricted to $\Moduli$.
\end{thm}

In other words, $\quant$ is a determinant-type functor on the
symplectic modules over any formal moduli problem.  To prove this
theorem, it suffices to verify it on all small algebras.

By Lemma \ref{detlemma}, we know the theorem holds on $k$, so now we need to extend to an arbitrary small commutative algebra.
In the usual style of arguments in formal geometry, we will show that the relevant property --- here, invertibility --- plays nicely with an elementary extension.

\begin{lemma}
For $A$ a small commutative algebra, an $A$-module $M$ is invertible if and only if its base-change $k \otimes_A M$ along the augmentation is invertible over $k$.
\end{lemma}

\begin{proof}
  Suppose we have an elementary extension \csquare{A}{k}{B}{k\oplus
    k[n]}{g}{f}{}{} and write $h$ for the composite $A \to k \oplus
  k[n]$. We will show that an $A$-module $X$ is invertible \IFF{}
  $f_{!}X$ and $g_{!}X$ are invertible; this will then imply the
  result by induction. The ``only if'' direction is obvious, since any
  strong symmetric monoidal functor preserves invertible objects.

  To prove the other direction, we start with the pullback square of
  $A$-modules 
  \csquare{A}{g^{*}k}{f^{*}B}{h^{*}(k\oplus k[n]).}{}{}{}{} 
  If we write $\txt{HOM}_{A}$ for the internal Hom,
  then any $A$-module $X$ yields a pullback square
  \nolabelcsquare{\txt{HOM}_{A}(X, A)}{\txt{HOM}_{A}(X,g^{*}k)}{\txt{HOM}_{A}(X, f^{*}B)}{\txt{HOM}_{A}(X, h^{*}(k\oplus k[n])).}  
  We have natural isomorphisms $\txt{HOM}_{A}(X, f^{*}B)
  \cong f^{*}\txt{HOM}_{B}(f_{!}X, B)$, so writing $\mathbb{D}_{A}X$ for
  the dual $\txt{HOM}_{A}(X, A)$, and so on, we have a pullback square of
  $A$-modules
  \nolabelcsquare{\mathbb{D}_{A}X}{g^{*}\mathbb{D}_{k}(g_{!}X)}{f^{*}\mathbb{D}_{B}(f_{!}X)}{h^{*}\mathbb{D}_{k\oplus k[n]}(h_{!}X).}  
      Tensoring with $X$ and applying the
  projection formula (\ie{} the natural equivalences $M \otimes f^{*}M'
  \simeq f^{*}(f_{!}M \otimes M')$, etc.), we finally get a pullback square
  \nolabelcsquare{X \otimes_{A} \mathbb{D}_{A}X}{g^{*}(g_{!}X \otimes_{k} \mathbb{D}_{k}(g_{!}X))}{f^{*}(f_{!}X \otimes_{B}
    \mathbb{D}_{B}(f_{!}X))}{h^{*}(h_{!}X \otimes_{k\oplus k[n]} \mathbb{D}_{k \oplus k[n]}(h_{!}X)).}  
  Using the evaluation maps, we get a map of pullback squares 
  from this square to the original pullback square from above (with $A$ in the upper left corner).
  Hence the evaluation map $X \otimes_{A} \mathbb{D}_{A}X \to A$ is an
  equivalence if the evaluation maps for $f_{!}X$, $g_{!}X$ and
  $h_{!}X$ are equivalences, \ie{} $X$ is invertible if $f_{!}X$,
  $g_{!}X$, and $h_{!}X$ are. But $h_{!}X$ is the image of $f_{!}X$
  (and of $g_{!}X$) under a strong symmetric monoidal functor, hence
  it is invertible if $f_{!}X$ (or $g_{!}X$) is. This means that $X$
  is invertible if $f_{!}X$ and $g_{!}X$ are invertible, as required.
\end{proof}

In consequence, we obtain the desired claim:

\begin{cor}
For $A$ a small commutative algebra over $k$ and $(V,\omega) \in \SSymp_1(A)$, the quantization $\quant(V,\omega)$ is invertible over $A$.
\end{cor}

\begin{proof}
If $k \otimes_{A} \quant(V,\omega)$ is invertible over $k$, then $\quant(V,\omega)$ is invertible over $A$. 
But the construction of quantization commutes with base-change, so
\[
k \otimes_{A} \quant(V,\omega) \simeq \quant(k \otimes_A V, k \otimes_A \omega),
\]
which is invertible since the base-change of $(V,\omega)$ is a small symplectic module over $k$. 
\end{proof}

\subsection{Symplectic Vector Bundles on Derived Stacks}\label{subsec:quantsympvb}

In this section, we examine how quantization behaves on derived Artin stacks.
As the conditions we are considering are checked locally, our work in the preceding sections implies:

\begin{cor}
For $V$ a 1-shifted symplectic module on a derived Artin stack $X$,
the fiber of $\quant(V)$ at any $k$-point $p: \Spec(k) \to X$ is invertible.
More generally, the pullback of $\quant(V)$ to the formal moduli stack $X^\wedge_p$ is invertible.
\end{cor}

Thus, we know that very locally --- in the formal neighborhood of any
$k$-point --- quantization is well-behaved, but its behavior as the
point varies is complicated.  In particular, the quantization of a
\emph{perfect} symplectic module need not be invertible. Here is a
particularly simple example:
\begin{ex}
  Consider $\mathbb{A}^1 = \Spec(k[t])$, where $t$ has degree 0, and
  the two-term $k[t]$-module
\[
0 \to k[t] \xto{t \cdot -} k[t] \to 0
\]
concentrated in degrees 0 and $-1$, which is perfect, and equipped with the natural shifted pairing from evaluation.
Then the quantization is given by the complex
\[
(k[t,x,\xi], \triangle = t x \partial/\partial \xi + \partial^2/\partial x \partial \xi),
\]
which is concentrated in degree $-1$ after specializing to $t = 0$ and is concentrated in degree 0 after specializing $t$ to any nonzero value.
In other words, the quantization jumps at the origin.
This result is not so strange as it may appear at first: 
this quadratic module presents the skyscraper sheaf at the origin and hence already jumps at the origin.
By contrast, the determinant of this complex is just $k[t]$ concentrated in degree 0.
  
\end{ex}

\begin{remark}
This example underscores a key difference between the usual determinant functor and linear BV quantization:
the determinant ignores the differential on the cochain complex, whereas BV quantization depends on the differential.
(Note that the Euler characteristic also has the remarkable property that it does not depend on the differential,
and the determinant is a kind of categorification of the Euler characteristic.)
\end{remark}

Hence, we would like to restrict our attention to some collection of symplectic modules
on which quantization does produce invertible modules.

Recall that the notion of being \emph{locally free of rank $d$} is well-behaved
in derived algebraic geometry in the \'etale topology. (See Section 2.9 of \cite{SAG}.)
Let $\VB$ denote the derived stack of \emph{vector bundles}, \ie{} modules that are 
locally free of some finite rank, which clearly admits a natural map to $\MMod$.
Let $\GVB$ denote the derived stack of \emph{graded vector bundles}, 
which consists of modules that are finite direct sums of shifts of vector bundles. 

\begin{defn}
The derived stack of \emph{$n$-shifted symplectic vector bundles} is the pullback
\[ \nodispcsquare{\SSymp_n\VB}{\SSymp_n}{\GVB}{\MMod}{}{}{}{} \]
of derived stacks.
\end{defn}

We now restrict to a collection of 1-shifted symplectic vector bundles
on which $\quant$ is well-behaved.

\begin{defn}
The {\em cotangent quantization} of a graded vector bundle $V$ on a derived stack $X$
is the quantization of its shifted cotangent bundle $T^*[1]V = V \oplus V^\vee[1]$, 
equipped with the skew-symmetrization of its evaluation pairing.
\end{defn}

Let $\mathcal{C}\rm{ot}$ denote the full subcategory of $\SSymp_1\VB$ 
whose objects are equivalent to a cotangent quantization.
Note that \emph{a priori} the construction $V \mapsto T^*[-1]V$ is not a functor,
as not every map of graded vector bundles gives a map between the shifted cotangent bundles.
By restricting to equivalences (\ie{} the underlying $\infty$-groupoid
$\iota \GVB$ of $\GVB$), 
one can form a functor $T^*[-1]-: \iota\GVB \to \mathcal{C}\rm{ot}$, with proper care.
The composite $\quant \circ T^*-$ would be a ``cotangent quantization functor";
we will not do that here.
Instead, for simplicitly,
we denote by $\cq$ the \emph{cotangent quantization functor} given 
by restricting $\quant$ to $\mathcal{C}\rm{ot}$.

\begin{propn}
The symmetric monoidal functor $\cq: \mathcal{C}\rm{ot}^\oplus \to \QCoh^\otimes$ 
factors through the substack $\PPic$.
\end{propn}

\begin{proof}
This can be checked locally on affines.
By \cite{SAG}*{Proposition 2.9.2.3}, we can reduce to the case where the graded vector bundle is a direct sum of shifted free modules. 
In this free case, working over a commutative algebra $A$, 
the cotangent quantization is equivalent to the module
\[ (\Sym_A(V \oplus V^\vee[1]), \d_V + \d_{V^\vee[1]} + \Delta),\]
where $\Delta$ is determined by the pairing.
The argument from Lemma \ref{detlemma} applies here, suitably interpreted.
(In fact, one can view this complex as base-changed from a complex over $k$.)
Hence, there is a natural $A$-linear quasi-isomorphism
\[ A[d_V] \simeq \cq(V), \]
where
\[ d_V = \sum_{n} (2n+1)( b_{2n+1} - b_{2n}) \]
with $b_m$ the number of degree $m$ generators of the graded vector bundle $V$.
\end{proof}

\begin{remark}
This construction applies to any 1-shifted symplectic module $V$
for which, sufficiently locally, the module is free and 
the pairing has the standard form.
In other words, it applies to any 1-shifted symplectic module
that is locally a shifted cotangent space.
It seems plausible that all 1-shifted symplectic vector bundles have this form, 
but we do not pursue a classification of shifted symplectic vector bundles here.
\end{remark}

\subsection{Higher BV Quantization}\label{subsec:higherbv}

The paper \cite{nmorita} constructs for any nice symmetric monoidal
\icat{} $\mathcal{C}$ a symmetric monoidal $(\infty,n)$-category
$\mathfrak{Alg}_{n}(\mathcal{C})$, whose objects are
$\mathrm{E}_{n}$-algebras in $\mathcal{C}$ and whose $i$-morphisms are
$i$-fold iterated bimodules in $\mathrm{E}_{n-i}$-algebras in
$\mathcal{C}$. Here the precise meaning of ``nice'' holds in
particular if $\mathcal{C}$ has sifted colimits and the tensor product
preserves these in each variable. Moreover, the construction is
natural in $\mathcal{C}$ with respect to symmetric monoidal functors
that preserve sifted colimits. We will now show that these assumptions
hold for the \icats{} and functors involved in our linear BV
quantization, giving:
\begin{propn}
  For any derived stack $X$ there is a diagram of symmetric monoidal
  $(\infty,n)$-categories and symmetric monoidal functors
\[ \mathfrak{Alg}_{n}(\QQuad_{1}(X)) \to
\mathfrak{Alg}_{n}(\MMod_{\mathcal{O}_{X}\mathbf{c}}{\LLie_{1}}(X)) \to
\mathfrak{Alg}_{n}(\MMod_{\mathcal{O}_{X}[\hbar,\mathbf{c}]}\AAlg_{\BD}(X))
\to \mathfrak{Alg}_{n}(\AAlg_{\teo}(X)).\]
\end{propn}

We must prove that the \icats{} in question have sifted colimits, and
that these are preserved by the functors between them. For the operad
algebra \icats{} this follows from
Proposition~\ref{propn:opdforgetsifted}, so it remains to check for
$\QQuad_{n}(X)$ and the Heisenberg algebra functor:

\begin{lemma}\label{lem:Quadsifted}\ 
  \begin{enumerate}[(i)]
  \item The \icat{} $\QQuad_{n}(A)$ has sifted colimits, and the
    forgetful functor $\QQuad_{n}(A) \to \MMod(A)$ detects these.
  \item The direct sum of quadratic modules preserves sifted colimits
    in each variable.
  \end{enumerate}
\end{lemma}

\begin{proof}
  In the Cartesian square
  \csquare{\QQuad_{n}(A)}{\MMod(A)_{/A[n]}}{\MMod(A)}{\MMod(A),}{p}{q}{u}{\Lambda^{2}_{[-n]}}
  the \icats{} $\MMod(A)$ and $\MMod(A)_{/A[n]}$ have all colimits,
  the forgetful functor $u$ preserves all colimits, and the functor
  $\Lambda^{2}_{[-n]}$ preserves sifted colimits. By \cite{HTT}*{Lemma
    5.4.5.5} the \icat{} $\QQuad(A)$ therefore has all sifted
  colimits, and if $\bar{f} \colon K^{\triangleright} \to \QQuad(A)$
  is a diagram with $K$ sifted, then $\bar{f}$ is a colimit diagram
  \IFF{} $p \circ \bar{f}$ and $q \circ \bar{f}$ are both colimit
  diagrams. But $u$ detects all colimits, so $p \circ \bar{f}$ is a
  colimit diagram \IFF{} $u \circ p \circ \bar{f}\simeq
  \Lambda^{2}_{[-n]}\circ q \circ \bar{f}$ is a colimit diagram. Since
  $\Lambda^{2}_{[-n]}$ preserves sifted colimits this is implied by $q \circ
  \bar{f}$ being one, so $\bar{f}$ is a colimit diagram \IFF{} $q
  \circ \bar{f}$ is one. In other words, $q$ detects sifted colimits,
  giving (i).

  (ii) then follows since the direct sum in $\MMod(A)$ preserves
  colimits in each variable.
\end{proof}

\begin{lemma}\label{lem:Heissifted}
  The functor $\HH\colon \QQuad_{n}(A) \to {\LLie_{n}}(A)$
  preserves sifted colimits.
\end{lemma}
\begin{proof}
  We have a commutative diagram
  \csquare{\QQuad_{n}(A)}{{\LLie_{n}}(A)}{\MMod_{n}(A)}{\MMod_{n}(A),}{\mathcal{H}}{}{}{(\blank)
    \oplus A\c}
  where the vertical functors detect sifted colimits by
  Lemma~\ref{lem:Quadsifted} and
  Proposition~\ref{propn:opdforgetsifted}. It therefore suffices to
  observe that the functor $(\blank) \oplus A$ preserves sifted
  colimits (and more generally colimits indexed by weakly contractible
  \icats{}) since colimits commute and the colimit diagram of a
  constant diagram indexed by a weakly contractible \icat{} is constant.
\end{proof}

\appendix

\section{Some Technicalities}

\subsection{From Pseudofunctors to Simplicial
  Categories}\label{sec:pseudo} 
In this appendix we will show that there is a natural way to
produce a simplicial category from a pseudofunctor $\simp^{\op} \to
\CAT$, where $\CAT$ is the 2-category of categories. More precisely,
we will see that there is a functor of 2-categories from the
2-category $\Fun^{\txt{Ps}}(\simp^{\op}, \CAT)$ of pseudofunctors to
the 2-category $\CAT_{\Delta}$ of simplicial categories.

Recall that if $\mathbf{C}$ is a category, a \emph{pseudofunctor} $F$
from $\mathbf{C}$ to the (strict) 2-category $\txt{CAT}$ of categories
(or to its underlying (2,1)-category) consists of the following data:
\begin{itemize}
\item for each object $X \in \mathbf{C}$, a category $F(X)$,
\item for each morphism $f \colon X \to Y$ in $\mathbf{C}$, a functor
  $F(f) \colon F(X) \to F(Y)$,
\item for each object $X \in \mathbf{C}$, a natural isomorphism $u_{X}
  \colon F(\id_{X}) \Rightarrow \id_{F(X)}$
\item for each pair of composable morphisms $f \colon X \to Y$, $g
  \colon Y \to Z$ in $\mathbf{C}$, a natural isomorphism $\eta_{f,g}
  \colon F(g \circ f) \Rightarrow F(g) \circ F(f)$,
\end{itemize}
such that
\begin{itemize}
\item for every morphism $f \colon X \to Y$, the triangles
\[  \nodispopctriangle{F(f)}{F(f) \circ
    F(\id_{X})}{F(f),}{\eta_{\id_{X},f}}{\id_{F(f)}}{F(f) \circ u_{X}} \qquad
  \nodispopctriangle{F(f)}{F(\id_{Y}) \circ
    F(f)}{F(f)}{\eta_{f,\id_{Y}}}{\id_{F(f)}}{u_{Y} \circ F(f)} \]
  both commute.
\item for composable triples of morphisms $f \colon X \to Y$, $g
  \colon Y \to Z$, $h \colon Z \to W$, the square
  \csquare{F(h\circ g \circ f)}{F(h) \circ F(g \circ f)}{F(h \circ g)
    \circ F(f)}{F(h) \circ F(g) \circ F(f)}{\eta_{g\circ f,
      h}}{\eta_{f, h \circ g}}{F(h) \circ \eta_{f,g}}{\eta_{g,h} \circ
    F(f)} commutes.
\end{itemize}

\begin{propn}\label{propn:pseudosimp}
 Suppose $\mathbf{C} \colon \simp^{\op} \to \txt{CAT}$ is a
pseudofunctor (where we write $\mathbf{C}_{n}$ for the image of $[n]$
and $\phi^{*} \colon \mathbf{C}_{n} \to \mathbf{C}_{m}$ for the image
of $\phi \colon [m] \to [n]$ in $\simp$). Then $\mathbf{C}$ determines
a simplicial category $\mathbf{C}^{\Delta}$ as follows:
\begin{itemize}
\item the objects of $\mathbf{C}^{\Delta}$ are the objects of $\mathbf{C}_{0}$,
\item for $x, y$ in $\mathbf{C}_{0}$ the $n$-simplices in
  $\mathbf{C}^{\Delta}$ are given by $\mathbf{C}_{n}(\sigma_{n}^{*}x,
  \sigma_{n}^{*}y)$, where $\sigma_{n}$ denotes the unique map $[n]
  \to [0]$ in $\simp$.
\item for $\phi \colon [m] \to [n]$ in $\simp$, the corresponding
  simplicial structure map in $\mathbf{C}^{\Delta}(x,y)$, which we
  denote $\tilde{\phi}^{*}$, takes $f \colon \sigma_{n}^{*}x \to
  \sigma_{n}^{*}y$ to the composite
  \[ \sigma_{m}^{*}x = (\sigma_{n} \circ \phi)^{*}x
  \xto{\eta_{\phi,\sigma_{n}}} \phi^{*}\sigma_{n}^{*}x \xto{\phi^{*}f}
  \phi^{*}\sigma_{n}^{*}y \xto{\eta^{-1}_{\phi,\sigma_{n}}}
  (\sigma_{n}\circ \phi)y = \sigma_{m}^{*}y.\]
\end{itemize}
\end{propn}
\begin{proof}
To see that this does indeed define a simplicial set, we must check
that these maps respect composition and identities. For composition,
take $\psi \colon [k] \to [m]$ and $\phi \colon [m] \to [n]$ and
consider for $f \colon \sigma_{n}^{*}x \to \sigma_{n}^{*}y$ the
commutative diagram
\[
\begin{tikzcd}
  \sigma_{k}^{*}x \arrow{r}{\sim} \arrow{d}{\sim}& (\psi\phi)^{*}\sigma_{n}^{*}x \arrow{r}{(\psi\phi)^{*}f}\arrow{d}{\sim} &
  (\psi\phi)^{*}\sigma_{n}^{*}y \arrow{d}{\sim} & \sigma_{k}^{*}y
  \arrow{l}{\sim} \arrow{d}{\sim} \\
  \psi^{*}\sigma_{m}^{*}x \arrow{r}{\sim} & \psi^{*}\phi^{*}\sigma_{n}^{*}x \arrow{r}{\psi^{*}\phi^{*}f}&
  \psi^{*}\phi^{*}\sigma_{n}^{*}y & \psi^{*}\sigma_{m}^{*}y, \arrow{l}{\sim}
\end{tikzcd}
\]
where the unlabelled maps come from the natural isomorphisms
$\eta$. Here the top horizontal composite is
$(\widetilde{\psi\phi})^{*}f$, the bottom horizontal composite is
$\psi^{*}\tilde\phi^{*}f$, and the composite from the top left to the
top right along the bottom row is
$\tilde{\psi}^{*}\tilde{\phi}^{*}f$. For identities, consider for $f$
as above the commutative diagram
\[
\begin{tikzcd}
  \sigma_{n}^{*}x \arrow{r}{\sim} \arrow{dr}{\id}& (\id_{[n]})^{*}\sigma_{n}^{*}x
  \arrow{r}{(\id_{[n]})^{*}f} \arrow{d}{\sim}&
  (\id_{[n]})^{*}\sigma_{n}^{*}y \arrow{d}{\sim} & \sigma_{n}^{*}y \arrow{l}{\sim} \arrow{dl}{\id}\\
   & \sigma_{n}^{*}x \arrow{r}{f} & \sigma_{n}^{*}y. &
\end{tikzcd}
\]
Here the top horizontal composite is $\widetilde{\id_{[n]}}^{*}f$ and the
composite along the bottom is $f$. It is then clear that composition
in the categories $\mathbf{C}_{n}$ induces composition maps for
$\mathbf{C}^{\Delta}$.
\end{proof}

A \emph{natural transformation} $\lambda$ from $F$ to $G$ of pseudofunctors from
$\mathbf{C}$ to $\txt{CAT}$ consists of the data of:
\begin{itemize}
\item for every $X \in \mathbf{C}$ a functor $\lambda_{X} \colon F(X)
  \to G(X)$,
\item for every morphism $f \colon X \to Y$ in $\mathbf{C}$ a natural
  isomorphism $\lambda(f) \colon \lambda_{Y} \circ F(f) \Rightarrow
  G(f) \circ \lambda_{X}$,
\end{itemize}
satisfying the obvious pentagon and triangle identities.

A \emph{modification} $\epsilon$ from $\lambda$ to $\mu$ of natural
transformations from $F$ to $G$ of pseudofunctors from 
$\mathbf{C}$ to $\txt{CAT}$ consists of the data of:
\begin{itemize}
\item for every object $X \in \mathbf{C}$, a natural transformation
  $\epsilon_{X} \colon \lambda_{X} \Rightarrow \mu_{X}$,
\end{itemize}
such that $\mu(f) \circ (G(f) \circ \epsilon_{X}) = (\epsilon_{Y}\circ
F(f)) \circ \lambda(f)$.

These give, respectively, the 1- and 2-morphisms in a 2-category
$\Fun^{\txt{Ps}}(\mathbf{C}, \CAT)$ of pseudofunctors. It is easy to
see that the construction taking a pseudofunctor $\mathbf{C} \colon
\simp^{\op} \to \CAT$ to the simplicial category $\mathbf{C}^{\Delta}$
is natural with respect to these morphisms and 2-morphisms, giving:
\begin{cor}
  The construction $(\mathbf{C} \colon \simp^{\op} \to \CAT) \mapsto
  \mathbf{C}^{\Delta}$ extends naturally to a functor of 2-categories
  $\Fun^{\txt{Ps}}(\simp^{\op}, \CAT) \to \CAT_{\Delta}$, where
  $\CAT_{\Delta}$ is the 2-category of simplicial categories.
\end{cor}
We leave the (straightforward) details of the proof to the
reader.

\subsection{Simplicial Enrichment of Model Categories}
\label{sec: SEMC}

The \icats{} we work with in this paper mostly arise from model
categories, and in order to carry out our construction we want to show
that these \icats{} can also be described using natural simplicial
enrichments of these model categories. For simplicial model categories
this is a standard result, originally due to Dwyer and Kan
\cite{DwyerKanFnCpx}. However, in our case the model categories are
not quite simplicial in the usual sense, so we will
need a slight variant of the comparison of Dwyer-Kan; this result is
no doubt well-known to the experts, but we have included a proof in
this appendix as we were not able to find a reference in the
literature.

We begin by recalling how we construct the \icat{} associated to a
model category, or more generally to a \emph{relative category}. A
relative category $(\bC, W)$ is a category $\bC$ equipped with a
collection $W$ of ``weak equivalences'', \ie{} a collection of
morphisms in $\bC$ that contains all the isomorphisms and satisies the
2-out-of-3 property: if $f$ and $g$ are composable morphisms such that
two out of $f$, $g$, and $gf$ are in $W$, then so is the
third. Relative categories are the most basic form of homotopical data
on a category, and have been studied as a model for \icats{} by
Barwick and Kan \cite{BarwickKanRelCat}. If $(\bC, W)$ is a relative
category, we can invert the weak equivalences $W$ to obtain an \icat{}
$\bC[W^{-1}]$:
\begin{defn}
  Let $\|\blank\| \colon \CatI \to \mathcal{S}$ be the left adjoint to
  the inclusion $\mathcal{S} \hookrightarrow \CatI$ of the \icat{}
  $\mathcal{S}$ of spaces into that of \icats{}. Thus if $\mathcal{C}$
  is an \icat{},  $\|\mathcal{C}\|$ is the space or $\infty$-groupoid
  obtained by inverting all the morphisms in $\mathcal{C}$. If
  $(\bC, W)$ is a relative category, let $\mathbf{W}$ denote
  the subcategory of $\bC$ containing only the morphisms in $W$. Then
  the localization $\bC[W^{-1}]$ is defined by the pushout square of \icats{}
  \nolabelcsquare{\mathbf{W}}{\|\mathbf{W}\|}{\bC}{\bC[W^{-1}].}
\end{defn}

\begin{remark}
  More generally, if $W$ is a collection of ``weak equivalences'' in
  an \icat{} $\mathcal{C}$, we similarly define $\mathcal{C}[W^{-1}]$
  by the pushout square
  \nolabelcsquare{\mathcal{W}}{\|\mathcal{W}\|}{\mathcal{C}}{\mathcal{C}[W^{-1}],}
  where $\mathcal{W}$ is the subcategory of $\mathcal{C}$ containing
  only the morphisms in $W$.
\end{remark}

We are interested in simplicial enrichments that arise in a
particularly pleasant way, as follows: Suppose $T_{\bullet}$ is a
simplicial monad on a category $\bC$, \ie{} a simplicial object in the
category of monads, or more explicitly a collection of functors $T_{n}
\colon \bC \to \bC$ together with natural transformations $\mu_{n}
\colon T_{n} \circ T_{n} \to T_{n}$ (the multiplication) and $\eta_{n}
\colon \id \to T_{n}$ (the unit) satisfying the usual identities, as
well as natural transformations $\phi^{*} \colon T_{n} \to T_{m}$ for
every map $\phi \colon [m] \to [n]$ in $\simp$, compatible with the
multiplication and unit transformations. Then we can define a
simplicial category $\bCs$ with the same objects as $\bC$, where the
mapping space $\Map_{\bCs}(X, Y)$ is given by $\Hom_{\bC}(X,
T_{\bullet}Y)$. The identity map of $X$ corresponds to the unit $X \to
T_{0}X$, and the composition $\Map_{\bCs}(X,Y)\times \Map_{\bCs}(Y, Z)
\to \Map_{\bCs}(X,Z)$ is given on simplices by taking $f\colon X \to
T_{n}Y$ and $g \colon Y \to T_{n}Z$ to the composite
\[ X \xto{f} T_{n}Y \xto{T_{n}g} T_{n}T_{n}Z \xto{\mu} T_{n}Z.\]

In other words, if we let $\bC_{n}$ be the Kleisli category of the
monad $T_{n}$, then $\bC_{\bullet}$ is a simplicial object in
categories with a constant set of objects, and $\bCs$ is the
associated simplicial category. Note that there is a canonical functor
$\bC \to \bCs$ that is the identity on objects and sends a map $X \to
Y$ to $X \to Y \xto{\eta} T_{0}Y$.

\begin{remark}
  Our notation is slightly abusive, in that the underlying category of
  the simplicial category $\bCs$ is not in general $\bC$ --- a
  morphism $X \to Y$ in $(\bCs)_{0}$ corresponds to a morphism $X \to
  T_{0}Y$ in $\bC$. However, the monad $T_{0}$ is typically the
  identity in examples.
\end{remark}

We'll now use results of Hovey to give conditions for a
simplicial monad to interact well with the weak equivalences in a
model category:
\begin{defn}\label{defn:cohfr}
  Let $\mathbf{C}$ be a model category. A \emph{coherent right
    framing} on $\mathbf{C}$ is a simplicial monad $T_{\bullet}$ on
  $\mathbf{C}$ such that for every object $X$ of $\mathbf{C}$ the unit
  maps $X \to T_{n}X$ are weak equivalences for all $n$, and if $X$ is
  fibrant then $T_{\bullet}X$ is a Reedy fibrant simplicial object of
  $\mathbf{C}$.
\end{defn}

\begin{propn}[Hovey]\label{propn:cohframe}
  Suppose $\mathbf{C}$ is a model category equipped with a coherent
  right framing $T_{\bullet}$. Then:
  \begin{enumerate}[(i)]
  \item If $Y$ is a fibrant object of $\mathbf{C}$, then the functor
    $\Hom_{\mathbf{C}}(\blank, T_{\bullet}Y)$ is a right Quillen
    functor from $\mathbf{C}^{\op}$ to $\sSet$, \ie{} it takes
    cofibrations and trivial cofibrations in $\mathbf{C}$ to Kan
    fibrations and trivial Kan fibrations. In particular, if $X$ is
    cofibrant, then $\Hom_{\mathbf{C}}(X, T_{\bullet}Y)$ is a Kan
    complex, and $\Hom_{\mathbf{C}}(\blank, T_{\bullet}Y)$ preserves
    weak equivalences between cofibrant objects.
  \item If $X$ is a cofibrant object of $\mathbf{C}$, then
    $\Hom_{\mathbf{C}}(X, T_{\bullet}\blank)$ preserves weak
    equivalences between fibrant objects in $\mathbf{C}$.
  \end{enumerate}
\end{propn}
\begin{proof}
  From the definition of a coherent right framing it is evident that
  for $Y$ fibrant the simplicial object $T_{\bullet}Y$ is a simplicial
  framing of $Y$ in the sense of \cite{HoveyModCats}*{Definition
    5.2.7}. Part (i) is therefore a special case of
  \cite{HoveyModCats}*{Corollary 5.4.4}, and (ii) of
  \cite{HoveyModCats}*{Corollary 5.4.8}.
\end{proof}

\begin{cor}\label{cor:cohframesimp}
  Suppose $\mathbf{C}$ is a model category equipped with a coherent
  right framing $T_{\bullet}$. Let $\bCcfs$ be the simplicial category
  of fibrant-cofibrant objects in $\mathbf{C}$ with mapping spaces
  $\Hom_{\mathbf{C}}(\blank, T_{\bullet}\blank)$. Then $\bCcfs$ is a
  fibrant simplicial category.
\end{cor}
\begin{proof}
  By Proposition~\ref{propn:cohframe}(i) the mapping spaces in
  $\bCcfs$ are all Kan complexes, so it is fibrant as a simplicial
  category.
\end{proof}

Our goal is now to prove the following comparison result:
\begin{propn}\label{propn:cohfrloc}
  Suppose $\mathbf{C}$ is a model category equipped with a coherent
  right framing. Then the natural maps
  \[ \mathbf{C}[W^{-1}] \to \NobCs[W^{-1}] \from
  \NbCcfs[W_{\txt{cf}}^{-1}] \from \NbCcfs\]
  are equivalences, where $W_{\txt{cf}}$ denotes the class of weak
  equivalences between fibrant-cofibrant objects and $\obCs$ is a
  fibrant replacement for the simplicial category $\bCs$.
\end{propn}

\begin{remark}
  This is just a minor variant of \cite{DwyerKanFnCpx}*{Proposition
    4.8}, although our proof is slightly different.
\end{remark}

We will prove this by considering the three maps separately. Let us
start with the easiest one:
\begin{lemma}\label{lem:cfdontloc}
  The map $\NbCcfs \to \NbCcfs[W_{\txt{cf}}^{-1}]$ is an equivalence
  of \icats{}.
\end{lemma}
\begin{proof}
  It suffices to show that the morphisms in $W_{\txt{cf}}$ are already
  equivalences in $\NobCcfs$. But if $f \colon X \to X'$ is a weak equivalence
  between fibrant-cofibrant objects, then it follows from
  Proposition~\ref{propn:cohframe}(i) that
  \[ \sHom_{\NobCcfs}(X', Z) \to \sHom_{\NobCcfs}(X, Z)\] is an
  equivalence for all fibrant-cofibrant $Z$, hence $f$ is an
  equivalence in $\NobCcfs$.
\end{proof}

Let us write $\bCc$ and $\bCf$ for the full subcategories of $\bC$
spanned by the cofibrant and fibrant objects, and $\bCcs$ and $\bCfs$
for the corresponding full subcategories of the simplicial category
$\bCs$. Then we have the following observation, a version of which is
found in the proof of \cite{HA}*{Theorem 1.3.4.20}:
\begin{propn}\label{propn:repladj}
  If $\bCs \to \obCs$, $\bCcs \to \obCcs$, $\bCfs \to \obCfs$, and
  $\bCcfs \to \obCcfs$ are (compatible) fibrant replacements, then we
  have a commutative diagram of \icats{}
  \csquare{\NobCcfs}{\NobCcs}{\NobCfs}{\NobCs}{j^{\txt{f}}}{j^{\txt{c}}}{i^{\txt{c}}}{i^{\txt{f}}}
  where all the functors are fully faithful. Here:
  \begin{enumerate}[(i)]
  \item The inclusion $i^{\txt{f}} \colon \NobCfs \hookrightarrow
    \NobCs$ has a left adjoint $l^{\txt{f}}$ and the unit $X \to
    i^{\txt{f}}l^{\txt{f}}X$ is in $W$ for all $X$.
  \item The inclusion $j^{\txt{f}} \colon \NobCcfs \hookrightarrow
    \NobCcs$ has a left adjoint $l^{\txt{f}}$ and the unit $X \to
    j^{\txt{f}}l^{\txt{f}}X$ is in $W^{\txt{c}}$ for all $X$.
  \item The inclusion $j^{\txt{c}} \colon \NobCcfs \hookrightarrow
    \NobCfs$ has a right adjoint $r^{\txt{c}}$ and the counit
    $r^{\txt{c}}j^{\txt{c}}X \to X$ is in $W^{\txt{f}}$ for all $X$.
  \item The inclusion $i^{\txt{c}} \colon \NobCcs \hookrightarrow
    \NobCs$ has a right adjoint $r^{\txt{c}}$ and the counit
    $r^{\txt{c}}i^{\txt{c}}X \to X$ is in $W$ for all $X$.
  \end{enumerate}
\end{propn}
\begin{proof}
  To prove (i), it suffices to show that for every $X \in \bC$ there
  exists a map $X \to X'$ such that $X'$ is fibrant and
  $\sHom_{\bCs}(X', Z) \to \sHom_{\bCs}(X, Z)$ is a weak equivalence in
  $\sSet$ for every fibrant object $Z$. By
  Proposition~\ref{propn:cohframe}(i) it suffices to factor $X \to *$
  as a trivial cofibration $X \to X'$ followed by a fibration $X' \to
  *$. The proof of (ii) is the same, and (iii) and (iv) follow
  similarly using Proposition~\ref{propn:cohframe}(ii) and the
  factorization of the map $\emptyset \to X$ as a cofibration
  $\emptyset \to X'$ followed by a trivial fibration $X' \to X$.
\end{proof}

\begin{cor}\label{cor:wsmcloceq}
  Inverting the weak equivalences gives equivalences of \icats{}
    \csquare{(\NobCcfs)[W_{\txt{cf}}^{-1}]}{(\NobCcs)[W_{\txt{c}}^{-1}]}{(\NobCfs)[W_{\txt{f}}^{-1}]}{(\NobCs)[W^{-1}].}{\sim}{\sim}{\sim}{\sim}
\end{cor}
\begin{proof}
  Combine Proposition~\ref{propn:repladj} with \cite{enr}*{Lemma
    5.3.14}, which is just an \icatl{} version of
  \cite{DwyerKanLoc}*{Corollary 3.6}.
\end{proof}

\begin{propn}\label{propn:simprelcat}
  Let $\bCs
  \to \obCs$ be a fibrant replacement of $\bCs$. Then the map $\bC \to
  \bCs$ induces an equivalence of \icats{}
  \[ \bC[W^{-1}] \isoto \NobCs[W^{-1}].\]
\end{propn}
The proof is a variant of that of \cite{HA}*{1.3.4.7}, and is based on
ideas that are implicit in the proofs of
\cite{DwyerKanFnCpx}*{Propositions 4.8 and 5.3}.
\begin{proof}
  We may regard the simplicial category $\bCs$ as a simplicial diagram
  $\bC_{\bullet}$ in categories, where $\bC_{n}$ has the same objects
  as $\bC$ and $\Hom_{\bC_{n}}(X, Y) = \Hom_{\bC}(X, T_{n}Y)$ --- as
  mentioned above, this is just the Kleisli category of the monad $T_{n}$.

  Let $W_{n}$ be the collection of morphisms in $\bC_{n}$
  corresponding to morphisms $X \to T_{n}Y$ in $\bC$ that are weak
  equivalences. This determines a simplicial subcategory
  $\mathbf{W}_{\Delta}$ of $\bCs$, and if $\mathbf{W}_{\Delta} \to
  \bar{\mathbf{W}}_{\Delta}$ is a fibrant replacement for this, then
  the \icat{} $\NobCs[W^{-1}]$ is by definition determined
  by the pushout square
  \nolabelcsquare{\mathrm{N}\bar{\mathbf{W}}_{\Delta}}{\|\mathrm{N}\bar{\mathbf{W}}_{\Delta}\|}{\NobCs}{\NobCs[W^{-1}].}

  The \icat{} $\mathrm{N}\obCs$ is the colimit of $\bC_{\bullet}$,
  regarded as a simplicial diagram of \icats{}. This result can be
  found, for example, as \cite{HA}*{Proposition 1.3.4.14}, but is
  certainly far older, and is implicitly used in \cite{DwyerKanLoc} in
  the context of simplicial categories with a fixed set of objects.

  Similarly, $\mathrm{N}\bar{\mathbf{W}}_{\Delta}$ is the colimit of
  $\mathbf{W}_{\bullet}$, and since $\|\blank\|$ preserves colimits (being
  a left adjoint) it follows that the \icat{}
  $\NobCs[W^{-1}]$ is the colimit of the simplicial diagram
  of \icats{} $\bC_{\bullet}[W_{\bullet}^{-1}]$.

  Since $\simp^{\op}$ is a weakly contractible category, to show that
  $\bC[W^{-1}] \to \NobCs[W^{-1}]$ is an equivalence it
  therefore suffices to show that the functor $\bC[W^{-1}] \to
  \bC_{n}[W_{n}^{-1}]$ is an equivalence of \icats{} for all $n$.

  This map arises from the functor $F_{n} \colon \bC \to \bC_{n}$ that is
  the identity on objects and sends a morphism $X \to Y$ to $X \to Y
  \to T_{n}Y$. Since $\bC_{n}$ is the Kleisli category of the monad
  $T_{n}$, this functor has a right adjoint $G_{n} \colon \bC_{n} \to
  \bC$, which sends an object $X$ to $T_{n}X$ and a morphism from $X$ to
  $Y$, which corresponds to a map $f \colon X \to T_{n}Y$ in $\bC$, to
  the map \[ T_{n}X \xto{T_{n}f} T_{n}T_{n}Y \xto{\mu} T_{n}Y.\] The
  composite functor $G_{n}F_{n}$ is $T_{n}$, and the unit $\id \to
  G_{n}F_{n}=T_{n}$ is the unit for $T_{n}$, which is by assumption
  given by maps $X \to T_{n}X$ in $W$ for all $X \in \bC$. On the
  other hand, the counit $F_{n}G_{n}X \to X$ corresponds to the
  identity $T_{n}X \to T_{n}X$, and so lies in $W_{n}$. By
  \cite{DwyerKanLoc}*{Corollary 3.6} this implies that the induced
  maps $\bC[W^{-1}] \rightleftarrows \bC_{n}[W_{n}^{-1}]$ are
  equivalences of \icats{}.
\end{proof}

\begin{proof}[Proof of Proposition~\ref{propn:cohfrloc}]
  Combining Proposition~\ref{propn:simprelcat},
  Proposition~\ref{cor:wsmcloceq}, and Lemma~\ref{lem:cfdontloc}, we
  get the required zig-zag of equivalences.
\end{proof}

\end{document}